\documentclass[11pt]{article}

\usepackage[numbers]{natbib}
\usepackage{amsmath}
\usepackage{amssymb}
\usepackage{amsthm}
\usepackage{xcolor}
\usepackage{graphicx}
\usepackage{hyperref}
\usepackage{algorithm}
\usepackage{algpseudocode}
\usepackage[section]{placeins}
\AddToHook{cmd/subsection/before}{\FloatBarrier}
\usepackage[margin=1in]{geometry}

\hypersetup{colorlinks=true,citecolor=blue,linkcolor=blue,hypertexnames=false}
\allowdisplaybreaks

\newcommand{\paperTitle}{Faster Linear Programming with \texorpdfstring{$\sqrt{\mathrm{rank}}$}{sqrt(rank)} Linear System Solves}
\newcommand{\paperAuthor}
{Zhao Song\thanks{\texttt{magic.linuxkde@gmail.com}. The author would lilke to thank Yin Tat Lee, Victor Reis, Omri Weinstein and Lichen Zhang for very helpful discussions. The first version of this draft is finished by September 2, 2026. This is version 2 of this draft. We provide a detailed comparison between the two versions in the Acknowledgments section of this paper.}}

\DeclareMathOperator{\nnz}{nnz}
\DeclareMathOperator{\Var}{Var}
\DeclareMathOperator{\tr}{tr}
\newcommand{\E}{\mathbb E}
\newcommand{\R}{\mathbb R}
\renewcommand{\d}{\mathop{}\mathrm{d}}

\theoremstyle{plain}
\newtheorem{theorem}{Theorem}[section]
\newtheorem{lemma}[theorem]{Lemma}
\newtheorem{definition}[theorem]{Definition}
\newtheorem{corollary}[theorem]{Corollary}
\theoremstyle{definition}
\newtheorem{remark}[theorem]{Remark}

\begin{document}

\date{September 22, 2026}
\title{\paperTitle}
\author{\paperAuthor}
\maketitle

\begin{abstract}
Lee and Sidford~\cite{ls19} showed that the linear program $\min\{c^\top x\ :\ A^\top x=b,\ l\leq x\leq u\}$ over $x\in\R^m$, where $A\in\R^{m\times n}$, can be solved to accuracy $\epsilon$ using $O(\sqrt n\log^{13}m\cdot\log(mU/\epsilon))$ solves of linear systems in $A^\top\mathbf DA$ for positive diagonal matrices $\mathbf D$, where $U$ bounds the magnitudes of the input and of the initial point. Theirs is the first such bound governed by $\mathrm{rank}(A)=n$ rather than by the number of constraints $m$. Song \cite{son19} mentioned that reducing the $\log^{13}m$ factor is an interesting future direction.

Given an interior point and a finite certified magnitude bound $U$, we give an algorithm that  uses $O(\sqrt n\log^{4}m\cdot\log(mU/\epsilon))$ solves of such linear systems, improving the Lee--Sidford bound by a factor of $\log^{9}m$.

We conjecture that $O(\sqrt n\log(mU/\epsilon))$ such linear-system solves suffice.
\end{abstract}
\newpage

\section{Introduction}
\label{sec:intro}

Linear programming has driven the development of algorithmic optimization for
more than seven decades.  Dantzig's simplex method~\cite{dantzig51} moves from
vertex to vertex and remains a foundational practical algorithm.  Its
worst-case running time, however, is exponential: the Klee--Minty family forces
standard pivot rules to visit exponentially many vertices~\cite{kleeminty72}.
Khachiyan's ellipsoid method gave the first polynomial-time algorithm for
linear programming~\cite{khachiyan79}.  For square instances of dimension $d$
and input bit length $L$, its classical analysis uses $O(d^4L)$ arithmetic
operations.  Karmarkar's projective method then initiated the modern theory of
polynomial-time interior-point algorithms~\cite{karmarkar84}, improving this
bound to $O(d^{3.5}L)$ arithmetic operations on $O(L)$-bit numbers.

The interior-point framework was subsequently simplified and generalized by
Newton and path-following methods~\cite{renegar88,nesterovnemirovskii94}.
Renegar obtained $\widetilde O(\sqrt r\log(1/\epsilon))$ Newton iterations for
an LP with $r$ inequalities, and Vaidya's volumetric-center method sharpened
the arithmetic complexity to
$O((rd^2+r^{3/2}d)L)$ for $r$ inequalities and $d$
variables~\cite{vaidya90}.  These developments shifted attention from polynomiality
alone to the cost of the linear algebra performed along a central path.

Modern work accelerates that linear algebra.  Lee and Sidford obtained
$\widetilde O(\sqrt{\operatorname{rank}(A)}\log(1/\epsilon))$ path-following
iterations, each using only $\widetilde O(1)$ linear-system solves
~\cite{ls14,ls19}.  For dense standard-form LPs with $d$ variables, Cohen,
Lee, and Song achieved
$\widetilde{O}(d^\omega+\allowbreak d^{2.5-\alpha/2}
+\allowbreak d^{2+1/6})$
time, which is $\widetilde{O}(d^\omega )$ for the current
matrix-multiplication parameters~\cite{cls19}. Here $\omega$ is the exponent of matrix multiplication; currently, $\omega<2.371177$~\cite{alphaevolve26}.  Brand subsequently
derandomized the Cohen--Lee--Song algorithm while preserving its asymptotic
running time~\cite{brand20}.  For tall dense linear programs with $n$
constraints and $d$ variables, Brand, Lee, Sidford, and Song gave a
randomized $\widetilde O(nd+d^3)$-time algorithm~\cite{blss20}. Later, that running time time has been improved to $\widetilde O(nd+d^{2.5})$~\cite{bllsssw21}.

\subsection{Our results}

We study the linear program
\begin{equation}
\mathrm{OPT}:=
\min_{\substack{x\in\R^m\ :\ A^\top x=b\\ l_i\leq x_i\leq u_i\ \text{for all }i\in[m]}}
c^\top x ,
\label{eq:intro_lp}
\end{equation}
where $A\in\R^{m\times n}$, $b\in\R^n$, $c\in\R^m$ and
$l,u\in(\R\cup\{-\infty,+\infty\})^m$.  Throughout we assume, as in \cite{ls19},
that $A$ is nondegenerate, meaning that it has full column rank and no zero
row; that $l_i<u_i$ and $(l_i,u_i)\neq\R$ for every $i\in[m]$; and that the
feasible region
$\Omega:=\{x\in\R^m:A^\top x=b,\ l_i\leq x_i\leq u_i\}$
has nonempty interior
$\Omega^\circ:=\{x\in\R^m:A^\top x=b,\ l_i<x_i<u_i\}$.
For the complexity statement, assume that $\Omega$ is bounded and that the
input includes a finite certified magnitude bound $U$.  Let $\mathcal S_0$
be the set of distances from $x_0$ to all finite coordinate endpoints,
\[
 \mathcal S_0:=\{x_{0,i}-l_i:l_i> -\infty\}
 \cup\{u_i-x_{0,i}:u_i<+\infty\}.
\]
The supplied bound is required to satisfy
\[
 U\geq4\max\{1,\|c\|_\infty,\operatorname{diam}_\infty(\Omega),
              \max\mathcal S_0,\max\{s^{-1}:s\in\mathcal S_0\}\}.
\]
The theorem does not require the algorithm to compute this global bound.
This normalization implies the standard interval-barrier bounds
$\|\phi'(x_0)\|_\infty\leq U$ and
$\sqrt{\phi_i''(x_i)}\geq1/U$ for every $x\in\Omega^\circ$.
As usual, logarithmic factors in asymptotic complexity bounds in this section
are understood to be at least $1$.

Lee and Sidford \cite{ls19} solve Eq.~\eqref{eq:intro_lp} to accuracy $\epsilon$
using $\widetilde O(\sqrt n\log(1/\epsilon))$ iterations, each of which solves a
constant number of linear systems in $A^\top\mathbf DA$ for a positive diagonal
matrix $\mathbf D$.  Theirs is the first method whose iteration count is
governed by $\mathrm{rank}(A)=n$ rather than by the number of constraints $m$,
and this distinction is substantive: no known transformation
puts Eq.~\eqref{eq:intro_lp} in standard form without increasing the rank. Written
out, the bound they obtain carries thirteen logarithmic factors,
\[
O(\sqrt n\log^{13}m\cdot\log(\frac{mU}{\epsilon})\cdot\mathcal T_w)\ \text{work}.
\]
The corresponding parallel depth bound has the same thirteen logarithmic
factors, with $\mathcal T_w$ replaced by $\mathcal T_d$, the depth of one
linear-system solve. Song \cite{son19} mentioned that reducing the $\log^{13}m$ factor is an interesting future direction.

We measure work and depth in the exact real-arithmetic model with
bounded-fan-in scalar operations.  Any preprocessing of $A$ is independent
of the input diagonal matrix and right-hand side of a solve.
We state our result as follows.

\begin{theorem}[Linear programming, informal version of Theorem~\ref{thm:log_four_formal}]
\label{thm:main}
Given $\epsilon\in(0,1)$, an interior point of Eq.~\eqref{eq:intro_lp}, and a finite certified magnitude bound $U$ satisfying the preceding display, there is an algorithm that outputs a feasible $x$ with $c^\top x\leq\mathrm{OPT}+\epsilon$ with constant probability using
\[
O(\sqrt n\log^{4}m\cdot\log(\tfrac{mU}{\epsilon})\cdot\mathcal T_w)\ \text{work}
\quad\text{and}\quad
O(\sqrt n\log^{3.5}m\cdot\log(\tfrac{mU}{\epsilon})\cdot\mathcal T_d)\ \text{parallel depth},
\]
where $U$ is defined above, and $\mathcal T_w\geq\operatorname{nnz}(A)+m$
and $\mathcal T_d\geq1$ are upper bounds on the work and depth of computing
$(A^\top\mathbf DA)^{-1}\mathbf q$ for an input positive diagonal matrix
$\mathbf D$ and vector $\mathbf q$.  The work normalization charges sparse
matrix-vector products and vector operations as well as the solve.
\end{theorem}

Thus our result improves the Lee--Sidford bound~\cite{ls19} by a factor of
$\log^{9}m$. We conjecture that $O(\sqrt n\log(mU/\epsilon))$ such linear-system solves suffice.

{\bf Organization of this paper.}
Section~\ref{sec:tech} gives an overview of the techniques and explains how the logarithmic overhead is reduced from $\log^{13}m$ to $\log^4m$.
Section~\ref{sec:preliminaries} introduces the notation, interval barriers, centrality measures, and path-following quantities used throughout the paper.
Section~\ref{sec:geometric_foundations} establishes the regularized Lewis-weight geometry and the projector and Jacobian defect bounds, and constructs the mixed-ball optimization oracle.
Section~\ref{sec:residual_defect_geometry} proves stochastic residual tracking and shows how the retained Jacobian defect controls prediction error and the admissible weight update.
Section~\ref{sec:initialization_path_certification} establishes warm starts, fixed-point initialization, and defect certificates, and gives the path-composition and terminal-output guarantees.
Section~\ref{sec:certified_prediction_assembly} constructs the computable predictor and assembles the centering and path-following routines, together with deterministic epoch normalization.
Section~\ref{sec:score_maintenance_checkpoints} develops hybrid score maintenance and certified checkpoints, including state initialization and the accounting for rejected attempts.
Section~\ref{sec:gram_transport_lazy_repairs} constructs Gram-pullback score transport and lazy auxiliary-weight repairs that charge score drift to the retained defect.
Finally, Section~\ref{sec:log_four} uses an amplified predictor and a squared potential of the projected Newton residual to permit larger path steps and prove the four-logarithm bound in Theorem~\ref{thm:main}.
\section{Technique Overview}
\label{sec:tech}

This section explains the ingredients of the four-logarithm bound.
Section~\ref{subsec:why_thirteen} reviews the logarithmic overhead in the
Lee--Sidford method.  Section~\ref{subsec:how_improve} describes the
regularized weight geometry and defect-adaptive prediction.
Section~\ref{subsec:gram_overview} explains how score transport and
lazy auxiliary repairs charge score maintenance to the retained defect.
Finally, Section~\ref{subsec:four_log_overview} uses a squared potential
of the projected residual to bound both the path length and the total
score drift, including the work of rejected attempts.

\subsection{The logarithmic overhead in the Lee--Sidford method}
\label{subsec:why_thirteen}

The Lee--Sidford method~\cite{ls19} combines path following with repeated
Lewis-weight computation.  The outer loop takes centering steps whose
size is restricted by the tracking tolerance and centrality radius.
The inner loop approximately solves a Lewis fixed point, with each pass
estimating leverage scores by linear-system solves.  Together these
costs give their $O(\sqrt n\log^{13}m\cdot\kappa)$ system bound, where
$\kappa$ accounts for the logarithmic path length and the target accuracy.
Our analysis keeps this precision dependence separate from the
logarithmic overhead in $m$.

\subsection{Weight geometry and defect-adaptive prediction}
\label{subsec:how_improve}

\paragraph{Regularized Lewis weights and sharper centrality geometry.}
The regularization floor is placed inside the Lewis fixed point.  Taking
$q:=\log(4m/n)$ and $p=1-\Theta(1/q)$ makes the sensitivity constant while
the consistency is $c_k=\Theta(q)$.  The orthogonal-projector and
Lewis-Jacobian Pythagorean identities permit the mixed norm
$N(z):=\|z\|_\infty+C\|z\|_w$ with $C=\Theta(\sqrt q)$, together with the
tracking tolerance $K=\Theta(q^{-1/2})$ (Lemmas~\ref{lem:pythagorean_projector}
and~\ref{lem:sharper_parameters}).

\paragraph{Retaining and estimating the Jacobian defect.}
For a logarithmic barrier movement $\psi$, let $B$ be the Lewis-weight
Jacobian and define $d^2:=p^2\|\psi\|_g^2-\|B\psi\|_g^2$.
This is the exact ideal-metric defect; maintained and surrogate metrics are
used only through the comparison lemmas in Section~\ref{sec:geometric_foundations}.
The retained-movement inequality gains a margin
$\Theta(C^2d^2/\delta)$.  The off-diagonal part of $B\psi$ factors through
$Q=W^{-1/2}(P\circ P)W^{-1/2}$, and the two-projection estimator has variance
$O(y^\top Qy)=O(d^2)$.  Thus the same defect that creates centrality slack
pays for the stochastic predictor (Lemmas~\ref{lem:retained_jacobian_defect},
\ref{lem:large_radius_retained_movement}, and
\ref{lem:linear_defect_projection}).  The defect can also be certified with
the required one-sided guarantee (Lemma~\ref{lem:defect_certificate}).

\paragraph{A larger centrality radius.}
The centrality geometry uses $E=\Theta(q^{-1/2}\ell^{-1})$ and
$R=\Theta(CE)=\Theta(\ell^{-1})$, where $\ell:=\log(4m)$.
The predictor radius is linear in the observed defect and in
$\delta/c_k$.  Young's inequality charges the linear defect against the
quadratic retained-movement margin, so the centrality radius is larger than
the observation scale by the factor $C$.  The adaptive residual game then
tracks the remaining stochastic error without changing the deterministic
movement cap.

\subsection{Score transport and lazy auxiliary repairs}
\label{subsec:gram_overview}

\paragraph{Diagonal self-transport.}
Consider two consecutive score queries $D_-$ and $D_+$.  After removing the
irrelevant common scalar, write their row-log displacement as $d$.  For a
positive companion $\widetilde w$, define
$\lambda_w:=(\sum_{i=1}^m\widetilde w_i d_i)/(\sum_{i=1}^m\widetilde w_i)$ and
$\rho_w:=\|d-\lambda_w\mathbf1\|_{\widetilde w}$.
If only row $i$ is multiplied by $\gamma$, its leverage score $s$ changes
exactly according to
$\mathcal T_\gamma(s):=\gamma^2s/(1+(\gamma^2-1)s)$
(Lemma~\ref{lem:candidate_one_row_transport}).
The algorithm applies this scalar map to its incoming approximate score and
uses one fresh shared sketch only for the response of all the other rows.
The exact one-row identity makes the deterministic error multiplier at most
$e^{O(\rho_w)}$, while the centered sketch innovation has conditional
variance proxy $O(\rho_w^2/k)$.  Thus a large coordinatewise row displacement
incurs no additional asymptotic cost when its weighted centered energy is small
(Lemma~\ref{lem:candidate_nonlocal_endpoint}).

\paragraph{Hybrid effective drift and restarts.}
For transitions better controlled coordinatewise, define
$\rho_\infty:=\min_\lambda\|d-\lambda\mathbf1\|_\infty$.  The provider uses
the eligible branch with smaller effective drift, namely
$\eta_t:=\min\{C_\infty\rho_{\infty,t},C_w\rho_{w,t}\}$.  It uses
$k_t=\Theta(\eta_tE^{-2}L_{\rm conf})$ rows for that transition and restarts
from a fresh absolute Gaussian sketch before the cumulative effective drift
exceeds a small constant.  After rescaling by the accumulated drift, the
relative-error recurrence is contractive, so maximal Bernstein gives
simultaneous raw relative accuracy (Algorithm~\ref{alg:candidate_hybrid_score}
and Lemma~\ref{lem:candidate_hybrid_maintenance}).

\paragraph{Gram-pullback transport and lazy repairs.}
Section~\ref{sec:gram_transport_lazy_repairs} charges a principal score transition to
$O(\overline d)$, where $\overline d$ is the computable defect certificate.
Pulling the new projector back to the old column space makes the Gram
perturbation control the sketch variance.  A coarse-center projection
of the weighted predictor makes this charge valid for every realization
of its independent weighted chain.  The auxiliary fixed point is repaired
only when its accumulated coordinate error exceeds $E/4$; the repair
cost is charged to the preceding predictor errors.  Thus the total
effective drift is bounded by $O(E+\sum_j\overline d_j)$
(Lemma~\ref{lem:gram_lazy}).  Every query endpoint and drift ticket
is fixed before the corresponding fresh score sketch is drawn.

\paragraph{Long paths and certified checkpoints.}
The score-maintenance restarts above control error inside one score-maintenance
block; they do not by themselves make a confidence bound independent of the
total path length.  For that purpose the complete path is also divided into
polynomial-length epochs.  At an epoch boundary the algorithm uses exact
score diagonals to recenter the path point and greedily reduce the
Lewis-weight tracking potential, certifies the resulting state outward, and
then starts new score and residual-game epochs.  A failed certificate causes
only the mathematical state of that epoch to be retried---resource counters
are never rolled back.  Because the normalization has fixed polynomial cost
and epochs are chosen longer than that cost, checkpoints add only a constant
factor to the ledger while allowing all confidence logarithms to remain
$O(\log m)$ for arbitrarily long paths (Lemmas~\ref{lem:epoch_normalization}
and~\ref{lem:checkpointed_path}, and Algorithm~\ref{alg:checkpointed_path}).
The two path phases and their final output guarantee are supplied by the outer
reduction (Algorithm~\ref{alg:lp_solve_fresh} and
Lemma~\ref{lem:outer_reduction}).

The resulting score cost is $O(M+E^{-2}L_{\rm conf}(1+\Omega))$, where
$\Omega:=\sum_{t=1}^M\eta_t$.  Fresh randomness enters at every block base and every
transition, so no algebraically recycled anchor or estimate is used.

\subsection{Four logarithms from the projected residual}
\label{subsec:four_log_overview}

Section~\ref{sec:log_four} combines the Gram-pullback score provider and lazy
repairs with an amplified weighted predictor and the actual projected
Newton residual $r$.  Write $k:=c_k=8q$ and define $X:=\|r\|_\infty$, $Y:=\|r\|_w$, and
$\mathcal V:=X^2+\Lambda kY^2$ for an absolute $\Lambda$.
For the specified interval barriers, $\phi'''=2\phi'\phi''$.
Together with the conservation identity $g^\top B=0$, this makes the
path forcing couple to the dissipated directions of the residual map.
After controlling the finite Newton step, predictor error, changing
projection, and maintained metric, the result is
\[
 \mathcal V^+\leq(1-1/(2k))\mathcal V
 -c_*k\mathfrak D^2+A_*kn\xi^2,
 \qquad \xi=t^+/t-1.
\]
The four nonnegative losses making up $\mathfrak D^2$ are defined in
Eq.~\eqref{eq:four_log_losses}.  Only $\mathcal V$ is computed by the
algorithm.  This inequality permits $|\xi|=\Theta(R/(k\sqrt n))$,
so the call count becomes $N=O(\sqrt n q\ell\kappa)$.
Telescoping both negative terms also pays for the additive floor in the
defect certificate.  For every good attempted prefix,
\[
 \sum_j\overline d_j^2=O(\frac{R^2}{k}(1+b/k)),
 \qquad \Omega=O(E(1+\sqrt b+b/\sqrt k)).
\]
The preannounced instruction-quota caps price rejected attempts by the same
formula.  Amplification costs $O(k\ell)$ systems per call, which is
absorbed because $q\leq\ell$.  The scalar optimizers retain their
$O(m\ell^2)$ per-call bound.  The resulting leading system count is
$O(\sqrt n q\ell^3\kappa)$, giving the four-logarithm work bound in
Theorem~\ref{thm:log_four_formal}.

Cold initialization costs $O(\min\{q\ell^3,n\log\ell\})$ systems.
The first terminal loop restarts its counter after normalization, while
the final fixed-weight Newton tail runs only once, outside the retry loop.
These costs are included in the same four-logarithm bound
(Lemmas~\ref{lem:four_log_terminal} and~\ref{lem:four_log_ledger}).
 
\section{Preliminaries}
\label{sec:preliminaries}

Section~\ref{subsec:basic_notation} fixes the basic vector and matrix
notation used throughout the paper.  Section~\ref{subsec:centrality_prelim}
then records the barrier, centrality, and path-following quantities needed in
the algorithm and its analysis.

\subsection{Basic notation}
\label{subsec:basic_notation}

For a positive integer $d$, write $[d]:=\{1,\ldots,d\}$.  All vectors are
column vectors; $e_i$ denotes the $i$-th standard basis vector, and $\mathbf1$
denotes the all-ones vector of the dimension determined by context.  For
$x\in\R^d$, write $\|x\|_1:=\sum_{i=1}^d|x_i|$,
$\|x\|_2:=(\sum_{i=1}^d x_i^2)^{1/2}$, and
$\|x\|_\infty:=\max_{i\in[d]}|x_i|$.  The Euclidean inner product is
$\langle x,y\rangle:=x^\top y$.

For a matrix $M$, the symbols $M^\top$, $M^{-1}$, and $M^\dagger$ denote its
transpose, its inverse when it is nonsingular, and its Moore--Penrose
pseudoinverse, respectively.  A superscript $+$ is instead reserved for an
updated quantity.  The identity matrix $I$ always has the dimension determined
by context.  For symmetric matrices, $M\succeq N$ means that $M-N$ is positive
semidefinite.  If $M\succeq0$, then $M^{1/2}$ denotes its positive
semidefinite square root; negative powers are used only when the relevant
matrix is positive definite.  Given vector norms $\|\cdot\|_a$ and
$\|\cdot\|_b$, the induced operator norm is
$\|M\|_{a\to b}:=\sup_{x\neq0}\|Mx\|_b/\|x\|_a$.

For a vector $x$, $\operatorname{Diag}(x)$ is the diagonal matrix with diagonal
$x$; for a square matrix $M$, $\operatorname{diag}(M)$ is its diagonal vector.
We use $\odot$ and $\circ$ for the coordinatewise and Hadamard products,
respectively.  Vector inequalities are coordinatewise, as are vector products,
quotients, powers, square roots, exponentials, and logarithms.  For a scalar
$z$, define $\cosh z:=(e^z+e^{-z})/2$ and
$\sinh z:=(e^z-e^{-z})/2$; both functions act coordinatewise on vectors.

\subsection{Centrality and Lee--Sidford path facts}
\label{subsec:centrality_prelim}

For each coordinate interval $(l_i,u_i)$, let $\phi_i$ be the
one-self-concordant interval barrier used by Lee--Sidford~\cite{ls19}, define
$\Phi''(x):=\operatorname{Diag}(\phi_i''(x_i))$, and write
$A_x:=\Phi''(x)^{-1/2}A$.  For a positive vector $w$, let
$W:=\operatorname{Diag}(w)$,
\[
 \|y\|_w:=(y^\top Wy)^{1/2},
 \qquad
 N_w(y):=\|y\|_\infty+C\|y\|_w.
\]

\begin{definition}[Centrality and permitted systems]
\label{def:centrality_permitted}
For an active cost vector $c_{\rm path}$, define
\[
 f_t^{c_{\rm path}}(x,w)
 :=t c_{\rm path}^\top x+\sum_{i=1}^mw_i\phi_i(x_i)
\]
and
\[
 \delta_t^{c_{\rm path}}(x,w)
 :=\min_{\eta\in\R^n}
 N_w(
 \frac{t c_{\rm path}+w\odot\phi'(x)-A\eta}
      {w\odot\sqrt{\phi''(x)}}).
\]
When the cost is clear, the superscript is omitted.  Define
\[
 P_{x,w}:=I-W^{-1}A_x(A_x^\top W^{-1}A_x)^{-1}A_x^\top.
\]
For
\[
 r_t^{c_{\rm path}}(x,w)
 :=\frac{t c_{\rm path}+w\odot\phi'(x)}
         {w\odot\sqrt{\phi''(x)}},
 \qquad
 \overline\delta_t^{c_{\rm path}}(x,w)
 :=N_w(P_{x,w}r_t^{c_{\rm path}}(x,w)),
\]
the first centrality-equivalence inequality is unconditional.  Whenever
$4g(x)/5\leq w\leq5g(x)/4$, Lemma~\ref{lem:pythagorean_projector}
supplies the uniform constant $c_\gamma$ used below and gives
\begin{equation}\label{eq:computable_centrality}
 \delta_t^{c_{\rm path}}(x,w)
 \leq\overline\delta_t^{c_{\rm path}}(x,w)
 \leq c_\gamma\delta_t^{c_{\rm path}}(x,w).
\end{equation}
All later uses of the upper bound are under this tracking band.
Thus $\overline\delta_t$ is a computable centrality proxy, obtained with the
same permitted solve as the projected Newton residual, and
$\overline\delta_t=0$ if and only if $\delta_t=0$.  We also write
$s_i(x):=\phi_i''(x_i)^{-1/2}$ for the positive barrier row scaling, so that
$A_x=\operatorname{Diag}(s(x))A$.
The projected Newton step is
\[
 h_t^{c_{\rm path}}(x,w)
 :=-\Phi''(x)^{-1/2}P_{x,w}
 \frac{t c_{\rm path}+w\odot\phi'(x)}
      {w\odot\sqrt{\phi''(x)}}.
\]
A \emph{permitted system} is a solve in $A^\top D A$ for a positive
diagonal matrix $D$ and one right-hand side.
\end{definition}

We use the following Lee--Sidford facts with the preceding notation.  Stating
them here fixes the hypotheses under which they are invoked later.

\begin{lemma}[Lee--Sidford path facts]
\label{lem:ls_path_facts}
Suppose $\delta:=\delta_t(x,w)\leq1/10$ and
$4g(x)/5\leq w\leq5g(x)/4$.  If $x^+=x+h_t(x,w)$, then $x^+$ remains in the
barrier domain,
\[
 \delta_t(x^+,w)\leq4\delta^2,
 \qquad
 N_w(\sqrt{\Phi''(x)}(x^+-x))\leq c_\gamma\delta.
\]
If $w'>0$ and $z:=N_w(\log w'-\log w)\leq1/10$, then
\[
 \delta_t(x,w')\leq(1+4z)(\delta_t(x,w)+z).
\]
For $|\xi|<1$ and $(1+\xi)t>0$,
\[
 \delta_{(1+\xi)t}(x,w)
 \leq(1+|\xi|)\delta_t(x,w)
   +|\xi|(1+C\sqrt{\|w\|_1}).
\]
Finally, define
\[
 \mathrm{OPT}(c_{\rm path})
 :=\min\{c_{\rm path}^\top x:x\in\Omega\}.
\]
If $x_t(w)$ minimizes $f_t^{c_{\rm path}}(\cdot,w)$ on the affine
feasible set, then
\[
 c_{\rm path}^\top x_t(w)-\mathrm{OPT}(c_{\rm path})
 \leq\frac{\|w\|_1}{t}.
\]
\end{lemma}
\begin{proof}
The Newton contraction is Lemma 15 of~\cite{ls19}, the
weight-change estimate is Lemma 17, and the $\xi\geq0$ case of the
$t$-change estimate is Lemma 14.  For $\xi<0$, evaluate the new residual
at the old minimizing multiplier scaled by $1+\xi$; the old residual is
multiplied by $1+\xi$ and the added cost term has mixed norm at most
$|\xi|(1+C\sqrt{\|w\|_1})$.  This proves the displayed signed
extension directly rather than attributing the negative case to Lemma 14.
The proofs of these three Lee--Sidford estimates use only that $N_w$ is a
norm dominating $\|\cdot\|_\infty$ and that the
centrality-equivalence constant is at most two.  They therefore apply
verbatim to the present coefficient $C$.  The normalized-step estimate is
their centrality-equivalence lemma.  The last display is their weighted
duality-gap lemma and follows directly by pairing the first-order optimality
equation at $x_t(w)$ with an optimal feasible point.
\end{proof}

\section{Geometric foundations and the mixed-ball oracle}
\label{sec:geometric_foundations}

Section~\ref{subsec:score_perturbations} establishes rowwise projector
stability and the regularized Lewis-weight fixed point used by the score
provider.  Section~\ref{subsec:projector_jacobian_defects} proves the
projector and Jacobian defect inequalities that sharpen the mixed-norm
geometry.  Section~\ref{subsec:robust_centrality_geometry} converts those
inequalities into robust centrality and fixes the structural parameters.
Finally, Section~\ref{subsec:mixed_ball_oracle} gives the finite mixed-ball
optimization oracle used by the residual-chasing step.

\subsection{Projection-score perturbations and regularized Lewis fixed points}
\label{sec:dynamic_oracle}
\label{subsec:score_perturbations}

\paragraph{Overview.} The argument uses a rowwise stability estimate for the
orthogonal projector under a diagonal rescaling, the global fixed-point
contraction, and a predicted warm start.  Along the path, the hybrid score
provider transports the diagonal response exactly, sketches only the
nonlocal endpoint discrepancy, and restarts after constant cumulative
effective drift.  Every transition sketch is drawn only after both query
endpoints are fixed.

This section exploits information unavailable within a single call to the
approximate Lewis-weight routine \texttt{computeApxWeight}: consecutive
centering steps request leverage scores for diagonal rescalings of the
\emph{same} matrix $A$ that differ only by $e^{\pm O(R)}$.  Recomputing an
$E$-accurate estimate from scratch at every step does not use this continuity.

We first isolate the relative-drift estimate used by both transition branches
of the hybrid provider.  Its Gaussian increment primitive is recorded later
in Lemma~\ref{lem:adaptive_increment}; no algebraically recycled anchor is used.

Throughout this section $A\in\R^{m\times n}$ is nondegenerate, $\mathbf D$
denotes a positive diagonal matrix, and
\[
P(\mathbf D):=\mathbf D A(A^\top \mathbf D^2 A)^{-1}A^\top \mathbf D
\]
is the orthogonal projection onto the range of $\mathbf D A$, so that
$\sigma_i(\mathbf D A)=\|P(\mathbf D)e_i\|_2^2$.

The first estimate is fundamental to the scheme: when two rescalings
differ by a small \emph{diagonal} factor, the corresponding columns of the two
projections differ by a small \emph{relative} amount.  An absolute bound would
not suffice, because $\sigma_i$ can be as small as $n/m$.

\begin{lemma}[Relative drift of projection columns]
\label{lem:relative_drift}
Let $\mathbf D=\Gamma\mathbf D'$ with $\Gamma$ positive diagonal and $\|\log\Gamma\|_\infty\leq\rho\leq 1/8$. Then for every $i\in[m]$,
\[
\|P(\mathbf D)e_i-P(\mathbf D')e_i\|_2\leq 10\rho\sqrt{\sigma_i(\mathbf D'A)}.
\]
Moreover,
\[
e^{-4\rho}\sigma_i(\mathbf D'A)
\leq \sigma_i(\mathbf DA)
\leq e^{4\rho}\sigma_i(\mathbf D'A).
\]
\end{lemma}
\begin{proof}
Let $B\in\R^{m\times n}$ have orthonormal columns spanning the range of $\mathbf D'A$, so that $P(\mathbf D')=BB^\top$ and $\sigma_i(\mathbf D'A)=\|B^\top e_i\|_2^2$. The range of $\mathbf DA$ is the range of $\Gamma B$, and therefore
\[
P(\mathbf D)=\Gamma B(B^\top\Gamma^2B)^{-1}B^\top\Gamma .
\]
Write $c:=B^\top e_i$, so that $\|c\|_2^2=\sigma_i(\mathbf D'A)$. Because $\Gamma$ is diagonal we have $\Gamma e_i=\Gamma_{ii}e_i$, hence $B^\top\Gamma e_i=\Gamma_{ii}c$, and consequently
\[
P(\mathbf D)e_i-P(\mathbf D')e_i
=(\Gamma_{ii}\,\Gamma B(B^\top\Gamma^2B)^{-1}-B)\,c .
\]
It suffices to bound the spectral norm of the matrix in parentheses. Define $F:=(B^\top\Gamma^2B)^{-1}-I$. From $\rho\leq 1/8$ we obtain $\|\Gamma-I\|\leq e^{\rho}-1\leq 1.14\rho$ and $|\Gamma_{ii}-1|\leq 1.14\rho$, and also $\|B^\top\Gamma^2B-I\|\leq e^{2\rho}-1\leq 2.6\rho\leq 0.33$, so that $\|F\|\leq 2.6\rho/(1-0.33)\leq 4\rho$. Using $\|B\|=1$ and $\|\Gamma B\|\leq 1+1.14\rho$, the triangle inequality gives
\[
\begin{aligned}
\|\Gamma_{ii}\Gamma B(I+F)-B\|
&\leq |\Gamma_{ii}-1|\cdot\|\Gamma B(I+F)\|+\|\Gamma B\|\cdot\|F\|+\|\Gamma B-B\|\\
&\leq 1.14\rho(1+1.14\rho)(1+4\rho)+(1+1.14\rho)\cdot 4\rho+1.14\rho
\leq 10\rho ,
\end{aligned}
\]
where the last step uses $\rho\leq 1/8$ and that each coefficient increases in $\rho$. For the leverage-score comparison, note that
\[
e^{-2\rho}(\mathbf D'A)^\top\mathbf D'A
\preceq(\mathbf D'A)^\top\Gamma^2\mathbf D'A
\preceq e^{2\rho}(\mathbf D'A)^\top\mathbf D'A ,
\]
so the inverse quadratic form in the $i$-th leverage score changes by a factor $e^{\pm2\rho}$. The two copies of $\Gamma_{ii}$ contribute another factor $e^{\pm2\rho}$, giving the stated $e^{\pm4\rho}$ comparison.
\end{proof}

\begin{remark}
Diagonality of $\Gamma$ is essential.  For a general $\Gamma$ with
$\|\Gamma-I\|\leq\rho$, one obtains only
$\|P(\mathbf D)e_i-P(\mathbf D')e_i\|_2=O(\rho)$, which is an absolute rather
than a relative bound and provides no improvement for small leverage scores.
The gain comes from $\Gamma e_i=\Gamma_{ii}e_i$, which allows the factor
$\sqrt{\sigma_i}$ to be extracted from the perturbation.
\end{remark}

Each centering step must issue only $O(1)$ score queries.  The loop of
\texttt{computeApxWeight} issues $T=\Theta(\log m)$ of them because its
convergence is measured in a Euclidean norm and the conversion back to
$\ell_\infty$ costs a factor of $n$.  Replacing its inner loop by the
fixed-point iteration below removes that loss, since the iteration contracts
in $\ell_\infty$ directly.  For $0<p<2$ define
$\vartheta:=1-\frac{2}{p}$ and, for a nondegenerate matrix $B$ with rows
$b_i^\top$ and $w\in\R^m_{>0}$, define
$\Lambda_{B,i}(w):=b_i^\top(B^\top\mathbf W^{\vartheta}B)^{-1}b_i$,
suppressing the subscript when the matrix is clear.

We work throughout with the \emph{regularized} Lewis vector $w^v_p(B)$, the solution of $w=\sigma(\mathbf W^{\vartheta/2}B)+v$ for a fixed floor $v\geq0$. Placing the regularizer inside the fixed point rather than adding it afterward is what bounds the dynamic range of the weights, and through it the sensitivity of the weight function; this is the content of Lemma~\ref{lem:sensitivity}, and it is the reason the exponents below are governed by $\log(4m/n)$ rather than by $\log(4m)$. The identity $\sigma_i(\mathbf W^{\vartheta/2}B)=w_i^{\vartheta}\Lambda_{B,i}(w)$ shows that a relative approximation of $\Lambda_B(w)$ is one leverage-score query, and rearranging the fixed-point equation as $w_i^{2/p}=\Lambda_{B,i}(w)+w_i^{-\vartheta}v_i$ exhibits $w^v_p(B)$ as the unique fixed point of
$$F_{B,i}(w):=(\Lambda_{B,i}(w)+w_i^{-\vartheta}v_i)^{p/2}.$$
Writing $\tau_i(w;B):=\sigma_i(\mathbf W^{\vartheta/2}B)+v_i$ for the \emph{regularized} score, the same map takes the multiplicative form
$$F_{B,i}(w)=w_i^{\,1-p/2}\,\tau_i(w;B)^{p/2},$$
because $-\vartheta p/2=1-p/2$. Two things are then immediate. The map sees the leverage scores only through $\tau$, so only $\tau$ has to be estimated relatively, which is what Section~\ref{sec:dynamic_oracle} exploits; and
$$\log F_{B,i}(w)-\log w_i=\tfrac p2(\log\tau_i(w;B)-\log w_i),$$
so the residual of the fixed-point equation is directly observable from an estimate of $\tau$.
The last part of the lemma is stated for the starting error itself rather than for the movement of the target, because in the final construction the target moves by $\Theta(R)$ while the accuracy to be restored is $E$, smaller by an unbounded factor. Lemma~\ref{lem:large_radius_warm_start} reconciles them.

\paragraph{Regularized Lewis fixed points.}
\label{subsec:regularized_fixed_points}

Before this fixed point can be used as a weight function, we need existence,
uniqueness, and differentiability under row rescaling.  The next lemma records
these facts together with the Jacobian formula used throughout the proof.
\begin{lemma}[Well-posed regularized Lewis weights]
\label{lem:regularized_weight_wellposed}
Let $0<p<1$, let $\vartheta:=1-2/p<0$, let $v>0$ coordinatewise, and let
$B\in\mathbb R^{m\times n}$ have full column rank and no zero row.  Then
$w=\sigma(W^{\vartheta/2}B)+v$ has a unique solution
$w=w_p^v(B)>0$.  This solution depends smoothly on every positive diagonal
row scaling of $B$.  If $B=\operatorname{Diag}(c)A$,
$\Lambda:=\operatorname{Diag}(\sigma)-P\circ P$, and
$a:=-\vartheta=2/p-1$, then
\[
 (W+a\Lambda)\,\d\log w=2\Lambda\,\d\log c.
\]
In particular, the derivative operator is invertible.
\end{lemma}
\begin{proof}
For $z=\log w$, define
\[
 \mathcal J_B(z):=\sum_{i=1}^m e^{z_i}-v^\top z
 -\frac1{\vartheta}\log\det(B^\top
          \operatorname{Diag}(e^{\vartheta z})B).
\]
The matrix in the determinant is positive definite.  Differentiating gives
\[
 \nabla\mathcal J_B(z)=w-v-\sigma(W^{\vartheta/2}B),
 \qquad
 \nabla^2\mathcal J_B(z)=W+a\Lambda.
\]
The projection identity
$\Lambda=\operatorname{Diag}(\sigma)-P\circ P\succeq0$ and
$W\succ0$ show that the Hessian is positive definite.  Moreover
$\mathcal J_B(z)\to+\infty$ as $\|z\|_\infty\to\infty$:
an unbounded positive coordinate is dominated by $\sum_{i=1}^m e^{z_i}$, while
if the positive coordinates stay bounded and a coordinate tends to
$-\infty$, the term $-v^\top z$ and the nonnegative
logarithmic-determinant growth prevent escape.  Thus $\mathcal J_B$ is
coercive and strictly convex, so it has a unique minimizer, exactly the
displayed fixed point.

The Hessian is nonsingular, hence the implicit-function theorem gives smooth
dependence on positive row scalings.  Finally,
$\d\sigma=2\Lambda(\d\log c+\vartheta\,\d((\log w)/2))$.
Differentiating $w-v=\sigma$ and rearranging gives the claimed equation.
\end{proof}

The fixed-point iteration must remain contractive when its leverage estimates
are approximate, and its target must move continuously under row rescaling.
The following lemma gives both properties in the same logarithmic metric.
\begin{lemma}[$\ell_\infty$ contraction and target stability]
\label{lem:linf_contraction}
Let $0<p<2$, let $v\geq0$, and let $\widehat F_i=e^{\pm\delta}F_{A,i}(w)$ for every $i\in[m]$, which holds as soon as the quantity $\Lambda_{A,i}(w)+w_i^{-\vartheta}v_i=w_i^{-\vartheta}(\sigma_i(\mathbf W^{\vartheta/2}A)+v_i)$ is estimated to within a factor $e^{\pm2\delta/p}$. Only that \emph{sum} has to be relatively accurate: an estimate of $\sigma_i$ with additive error $O(\delta)(\sigma_i+v_i)$ suffices, which is strictly weaker than a relative estimate of $\sigma_i$ and is supplied by the hybrid score provider. Then
\[
\|\log\widehat F-\log w^v_p(A)\|_\infty
\leq(1-\frac{p}{2})\|\log w-\log w^v_p(A)\|_\infty+\delta .
\]
Moreover, if $A'=\mathbf DA$ for a positive diagonal matrix $\mathbf D$ satisfying $\|\log\mathbf D\|_\infty\leq r$, then
\[
\|\log w^v_p(A')-\log w^v_p(A)\|_\infty\leq 4r.
\]
Consequently, fix $p_0>0$. There is an integer $t_0=O_{p_0}(1)$ with the following property. If $p\in[p_0,1)$, the starting point satisfies
\[
\|\log w^{(0)}-\log w^v_p(A')\|_\infty\leq E
\]
for some $E>0$, and every iteration $w^{(t+1)}=\widehat F_{A'}(w^{(t)})$ has logarithmic error $\delta\leq pE/8$, then
\[
\|\log w^{(t_0)}-\log w^v_p(A')\|_\infty\leq E/2.
\]
\end{lemma}

\begin{proof}
The comparison below applies to any two positive inputs, so the map
$z\mapsto\log F_A(e^z)$ is a global contraction in $\ell_\infty$ with factor
$1-p/2<1$.  Banach's theorem therefore supplies the unique positive fixed
point $w^v_p(A)$.  Write $w_*:=w^v_p(A)$ and $w=e^{u}w_*$, so $\mathbf W^{\vartheta}=e^{\vartheta u}\mathbf W_*^{\vartheta}$ with $e^{\vartheta u}$ diagonal. Hence
\[
e^{-|\vartheta|\|u\|_\infty}A^\top\mathbf W_*^{\vartheta}A
\preceq A^\top\mathbf W^{\vartheta}A
\preceq e^{|\vartheta|\|u\|_\infty}A^\top\mathbf W_*^{\vartheta}A ,
\]
and inverting and applying the resulting inequality to $a_i$ gives
\[
\Lambda_{A,i}(w)=e^{\pm|\vartheta|\|u\|_\infty}
\Lambda_{A,i}(w_*).
\]
The second summand obeys the same estimate, since $w_i^{-\vartheta}v_i=e^{-\vartheta u_i}(w_*)_i^{-\vartheta}v_i$; as both summands are nonnegative, their sum is within $e^{\pm|\vartheta|\|u\|_\infty}$ of its value $(w_*)_i^{2/p}$ at $w_*$. Raising to the power $p/2$ yields
\[
F_{A,i}(w)
=e^{\pm\frac{p}{2}|\vartheta|\|u\|_\infty}(w_*)_i,
\qquad
\frac{p}{2}|\vartheta|=1-\frac{p}{2},
\]
and the error in $\widehat F$ contributes a further factor $e^{\pm\delta}$, which proves the first claim.

For target stability, note that for every $w>0$
\[
A'^\top\mathbf W^\vartheta A'=A^\top\mathbf D\mathbf W^\vartheta\mathbf DA
\]
lies between $e^{-2r}A^\top\mathbf W^\vartheta A$ and $e^{2r}A^\top\mathbf W^\vartheta A$. The inverse quadratic form changes by $e^{\pm2r}$, and the factor $\mathbf D_{ii}^2$ in the $i$-th row contributes another $e^{\pm2r}$, so $\Lambda_{A',i}(w)=e^{\pm4r}\Lambda_{A,i}(w)$. The second summand of $F$ does not involve $A$ at all, so the sum moves by at most $e^{\pm4r}$ as well, and
\[
\|\log F_{A'}(w)-\log F_A(w)\|_\infty\leq 2pr.
\]
Define $w_*=w^v_p(A)$, $w_*'=w^v_p(A')$ and $d:=\|\log w_*'-\log w_*\|_\infty$. Applying the first claim to the exact map $F_{A'}$ and then the preceding comparison gives $d\leq(1-\frac p2)d+2pr$, so $d\leq4r$. For the final claim define $\varrho:=1-p/2\leq1-p_0/2<1$. Unrolling the first claim, the inexact iteration satisfies
\[
\|\log w^{(t)}-\log w^v_p(A')\|_\infty
\leq \varrho^{\,t}E+\frac{\delta}{1-\varrho}
\leq \varrho^{\,t}E+\frac E4 ,
\]
where the first step follows from unrolling the inexact contraction recurrence,
and the second step follows from $1-\varrho=p/2$ and $\delta\leq pE/8$.
Choosing $t_0$ with $\varrho^{\,t_0}\leq1/4$ proves the result.
\end{proof}

\subsection{Projector and Jacobian defects}
\label{sec:sharper_norm}
\label{subsec:projector_jacobian_defects}

Section~\ref{sec:dynamic_oracle} changes the Lewis-weight computation while
preserving the path geometry.  This section changes the geometry while
preserving score maintenance.  The two ingredients are combined by
the hybrid provider in Section~\ref{subsec:hybrid_maintenance_stopping}.
Throughout, define
$\ell:=\log(4m)$ and $q:=\log\frac{4m}{n}$, and let $L_*:=\ell$ denote the
deterministic confidence scale.  The
checkpointed execution below keeps every stochastic horizon polynomial in
$m$, independently of the total path length $\kappa$.
$\|y\|_{w+\infty}:=\|y\|_\infty+C\|y\|_w$ is the mixed centrality norm
of~\cite{ls19}, and $K$ is the radius of the invariant
$\|\log w-\log g(x)\|_\infty\leq K$ maintained by the chasing game.  The
parameters $C,K,p,R,E$ will be chosen below.  Here $R$ is the centrality
radius tolerated by the centering argument, whereas $E$ is the score accuracy
observed by the chasing game.  Lemma~\ref{lem:sharper_parameters} fixes
$C$, $K$, and $p$; Section~\ref{subsec:large_radius_refinement} makes the final
choices of $R$ and $E$.  Until then, $R$ is only assumed to be at most
$r_0/c_k$ for a sufficiently small absolute constant $r_0\leq1/160$.

Three parameter choices in the analysis of~\cite{ls19} arise from bounds that
appear linear but admit sharper estimates.  The coefficient
$C=24\sqrt{c_s}c_k$ absorbs the $\Theta(1/C)$ error of two triangle
inequalities, one for the weighted projector and one for the Lewis-weight
Jacobian.  The tolerance $K=1/(16c_k)$ absorbs the $\Theta(K)$ error of
converting the ideal norm into the maintained one.  Finally, the consistency
$c_k=\Theta(1/(1-p))$ grows to $\Theta(\log m)$ because the sensitivity of
their weight function is $m^{\Theta(1-p)}$.  The first two losses are
quadratic under the refined analysis, and the third is an artifact of adding
the regularizer after the Lewis fixed point instead of inside it.  Taking
\[
 p:=1-\frac1{4q}, \qquad c_k:=8q, \qquad
 C:=8\sqrt{2c_sc_k}, \qquad K:=\frac1{128\sqrt{c_k}}
\]
therefore suffices. Enlarging $K$ relaxes the accuracy $E$ to which the weights must be maintained, a smaller $C$ lengthens the admissible change in the path parameter, and the regularized weight function makes both run on $q$ rather than on $\ell$. Section~\ref{sec:predicted} changes what the chasing game is asked to track rather than any of the three quantities above; the defect-adaptive construction in Section~\ref{sec:defect_adaptive} subsequently separates $R$ from $E$ and supplies the radius used by the final theorem.

The first geometric improvement is a Pythagorean estimate for the projected
Newton direction in the mixed norm.  It replaces two independent triangle
inequalities by one joint bound.
\begin{lemma}[Projector defect]
\label{lem:pythagorean_projector}
Let $x\in\Omega^\circ$ and let $w$ satisfy $\frac45g(x)\leq w\leq\frac54g(x)$. For every $C>0$ and every $y\in\R^m$,
$$\|P_{x,w}y\|_\infty+C\|P_{x,w}y\|_w\leq\|y\|_\infty+\sqrt{C^2+2c_s}\,\|y\|_w,$$
and consequently $c_\gamma\leq\sqrt{1+{2c_s}/{C^2}}$.
\end{lemma}
\begin{proof}
Write $P:=P_{x,w}$, so that $I-P=W^{-1}A_x(A_x^\top W^{-1}A_x)^{-1}A_x^\top$. This matrix is idempotent, and $W(I-P)=A_x(A_x^\top W^{-1}A_x)^{-1}A_x^\top$ is symmetric, so $I-P$ is self-adjoint for the inner product $\langle u,y\rangle_w:=u^\top Wy$. Hence $P$ is an orthogonal projection in that inner product, and with $s:=\|Py\|_w$ and $t:=\|(I-P)y\|_w$ we have $s^2+t^2=\|y\|_w^2$.

Under the hypothesis on $w$, \cite{ls19} bound
$\|I-P\|_{w\to\infty}^2\leq2c_s$.  Applying that bound to $(I-P)y$ rather
than to $y$, and using $(I-P)^2=I-P$, gives
$\|(I-P)y\|_\infty\leq\sqrt{2c_s}\,t$.  The key observation is that the
error may therefore be charged to $t$ instead of to $\|y\|_w$.  Hence
$\|Py\|_\infty\leq\|y\|_\infty+\sqrt{2c_s}\,t$, and Cauchy--Schwarz in
$\R^2$ gives
$\sqrt{2c_s}\,t+Cs\leq\sqrt{2c_s+C^2}\sqrt{t^2+s^2}$.  The second claim
follows from $\sqrt{C^2+2c_s}=C\sqrt{1+{2c_s}/{C^2}}$ together with
$\sqrt{1+{2c_s}/{C^2}}\geq1$.
\end{proof}

The second geometric improvement retains the same Pythagorean slack inside
the Lewis-weight Jacobian.  The next lemma couples its weighted contraction
with a coordinatewise approximation by a diagonal map.

\begin{lemma}[Regularized Jacobian defects]
\label{lem:jacobian_defect}
Let $0<p<1$, let $v>0$ coordinatewise, let $c>0$ be a row scaling and let $w:=w^v_p(\mathrm{Diag}(c)A)$, so that $w=\sigma(\mathbf W^{\vartheta/2}\mathrm{Diag}(c)A)+v$ with $\vartheta=1-2/p$. Let $B$ denote the Jacobian of $\log w$ with respect to $\log c$ and define $b:=Bh$. Then
\begin{itemize}
\item[(a)] \phantomsection\label{item:jacobian_defect_a}
$\|b-ph\|_w^2+\|b\|_w^2\leq p^2\|h\|_w^2$.
\item[(b)] \phantomsection\label{item:jacobian_defect_b}
There is a diagonal matrix $D$ with $0\preceq D\preceq pI$ such that
$\|b-Dh\|_\infty^2\leq\frac{4}{p^2}(p^2\|h\|_w^2-\|b\|_w^2)$.
\item[(c)] \phantomsection\label{item:jacobian_defect_c}
$\|b\|_{w+\infty}\leq p\sqrt{1+{4}/{(p^2C^2)}}\,\|h\|_{w+\infty}$ for every $C>0$.
\end{itemize}
\end{lemma}
\begin{proof}
\noindent{\bf Proof of (a).}
Define $a:=-\vartheta=\frac2p-1>0$, let $\Sigma$ and $P$ be the
leverage-score matrix and the orthogonal projection matrix of
$\mathbf W^{\vartheta/2}\mathrm{Diag}(c)A$, and let
$\Lambda:=\Sigma-P\circ P$.  Lemma~\ref{lem:regularized_weight_wellposed}
makes the fixed point unique and smooth in $\log c$.  Differentiating its
equation and using
$\d\sigma=2\Lambda\,\d\log(\mathbf W^{\vartheta/2}\mathrm{Diag}(c))$
therefore gives $(W+a\Lambda)\,\d\log w=2\Lambda\,\d\log c$, so $B=(W+a\Lambda)^{-1}2\Lambda$. Passing to $w$-normalized coordinates, with
$$S:=W^{-1/2}\Sigma W^{-1/2}, \quad Q:=W^{-1/2}(P\circ P)W^{-1/2}, \quad L:=S-Q,$$
we get $T:=W^{1/2}BW^{-1/2}=2(I+aL)^{-1}L$. Here $\Lambda\succeq0$ gives $L\succeq0$, and $\sigma_i\leq w_i$ gives $S\preceq I$, hence $0\preceq Q\preceq S\preceq I$ and $0\preceq L\preceq I$. The eigenvalues $2\lambda/(1+a\lambda)$ of $T$ increase in $\lambda$ and equal $p$ at $\lambda=1$, so $T$ is symmetric with $0\preceq T\preceq pI$. Writing $x:=W^{1/2}h$, so that $\|h\|_w=\|x\|_2$ and $\|b\|_w=\|Tx\|_2$, the first inequality is the scalar bound $(p-\lambda)^2+\lambda^2\leq p^2$ on $[0,p]$ applied in the eigenbasis of $T$.

\medskip
\noindent{\bf Proof of (b).}
Define $D:=2(I+aS)^{-1}S$, diagonal because $S$ is, with $0\preceq D\preceq pI$ for the same reason as $T$. From $(I+aX)^{-1}X=a^{-1}(I-(I+aX)^{-1})$ we get $K:=T-D=\frac2a((I+aS)^{-1}-(I+aL)^{-1})$, and the resolvent identity together with $S-L=Q$ gives
$$K=-2(I+aS)^{-1}Q(I+aL)^{-1},$$
which is symmetric, being the difference of two symmetric matrices. Factoring $p^2-t^2$ at $t=2\lambda/(1+a\lambda)$ as $(p(1+a\lambda)-2\lambda)(p(1+a\lambda)+2\lambda)$ and using $pa=2-p$,
$$M:=p^2I-T^2=p^2(I-L)(I+\tfrac{4-p}{p}L)(I+aL)^{-2}.$$
Hence $KM^\dagger K=\frac4{p^2}(I+aS)^{-1}Q(I-L)^\dagger(I+\frac{4-p}{p}L)^{-1}Q(I+aS)^{-1}$, the powers of $I+aL$ cancelling because they commute with the other factors. The two middle factors are commuting positive semidefinite functions of $L$ and the second is at most $I$, while $Q\preceq I-S+Q=I-L$ gives $Q(I-L)^\dagger Q\preceq Q$; therefore
$$KM^\dagger K\preceq\frac4{p^2}(I+aS)^{-1}Q(I+aS)^{-1}.$$
Taking the $i$-th diagonal entry, which is legitimate because $I+aS$ is diagonal, and using $Q_{ii}=\sigma_i^2/w_i\leq w_i$, we get
\begin{equation}
\label{eq:jacobian_defect_diagonal}
e_i^\top KM^\dagger Ke_i\leq\frac4{p^2}w_i.
\end{equation}
The pseudo-inverse is used legitimately here: $\ker M$ is the eigenspace of $L$ at the eigenvalue $1$, on which $Q$ vanishes because $0\preceq Q\preceq I-L$, so every column of $K$ is orthogonal to $\ker M$. Since $b-Dh=W^{-1/2}Kx$, Cauchy--Schwarz in the seminorm of $M$ gives
$$|(b-Dh)_i|^2=\frac{|e_i^\top Kx|^2}{w_i}\leq\frac{e_i^\top KM^\dagger Ke_i}{w_i}\,x^\top Mx\leq\frac4{p^2}(p^2\|h\|_w^2-\|b\|_w^2),$$
where the first step follows from $b-Dh=W^{-1/2}Kx$, the second step follows from Cauchy--Schwarz in the seminorm induced by $M$, and the third step follows from Eq.~\eqref{eq:jacobian_defect_diagonal} and $x^\top Mx=p^2\|h\|_w^2-\|b\|_w^2$. Taking the maximum over $i\in[m]$ gives the claimed $\ell_\infty$ bound.

\medskip
\noindent{\bf Proof of (c).}
Define $r:=\|h\|_w$, $s:=\|b\|_w$ and $d:=\sqrt{p^2r^2-s^2}$. Then $\|b\|_\infty\leq\|Dh\|_\infty+\|b-Dh\|_\infty\leq p\|h\|_\infty+\frac2pd$, and Cauchy--Schwarz in $\R^2$ gives $\frac2pd+Cs\leq\sqrt{4/p^2+C^2}\sqrt{d^2+s^2}=pC\sqrt{1+4/(p^2C^2)}\,r$. As $\sqrt{1+4/(p^2C^2)}\geq1$, the claim follows.
\end{proof}

\subsection{Robust centrality geometry}
\label{subsec:robust_centrality_geometry}

The centrality proof needs contraction in the norm defined by the maintained
weights rather than by the inaccessible exact weights.  The next estimate
shows that this change of norm incurs only a quadratic loss.
\begin{lemma}[Robust consistency]
\label{lem:robust_consistency}
Let $b=Bh$ be as in Lemma~\ref{lem:jacobian_defect}-\hyperref[item:jacobian_defect_a]{(a)}, let $\widetilde w>0$ satisfy $\|\log\widetilde w-\log w\|_\infty\leq H\leq1/4$, and define $N(y):=\|y\|_\infty+C\|y\|_{\widetilde w}$. Then, with $L_H:=e^{3H/2}-1\leq2H$,
$$N(b)\leq p(1+\frac{e^{H/2}}{2}(\frac{2}{pC}+L_H)^2)\,N(h).$$
\end{lemma}
\begin{proof}
Define $r:=\|h\|_w$, $s:=\|b\|_w$ and $d:=\sqrt{p^2r^2-s^2}$, which is real by Lemma~\ref{lem:jacobian_defect}-\hyperref[item:jacobian_defect_a]{(a)}. Lemma~\ref{lem:jacobian_defect}-\hyperref[item:jacobian_defect_a]{(a)} gives $\|b-ph\|_w\leq d$, and Lemma~\ref{lem:jacobian_defect}-\hyperref[item:jacobian_defect_b]{(b)} gives $\|b\|_\infty\leq p\|h\|_\infty+\frac2pd$.

Let $F(y):=\|y\|_{\widetilde w}-\|y\|_w$, a positively homogeneous function. For $y\neq0$ define $\theta:=\|y\|_w/\|y\|_{\widetilde w}$, so that $\theta\in[e^{-H/2},e^{H/2}]$. Then $\nabla F(y)=\|y\|_w^{-1}(\theta\widetilde W-W)y$, so its norm dual to $\|\cdot\|_w$ is at most the operator norm of $\theta\widetilde WW^{-1}-I$. The entries of $\widetilde WW^{-1}$ lie in $[e^{-H},e^{H}]$, so that norm is at most $L_H$, and therefore $|F(y)-F(y')|\leq L_H\|y-y'\|_w$.

Hence $C\|b\|_{\widetilde w}-pC\|h\|_{\widetilde w}=C(s-pr)+C(F(b)-F(ph))\leq C(s-pr)+CL_Hd$, using homogeneity and $\|b-ph\|_w\leq d$. Adding the $\ell_\infty$ estimate gives $N(b)-pN(h)\leq(\frac2p+CL_H)d+C(s-pr)$. Write $d=pru$ and $s=pr\sqrt{1-u^2}$ with $u\in[0,1]$; then $s-pr\leq-pru^2/2$, so the right-hand side is at most $pr((\frac2p+CL_H)u-\frac C2u^2)\leq{pr(\frac2p+CL_H)^2}/{(2C)}$. Finally $N(h)\geq C\|h\|_{\widetilde w}\geq Ce^{-H/2}r$, and dividing proves the claim, since $(\frac2p+CL_H)^2/C^2=(\frac2{pC}+L_H)^2$.
\end{proof}

For the rest of the paper define
\[
 c_1:=\sup_{x\in\Omega^\circ}\|g(x)\|_1,
 \qquad
 c_s:=\sup_{x\in\Omega^\circ,\,i\in[m]}
 \frac{\sigma_i(G(x)^{-1/2}A_x)}{g_i(x)}.
\]
The symbol $c_\gamma$ denotes the centrality-equivalence constant in
Eq.~\eqref{eq:computable_centrality}, and $c_k$ is the explicit consistency
scale fixed in Lemma~\ref{lem:sharper_parameters}.

To instantiate the abstract path-following framework, we need uniform bounds
on both the total weight and its leverage sensitivity.  Regularization makes
both quantities absolute up to the natural factor $n$.
\begin{lemma}[Size and sensitivity]
\label{lem:sensitivity}
Let $p=1-\frac1{4q}$, $v:=\frac nm\mathbf 1$ and $g(x):=w^v_p(A_x)$. Then $c_1(g)\leq2n$ and $c_s(g)\leq2$.
\end{lemma}
\begin{proof}
Summing the fixed-point equation gives $\|g(x)\|_1=\sum_{i=1}^m\sigma_i(\mathbf G^{\vartheta/2}A_x)+\|v\|_1=n+n=2n$, which is the first claim. For the second, $\sigma_i\leq1$ and $v_i=n/m$ give $g_i\in[n/m,2]$, so the largest coordinate exceeds the smallest by a factor at most $2m/n$. Define $\gamma:=\frac{1-p}{p}\geq0$ and $M:=\mathbf G^{\vartheta/2}A_x$, so that $\mathbf G^{-1/2}A_x=\mathbf G^{\gamma}M$. For a positive diagonal $D$ one has $M^\top D^2M\succeq(\min_jd_j^2)M^\top M$ and hence $\sigma_i(DM)\leq(d_i^2/\min_jd_j^2)\sigma_i(M)$; with $d_i=g_i^{\gamma}$ and $\sigma_i(M)\leq g_i$ this gives
$$c_s(g)=\max_{i\in[m]}\frac{\sigma_i(\mathbf G^{-1/2}A_x)}{g_i}\leq(\frac{2m}{n})^{2\gamma}=\exp(\frac{2(1-p)}{p}\log\frac{2m}{n})\leq\exp(\frac1{2p})\leq2,$$
where the last two steps use $1-p=\frac1{4q}$ with $\log\frac{2m}{n}\leq q$, and $p\geq\frac34$, which holds because $q\geq\log4>1$.
\end{proof}

For the executable algorithm fix the certified numerical bounds
\[
 \overline c_1:=2n,\qquad
 \overline c_s:=2,\qquad
 \overline c_\gamma:=1+\frac1{128c_k}.
\]
The unbarred quantities remain the analytic suprema above.  Every parameter
and failure test in the pseudocode uses the displayed barred constants; no
supremum over $\Omega^\circ$ is evaluated by the algorithm.

This is the one place where the position of the regularizer matters. In \cite{ls19} the floor is added after the Lewis fixed point, so the exponent $\frac12-\frac1p$ acts on the \emph{unregularized} Lewis weights, whose coordinates have no lower bound; the same computation then yields only $c_s=O(m^{\Theta(1-p)})$, and $c_s=O(1)$ forces $1-p=O(1/\ell)$. Putting the floor inside the fixed point bounds the dynamic range by $2m/n$ and replaces $\ell$ by $q$ throughout.

Since $c_s$ is only required to be an upper bound on the sensitivity, we assume throughout that $c_s\in[1,2]$.

The preceding estimates determine a compatible choice of the weight exponent,
mixed-norm coefficient, and tracking radius.  The next lemma verifies these
parameters and the resulting contraction gap along every Newton segment.
\begin{lemma}[Weight-function parameters]
\label{lem:sharper_parameters}
With $p=1-\frac1{4q}$, $v=\frac nm\mathbf 1$, $g(x)=w^v_p(A_x)$ and
\[
 c_k:=8q,\qquad C:=8\sqrt{2\overline c_sc_k},\qquad
 K:=1/(128\sqrt{c_k}),
\]
one has $c_\gamma(g)\leq1+1/(128c_k)$.  Moreover, there is an
absolute $r_0>0$ such that the following holds.  If a Newton segment starts
at
\[
 \delta:=\delta_t(x,w)\leq R\leq r_0/c_k,
 \qquad
 \|\log w-\log g(x)\|_\infty\leq K,
\]
then, in the maintained norm, at every point of that segment,
\[
 \|G^{-1}J_g(x)(\Phi''(x))^{-1/2}\|_{w+\infty\to w+\infty}
 \leq1-\frac{15}{8c_k}.
\]
\end{lemma}
\begin{proof}
Lemma~\ref{lem:pythagorean_projector} and $C^2=128\overline c_sc_k$ give
$c_\gamma\leq\sqrt{1+1/(64c_k)}\leq1+1/(128c_k)$.

We first verify the tracking band all along the segment without using the
Jacobian estimate that is to be proved.  Write
$x_s:=x+s\Delta$, where $\Delta$ is the projected Newton direction, and
define $h:=\sqrt{\Phi''(x)}\Delta$.  Lemma~\ref{lem:ls_path_facts}
gives $N_w(h)\leq c_\gamma\delta$, hence
$\|h\|_\infty\leq2R$ after decreasing $r_0$.  Coordinatewise
self-concordance gives
\[
 |\frac{\d}{\d s}\log s_i(x_s)|
 \leq\frac{|h_i|}{1-s|h_i|}.
\]
Consequently, for every $s\in[0,1]$,
\[
 \|\log s(x_s)-\log s(x)\|_\infty
 \leq\frac{\|h\|_\infty}{1-\|h\|_\infty}
 \leq4\delta.
\]
The target-stability part of Lemma~\ref{lem:linf_contraction} now gives
\[
 \|\log g(x_s)-\log g(x)\|_\infty\leq16\delta\leq16R.
\]
Fix $r_0$ so that $16R\leq K$ uniformly for $R\leq r_0/c_k$.
Thus
$\|\log w-\log g(x_s)\|_\infty\leq2K$ for the entire
segment.  This closes the former first-exit seam before the operator bound is
invoked.

Apply Lemma~\ref{lem:robust_consistency} with
$\widetilde w:=w$ and $H\leq2K$.  Then
\begin{equation}
\label{eq:robust_consistency_lh_bound}
 L_H\leq2H\leq\frac{1}{32\sqrt{c_k}}
 <\frac{0.032}{\sqrt{c_k}}.
\end{equation}
Moreover, $p\geq3/4$ and $\overline c_s\geq1$ give
\begin{equation}
\label{eq:mixed_norm_coefficient_bound}
 \frac{2}{pC}\leq\frac{1}{3\sqrt2\sqrt{c_k}}
 <\frac{0.236}{\sqrt{c_k}}.
\end{equation}
Therefore
\[
 \frac{e^{H/2}}2(\frac2{pC}+L_H)^2
 \leq\frac{e^{1/8}}2\frac{(0.236+0.032)^2}{c_k}
 <\frac1{8c_k},
\]
where the first step follows from $H/2\leq1/8$,
Eq.~\eqref{eq:mixed_norm_coefficient_bound}, and
Eq.~\eqref{eq:robust_consistency_lh_bound}, and the second step follows from the
numerical inequality $e^{1/8}(0.236+0.032)^2/2<1/8$.
The chain rule through the barrier contributes the diagonal map
$q_i=-\frac12\phi_i'''(x_i)\phi_i''(x_i)^{-3/2}h_i$.
Self-concordance gives $|q_i|\leq|h_i|$, so this map increases neither
component of the maintained norm.  Finally, $p=1-2/c_k$, and hence
\[
 p(1+\frac1{8c_k})
 =1-\frac{15}{8c_k}-\frac1{4c_k^2}
 \leq1-\frac{15}{8c_k},
\]
where the first step follows from $p=1-2/c_k$, and the second step follows
from $c_k>0$.
\end{proof}

\subsection{The mixed-ball response oracle}
\label{sec:predicted}
\label{subsec:mixed_ball_oracle}

The chasing game is played with an enlargement parameter: the ideal weights move inside a set $U$ and the player is allowed to move inside $(1+\epsilon_{\mathrm{slack}})U$. The overshoot $\epsilon_{\mathrm{slack}}|U|$ is charged against the centrality margin, which is only $\Theta(\delta/c_k)$, so \cite{ls19} must take $\epsilon_{\mathrm{slack}}=\Theta(1/c_k)$. Their chasing theorem then returns a tracking error $\frac{12E}{\epsilon_{\mathrm{slack}}}\log\frac{12m\tau}{\epsilon_{\mathrm{slack}}}$, and requiring this to be at most $K$ forces $E=\Theta(K/(c_kL_{\mathrm{ch}}))$. One whole factor $c_k$ in $E^{-1}$ is therefore the price of the enlargement.

That price is avoidable. Suppose that before moving the weights we can \emph{predict} the movement $u_k:=\log g(x_k)-\log g(x_{k-1})$ to within $\Theta(\delta/c_k)$ in the mixed norm. Chase the prediction error instead of the movement itself: the residual game has a movement set of radius $\Theta(\delta/c_k)$ to begin with, so a \emph{constant} enlargement already costs only $\Theta(\delta/c_k)$, and the tracking error becomes $O(EL_{\mathrm{ch}})$ rather than $O(Ec_kL_{\mathrm{ch}})$. The stochastic residual game below proves the needed tracking statement; Section~\ref{sec:defect_adaptive} supplies the final predictor.

Prediction buys a second saving, of a different kind, and it is the larger of the two. Once only the residual is chased, what the chasing theorem constrains is the accuracy $E$ with which $g(x)$ is observed and the radius of the residual movement set; the centrality of the iterate enters neither hypothesis. The final defect-adaptive construction exploits this separation by taking $R=\Theta(CE)=\Theta(1/\ell)$ while keeping $E=\Theta(q^{-1/2}/\ell)$.

The greedy player uses the following deterministic oracle.  Recording its
certified approximation guarantee closes one of the two non-system
optimizations inside a centering step.  All scalar-work and scalar-depth
bounds for this oracle and the defect projection use the exact real-RAM
model; the theorem makes no bit-complexity claim.  In particular, logarithms,
exponentials, linear-system solves, equality tests, and emptiness tests are
exact real operations.  The finite-precision certificates below show that an
approximate scalar optimizer is never mistaken for an exact KKT point; they do
not assert a Turing-model bit bound.  A bit-complexity version would replace
the zero and endpoint-equality branches by certified tolerance tests and is
outside the present theorem.

\begin{algorithm}[!ht]
\caption{Linear optimization over a mixed ball}
\label{alg:mixed_ball_linear_oracle}
\begin{algorithmic}[1]
\Procedure{MixedBallLinearOracle}{$w,C,s,a,\epsilon_{\rm g}$}
\Comment{Lemma~\ref{lem:mixed_ball_linear_oracle}}
\If{$a=0$}
  \State \Return $0$.
\EndIf
\State $a\gets a/\|a\|_\infty$,
$G\gets\|a\|_1+C^{-1}\|W^{-1/2}a\|_2$, and
$H_{\rm lb}\gets s/(1+C\sqrt{\max_{i\in[m]}w_i})$.
\State $J\gets\lceil\log_2(2Gs/(\epsilon_{\rm g}H_{\rm lb}))\rceil$,
$t_-\gets0$, $t_+\gets s$, and $\chi_{\rm best}\gets0$.
\For{$j=1,\ldots,J$}
  \State $t\gets(t_-+t_+)/2$, $A\gets(s-t)/C$, and
  $y_i\gets-t\operatorname{sign}(a_i)$.
  \If{$\|y\|_w\leq A$}
    \State $\lambda\gets0$, $\chi\gets y$, and
    $\alpha_i\gets|a_i|$ for every $i$.
  \Else
    \State $b_i\gets |a_i|/(w_it)$; sort the $b_i$ in nonincreasing order.
    \State For each resulting prefix $I$ with positive denominator, compute
    \[
      \lambda_I\gets
      \sqrt{\frac{\sum_{i\in[m]\setminus I} a_i^2/w_i}
      {A^2-t^2\sum_{i\in I} w_i}}
    \]
    \State Choose the prefix $I$ satisfying
    $b_i\geq\lambda_I$ on $I$ and $b_i\leq\lambda_I$ off $I$,
    and set $\lambda\gets\lambda_I$.
    \State $\chi_i\gets-\operatorname{sign}(a_i)
    \min\{t,|a_i|/(\lambda w_i)\}$ and
    $\alpha_i\gets(|a_i|-\lambda w_it)_+$.
  \EndIf
  \State If $\langle a,\chi\rangle<
  \langle a,\chi_{\rm best}\rangle$, set
  $\chi_{\rm best}\gets\chi$.
  \State $g\gets\lambda A/C-\sum_{i=1}^m\alpha_i$.
  \If{$g>0$}
    \State $t_+\gets t$.
  \ElsIf{$g<0$}
    \State $t_-\gets t$.
  \Else
    \State \Return $\chi$.
  \EndIf
\EndFor
\State \Return $\chi_{\rm best}$.
\EndProcedure
\end{algorithmic}
\end{algorithm}

The finite procedure above is useful only if it returns a feasible mixed-ball
point with a certified near-optimal linear objective.  The next lemma proves
that guarantee and records its scalar work and depth.
\begin{lemma}[Linear optimization over a mixed ball]
\label{lem:mixed_ball_linear_oracle}
Let $w>0$, $C,s>0$, and $a\in\R^m$.  For every
$\epsilon_{\rm g}\in(0,1/10)$, the procedure
\textnormal{\textsc{MixedBallLinearOracle}} (Algorithm~\ref{alg:mixed_ball_linear_oracle}) returns
$\chi$ with $N_w(\chi)\leq s$ and
\[
 \langle a,\chi\rangle
 \leq-(1-\epsilon_{\rm g})sN_w^*(a).
\]
If $0<\underline w\leq w_i\leq\overline w$ for every $i$, define
\[
 L_w:=\log(2+(1+C\sqrt{\overline w})
 (m+\sqrt{m/\underline w}/C)).
\]
The algorithm uses no permitted systems and has $O(mJ\log m)$ scalar work and
$O(J\log m)$ arithmetic depth, where
$J=O(L_w+\log(1/\epsilon_{\rm g}))$.
\end{lemma}
\begin{proof}
If $a=0$, the claim is immediate.  Scaling $a$ to
$\|a\|_\infty=1$ changes neither the optimizer nor a relative objective
guarantee.  For $t\in[0,s]$, define $A(t)=(s-t)/C$ and let
$\varphi(t)$ be the minimum of $\langle a,x\rangle$ subject to
$|x_i|\leq t$ and $\|x\|_w\leq A(t)$.  This is the partial minimum of a
jointly convex problem, so $\varphi$ is convex.

We first verify that every operation in the pseudocode is explicit.  If
$y_i=-t\operatorname{sign}(a_i)$ lies in the $w$-ball, it is the exact
fixed-$t$ minimizer and the ball multiplier is $\lambda=0$.  Otherwise the
ball is active and the KKT equations~\cite{kuhntucker51} give
\[
 x_i=-\operatorname{sign}(a_i)
 \min\{t,|a_i|/(\lambda w_i)\}.
\]
For a proposed clipped set $I$, the equality $\|x\|_w=A(t)$ gives exactly
the displayed formula for $\lambda$.  Sorting
$b_i=|a_i|/(w_it)$ partitions the positive line into at most $m+1$
regions.  Scanning the prefixes finds a region whose formula has positive
denominator and satisfies its clipping inequalities; continuity and
monotonicity of $\|x(\lambda)\|_w$ give existence, and all resulting
vectors coincide at a tie.  Thus the scan returns the exact fixed-$t$
minimizer using $O(m\log m)$ scalar work and $O(\log m)$ depth.

At that KKT tuple define
\[
 \alpha_i=(|a_i|-\lambda w_it)_+.
\]
This is the multiplier of the active face with sign opposite to $a_i$;
the other face multiplier is zero.  Differentiating only the moving
constraints gives the exact subgradient
\[
 g(t)=\frac{\lambda A(t)}C-\sum_{i=1}^m\alpha_i\in\partial\varphi(t).
\]
There is no multiplier ambiguity at a clipping tie because then
$\alpha_i=0$.  Moreover $\sum_{i=1}^m\alpha_i\leq\|a\|_1$.  Coordinatewise,
$\lambda w_i|x_i|\leq|a_i|$, and hence
\[
 \lambda A(t)\leq\|W^{-1/2}a\|_2.
\]
Consequently every returned subgradient has magnitude at most
\[
 G:=\|a\|_1+C^{-1}\|W^{-1/2}a\|_2.
\]

The sign of an exact subgradient preserves an interval containing a minimizer
of the convex function $\varphi$.  After $J$ iterations that interval has
length at most $s2^{-J}$.  The midpoint evaluated at the last iteration is
an endpoint of the retained interval, so it is within this distance of a
minimizer.  The subgradient bound gives
\[
 \langle a,\chi_{\rm best}\rangle-\min_t\varphi(t)
 \leq Gs2^{-J}\leq\epsilon_{\rm g}H_{\rm lb}/2.
\]
If $|a_j|=1$, the one-coordinate vector with
$|x_j|=s/(1+C\sqrt{w_j})$ proves
$-\min_t\varphi(t)\geq H_{\rm lb}$.  This yields the stated relative
support guarantee.

Finally,
$G\leq m+\sqrt{m/\underline w}/C$, so the displayed definition of $L_w$
bounds $\log(2Gs/H_{\rm lb})$.  Each of the $J$ iterations performs one
sort and one prefix scan, for total $O(mJ\log m)$ scalar work and
$O(J\log m)$ arithmetic depth.  In the application,
$e^{-K}n/m\leq w_i\leq2e^K$ and $C^2=O(q)$, so $L_w=O(L_*)$.
\end{proof}

We apply the mixed-ball oracle to the softmax potential using only an
approximate observation of the residual.  The following one-step estimate
quantifies the resulting descent and the cost of the observation error.
\begin{lemma}[Observed greedy softmax step]
\label{lem:observed_softmax_step}
Let $N(z)=\|z\|_\infty+C\|z\|_w$,
$V=\{z:N(z)\leq s\}$, and
$\Phi(d)=2\sum_{i=1}^m\cosh(\mu d_i)$.  Suppose
$b\in V/8$, $s\leq E/100$, $\mu E=1/100$, and
$\|q_{\rm obs}-q\|_\infty\leq5E$.  Let $\chi$ be the output of
Lemma~\ref{lem:mixed_ball_linear_oracle} for radius
$(1+\epsilon_0)s$, gradient $\nabla\Phi(q_{\rm obs})$,
$\epsilon_0=1/10$, and a sufficiently small absolute
$\epsilon_{\rm g}$.  Then
\[
 \Phi(q-b+\chi)
 \leq\Phi(q)-\frac34sN^*(\nabla\Phi(q))+C_{\rm sm}\mu ms
\]
for an absolute constant $C_{\rm sm}$.
\end{lemma}
\begin{proof}
The elementary inequalities for
$\phi'(t)=2\mu\sinh(\mu t)$ imply, whenever $|r-t|\leq5E$,
\[
 |\phi'(r)|\geq e^{-5\mu E}|\phi'(t)|-\mu,
 \qquad
 |\phi'(r)-\phi'(t)|
 \leq(e^{5\mu E}-1)|\phi'(t)|+\mu.
\]
Because $V$ is unconditional, these coordinatewise inequalities pass to its
support function, with an additive loss at most $\mu ms$.
The oracle guarantee and $b\in V/8$ therefore give
\[
 \langle\nabla\Phi(q),-b+\chi\rangle
 \leq-((1+\epsilon_0)(1-\epsilon_{\rm g})e^{-5\mu E}
             -\frac18-O(\mu E))
       sN^*(\nabla\Phi(q))+O(\mu ms).
\]
The coefficient in parentheses is larger than $4/5$ after fixing
$\epsilon_{\rm g}$ small.  Also
$N(-b+\chi)\leq(1/8+1+\epsilon_0)s$ and
$\|-b+\chi\|_\infty\leq(1/8+1+\epsilon_0)s$.  Taylor's theorem,
$\phi''=\mu^2\phi$, and
\[
 N(z\odot z)\leq\|z\|_\infty N(z),
 \qquad
 \mu\phi(t)\leq|\phi'(t)|+2\mu,
\]
bound the quadratic remainder by
$O(\mu s)sN^*(\nabla\Phi(q))+O(\mu^2ms^2)$.
Since $\mu s\leq10^{-4}$, this loss is smaller than the slack between
$4/5$ and $3/4$, proving the claim.
\end{proof}

\section{Residual chasing and retained-defect geometry}
\label{sec:residual_defect_geometry}

Section~\ref{subsec:residual} proves the stochastic chasing guarantee that
keeps the maintained weights close to their moving target despite prediction
error.  Section~\ref{subsec:large_radius_defect_geometry} then identifies the
Jacobian defect retained along a finite Newton segment and converts it into a
large-radius split predictor.  Finally,
Section~\ref{subsec:defect_body_projection} gives a computable projection that
turns the predictor and its defect certificate into a legal residual-game
move.

\subsection{Stochastic residual chasing}
\label{subsec:residual}

We combine the observed greedy step with predictable bias and martingale
fluctuations in the residual motion.  The next lemma is the tracking statement
that will be invoked by every later centering call.
\begin{lemma}[Stochastic residual chasing]
\label{lem:residual_chasing}
Let $s_k>0$ be predictable, let $\widehat u_k$ be a predicted ideal-weight
increment, and write $e_k:=u_k-\widehat u_k$,
$b_k:=\E[e_k\mid\mathcal F_{k-1}]$ and $Z_k:=e_k-b_k$.
For the $k$th update write
\[
 N_k(z):=\|z\|_\infty+C\|z\|_{w_{k-1}},
\]
and let $N_k^*$ be its dual norm.  Thus every mixed norm in the $k$th
hypotheses, movement set, and movement conclusion is evaluated at the
predictable pre-step weight $w_{k-1}$.  Define
$\beta:=1+C\sqrt{2\overline c_1}$ and assume
$$N_k(b_k)\leq\frac{s_k}{8},
\qquad
\|Z_k\|_\infty\leq\frac{s_k}{2}\ \text{almost surely},
\qquad
\E[N_k(Z_k)^2\mid\mathcal F_{k-1}]\leq\nu^2s_k^2,
\qquad \nu:=1/5.$$
Play the game of \cite{ls19} on the shifted pair $\bar x_k:=\log w_k-\sum_{j=1}^k\widehat u_j$, $\bar y_k:=\log g(x_k)-\sum_{j=1}^k\widehat u_j$, with movement sets $V_k:=\{v:N_k(v)\leq s_k\}$, enlargement $\epsilon_0:=1/10$, softmax parameter $\mu:=1/(100E)$, and observations of $\bar y_k$ carrying $\ell_\infty$ error at most $E$.  The player uses Lemma~\ref{lem:mixed_ball_linear_oracle} with a fixed sufficiently small $\epsilon_{\rm g}$; its constant relative support loss is included in the numerical slack below. If $s_k\leq E/100$ for every $k$ and
$$E\leq\frac{K}{400\,L_{\mathrm{ch}}},$$
let $A_{\rm ch}\geq3$ be a sufficiently large absolute constant, define
\[
 M_{\rm ch}:=A_{\rm ch}m\beta,
 \qquad d_0:=\log w_0-\log g(x_0),
 \qquad \Phi(d):=2\sum_{i=1}^m\cosh(\mu d_i),
\]
and assume
\[
 \Phi(d_0)\leq M_{\rm ch}.
\]
Choose $A_0\geq40A_{\rm ch}$ and require
$L_{\rm ch}\geq\log(A_0Nm\beta)$.  Then with probability at least
$19/20$ the stronger invariant
\[
 \|\log w_k-\log g(x_k)\|_\infty\leq K/4
\]
holds simultaneously for all $k\leq N$. If in addition
$$N_k(\widehat u_k)+1.1s_k
\leq(1-\tfrac1{c_k})\delta_k,$$
then, deterministically,
$$N_k(\log w_k-\log w_{k-1})
\leq(1-\tfrac1{c_k})\delta_k.$$
The same conclusions hold for an externally stopped process.  Namely, let
$\tau_{\rm ext}$ be such that the decision to execute step $k$, equivalently
$\mathbf1_{\{k<\tau_{\rm ext}\}}$, is measurable in the pre-noise
sigma-field $\mathcal F_{k-1}$, and freeze the process from
$\tau_{\rm ext}$ onward.  Then the tracking and movement conclusions hold
through $N\wedge(\tau_{\rm ext}-1)$; on $\{\tau_{\rm ext}>N\}$ they give
the full-horizon conclusions above.
\end{lemma}
\begin{proof}
The shifted target moves by $\bar y_k-\bar y_{k-1}=e_k$ and the player by $\bar x_k-\bar x_{k-1}=\chi_k\in(1+\epsilon_0)V_k$, so $d_k:=\bar x_k-\bar y_k=\log w_k-\log g(x_k)$ satisfies $d_k=d_{k-1}-e_k+\chi_k$.
Define $\tau_K:=\inf\{k\geq1:\|d_k\|_\infty>K\}$ and include the
crossing update $k=\tau_K$ before freezing.  The external stop instead
freezes before step $\tau_{\rm ext}$.  Thus the activity indicator is
$\mathbf1_{\{k\leq\tau_K\}}\mathbf1_{\{k<\tau_{\rm ext}\}}$,
which is $\mathcal F_{k-1}$-measurable because the internal test uses only
$d_0,\ldots,d_{k-1}$.  For every active step, including the crossing
update, the pre-step tracking invariant gives
\[
 \|w_{k-1}\|_1\leq e^K\|g(x_{k-1})\|_1\leq2c_1.
\]
Thus, for the stopped process, each predictable $V_k$ is symmetric and
convex with $\ell_\infty$ radius at most $s_k\leq E/100$, and
$\|v\|_{w_{k-1}}\leq\|v\|_\infty\sqrt{\|w_{k-1}\|_1}$ gives
$N_k(v)\leq\beta\|v\|_\infty$, so $V_k$ contains an $\ell_\infty$ ball
of radius $s_k/\beta$.  Write
$\Phi(d):=2\sum_{i=1}^m\cosh(\mu d_i)$ and, for the current step,
abbreviate $N:=N_k$ and $N^*:=N_k^*$, so that the support function of
$V_k$ is $h_k(a)=s_kN^*(a)$.

\emph{The centered part.} The mixed norm is absolute, and coordinatewise
$\|z\odot z\|_\infty=\|z\|_\infty^2$ while
$\|z\odot z\|_{w_{k-1}}\leq
\|z\|_\infty\|z\|_{w_{k-1}}$.  Thus, with
$\alpha:=\|z\|_\infty$ and $\gamma:=C\|z\|_{w_{k-1}}$,
$$N(z\odot z)\leq\alpha^2+\alpha\gamma=\alpha(\alpha+\gamma)\leq N(z)^2 .$$
Let $x$ be $\mathcal F_{k-1}$-measurable and define
$\phi(t):=e^{\mu t}+e^{-\mu t}$, acting coordinatewise on vectors.

We can show
\[
\begin{aligned}
 \E[\Phi(x-Z_k)\mid\mathcal F_{k-1}]
 &\leq\Phi(x)+\tfrac12\mu^2
 \E[e^{\mu\|Z_k\|_\infty}
       \langle\phi(x),Z_k\odot Z_k\rangle\mid\mathcal F_{k-1}]\\
 &\leq\Phi(x)+\tfrac12\mu^2e^{\mu s_k/2}
 \E[\langle\phi(x),Z_k\odot Z_k\rangle\mid\mathcal F_{k-1}],
\end{aligned}
\]
where the first step follows from Taylor's theorem, $\phi''=\mu^2\phi$,
and $\E[Z_k\mid\mathcal F_{k-1}]=0$, and the second step follows from
$\|Z_k\|_\infty\leq s_k/2$ and the
$\mathcal F_{k-1}$-measurability of $s_k$.
Since $\cosh\leq|\sinh|+1$ we have $\mu\phi(t)\leq|\phi'(t)|+2\mu$, whence
$$\mu\langle\phi(x),z\odot z\rangle
\leq\langle|\nabla\Phi(x)|,z\odot z\rangle+2\mu\|z\odot z\|_1
\leq N^*(\nabla\Phi(x))\,N(z)^2+2\mu m\|z\|_\infty^2 ,$$
by H\"older for the pair $N,N^*$ and the displayed quadratic bound. With $\mu\|Z_k\|_\infty\leq\mu s_k/2\leq10^{-4}$ and $\E[N(Z_k)^2\mid\mathcal F_{k-1}]\leq\nu^2s_k^2$ this gives
$$\E[\Phi(x-Z_k)\mid\mathcal F_{k-1}]\leq\Phi(x)+\nu^2\mu s_k\,h_k(\nabla\Phi(x))+2\nu^2\mu^2ms_k^2 .$$
The key feature of the construction is that a mean-zero error enters the
potential only through its \emph{second} moment, and only against the same
support function $h_k$ that the player's move decreases.

\emph{The bias and the player.} We make the conditioning order explicit because the
player's greedy move is chosen after $Z_k$ is revealed.  Condition first on
$\mathcal F_{k-1}$ and then on a realization of $Z_k$, and define
$q_k:=d_{k-1}-Z_k$.  The observed residual used by the greedy move differs from $q_k$ by at
most $E+\|b_k\|_\infty\leq E+s_k/8<5E$.  Lemma~\ref{lem:observed_softmax_step},
applied with $b_k\in V_k/8$, gives
\[
 \Phi(q_k-b_k+\chi_k)
 \leq \Phi(q_k)-\frac34h_k(\nabla\Phi(q_k))+C_0\mu ms_k.
\]
This is the required deterministic one-step estimate; in particular, the
factor $1/8$ from the bias body is present in the descent coefficient.

It remains to compare the post-noise gradient in this display with the
pre-noise gradient in the centered Taylor estimate.  The
coordinate inequalities for $\phi'(t)=2\mu\sinh(\mu t)$ imply that whenever
$\|x-y\|_\infty\leq a\leq1/(5\mu)$ and $V_k$ is symmetric and contained in
an $\ell_\infty$ ball of radius $s_k$,
\[
 h_k(\nabla\Phi(x))
 \geq e^{-\mu a}h_k(\nabla\Phi(y))-\mu ms_k.
\]
Indeed, coordinatewise
$|\phi'(x_i)|\geq e^{-\mu a}|\phi'(y_i)|-\mu$, and symmetry of $V_k$
allows the maximizing vector to be chosen with the signs of the gradient;
summing the additive loss against that vector costs at most $\mu ms_k$.
Taking $x=q_k$, $y=d_{k-1}$ and
$a=\|Z_k\|_\infty\leq s_k/2$ therefore gives, for every realization of
$Z_k$,
\[
 h_k(\nabla\Phi(q_k))
 \geq e^{-\mu s_k/2}h_k(\nabla\Phi(d_{k-1}))-\mu ms_k.
\]
This is the comparison needed before averaging; no independence between the
greedy move and $Z_k$ is being asserted.

Take conditional expectations in the deterministic player inequality and
combine it with the centered estimate above.  Since $\mu s_k\leq10^{-4}$,
$2\nu^2\mu^2ms_k^2=O(\nu^2\mu ms_k)$, and the fixed choice
$\nu=1/5$ leaves ample slack below the $3/4$ descent coefficient, the result
is
\[
 \E[\Phi(d_k)\mid\mathcal F_{k-1}]
 \leq\Phi(d_{k-1})-\frac12h_k(\nabla\Phi(d_{k-1}))+O(\mu ms_k).
\]

\emph{From one step to all steps.} Since $V_k\supseteq\{v:\|v\|_\infty\leq s_k/\beta\}$ we have $h_k(a)\geq(s_k/\beta)\|a\|_1$, and $|\phi'(t)|\geq\mu(\phi(t)-2)$ gives $\|\nabla\Phi(d)\|_1\geq\mu(\Phi(d)-2m)$. Choose $A_{\rm ch}$ so that
$$\Phi(d_{k-1})\geq M_{\rm ch}
\qquad\Longrightarrow\qquad
\E[\Phi(d_k)\mid\mathcal F_{k-1}]\leq\Phi(d_{k-1}),$$
so $\Phi$ is a nonnegative supermartingale as long as it stays above
$M_{\rm ch}$.  By hypothesis $\Phi(d_0)\leq M_{\rm ch}$. One step moves $d$
by $\|e_k\|_\infty+\|\chi_k\|_\infty=O(s_k)$ in $\ell_\infty$, hence changes
$\Phi$ by a factor at most $e^{O(\mu s_k)}\leq2$, so every excursion above
$M_{\rm ch}$ begins at a value at most $2M_{\rm ch}$.

Define the excursion stopping times as follows.  Define
$\tau_0:=0$ and, successively,
\[
 \sigma_j:=\inf\{k>\tau_{j-1}:\Phi(d_k)>M_{\rm ch}\},\qquad
 \tau_j:=\inf\{k\geq\sigma_j:\Phi(d_k)\leq M_{\rm ch}\}\wedge(N+1),
\]
with the convention that an empty infimum is $N+1$.  Conditional on
$\mathcal F_{\sigma_j}$, the process
$\Phi(d_{(\sigma_j+r)\wedge\tau_j})$ is a nonnegative supermartingale:
before $\tau_j$ its state is above $M_{\rm ch}$.  The one-step drift
inequality applies to every active update, including $k=\tau_K$; after
freezing the potential is constant.  Stopping this process at $\tau_j$
therefore preserves the supermartingale property.
Moreover $\Phi(d_{\sigma_j})\leq2M_{\rm ch}$.  Doob's maximal inequality
therefore gives
\[
 \Pr[\max_{\sigma_j\leq k\leq\tau_j-1}\Phi(d_k)>40NM_{\rm ch}
      \mid\mathcal F_{\sigma_j}]
 \leq\frac1{20N}.
\]
There are at most $N$ nonempty excursions.  A conditional union bound shows
that
\[
 \Pr[\max_{0\leq k\leq N}\Phi(d_k)>40NM_{\rm ch}]\leq\frac1{20}.
\]
On the complementary event,
\[
 \|d_k\|_\infty
 \leq100E\log(40NA_{\rm ch}m\beta)
 \leq100E L_{\rm ch}
 \leq\frac K4.
\]
Thus $\tau_K$ cannot occur before $\tau_{\rm ext}$ or before the end of the
horizon.  Removing the tracking stop proves the externally stopped claim;
on $\{\tau_{\rm ext}>N\}$ it proves the original claim.

\emph{The movement.} This part is deterministic, independently of the
estimator outputs.  Since $N_k(\chi_k)\leq(1+\epsilon_0)s_k=1.1s_k$, the additional
hypothesis gives
$$N_k(\log w_k-\log w_{k-1})=N_k(\widehat u_k+\chi_k)
\leq(1-\frac1{c_k})\delta_k,$$
where the first step follows from the definition of the shifted player and the
update $\bar x_k-\bar x_{k-1}=\chi_k$, and the second step follows from the
triangle inequality, $N_k(\chi_k)\leq(1+\epsilon_0)s_k=1.1s_k$, and the
additional hypothesis of the lemma.
\end{proof}

Because the predictor is a Jacobian--vector product, we must control the
variation of the Jacobian along a centering step.  The resolvent
representation below provides the appropriate form.

\begin{lemma}[Resolvent form]
\label{lem:resolvent_form}
With $\Lambda:=\Sigma-P\circ P$ and $\mathcal L:=W^{-1}\Lambda$, the Jacobian from Lemma~\ref{lem:jacobian_defect}-\hyperref[item:jacobian_defect_a]{(a)} is
$$B=2(I+a\mathcal L)^{-1}\mathcal L=\frac2a(I-(I+a\mathcal L)^{-1}),$$
and $(I+a\mathcal L)^{-1}$ is entrywise nonnegative with all row sums equal to $1$, and is self-adjoint and positive definite for $\langle\cdot,\cdot\rangle_w$. In particular
$$\|(I+a\mathcal L)^{-1}\|_{\infty\to\infty}=1 \qquad\text{and}\qquad \|(I+a\mathcal L)^{-1}\|_{w\to w}\leq1 .$$
\end{lemma}
\begin{proof}
Lemma~\ref{lem:regularized_weight_wellposed} justifies the derivative used
here, and the proof of Lemma~\ref{lem:jacobian_defect}-\hyperref[item:jacobian_defect_a]{(a)} gives
$B=(W+a\Lambda)^{-1}2\Lambda$.  Hence
$B=(I+aW^{-1}\Lambda)^{-1}W^{-1}2\Lambda
=2(I+a\mathcal L)^{-1}\mathcal L$, and the identity $(I+aX)^{-1}X=a^{-1}(I-(I+aX)^{-1})$ gives the second form. Now $\Lambda_{ii}=\sigma_i-\sigma_i^2$ and $\Lambda_{ij}=-P_{ij}^2$ for $j\neq i$, so $I+a\mathcal L$ has positive diagonal and nonpositive off-diagonal entries; and $\sum_{j=1}^mP_{ij}^2=P_{ii}=\sigma_i$ gives $\Lambda\mathbf 1=0$, so all its row sums are exactly $1$. It is therefore strictly diagonally dominant with nonpositive off-diagonals, hence a nonsingular $M$-matrix, so its inverse is entrywise nonnegative; and $(I+a\mathcal L)\mathbf 1=\mathbf 1$ gives $(I+a\mathcal L)^{-1}\mathbf 1=\mathbf 1$. A nonnegative matrix with unit row sums has $\ell_\infty$ operator norm $1$. Finally $W\mathcal L=\Lambda$ is symmetric, so $\mathcal L$ is $w$-self-adjoint, and $\langle y,\mathcal Ly\rangle_w=y^\top\Lambda y\geq0$ since $\Lambda\succeq0$; hence the $w$-spectrum of $I+a\mathcal L$ lies in $[1,\infty)$. Both properties hold for \emph{every} positive $w$, not only at the Lewis fixed point: $\Lambda\mathbf 1=0$ and the symmetry of $\Lambda$ are identities of the projection of $\mathbf W^{\vartheta/2}\mathrm{Diag}(c)A$ whatever $w$ is, because $\Sigma$ is by definition the diagonal of that projection. They therefore survive the substitution of an approximate weight vector for the exact one; the resulting stability estimate is quantified below.
\end{proof}

The resolvent representation converts changes of the Lewis-weight fixed point
into a perturbation problem for a stochastic $M$-matrix.  The next lemma uses
that structure to control the Jacobian simultaneously in both maintained norms.
\begin{lemma}[Jacobian stability]
\label{lem:jacobian_stability}
For two row scalings $c,c'>0$, let
$w=w_p^v(\operatorname{Diag}(c)A)$ and
$w'=w_p^v(\operatorname{Diag}(c')A)$ be the corresponding exact regularized
Lewis vectors, and let $B,B'$ be the Jacobians of their logarithms with
respect to the logarithmic row scalings.  Suppose
$\|\log(c'/c)\|_\infty\leq\rho$ and
$\|\log(w'/w)\|_\infty\leq\rho\leq1/16$. Then
$$\|B'-B\|_{w+\infty\to w+\infty}=O(\rho).$$
\end{lemma}
\begin{proof}
By Lemma~\ref{lem:resolvent_form} and the resolvent identity,
$$B'-B=2(I+a\mathcal L')^{-1}(\mathcal L'-\mathcal L)(I+a\mathcal L)^{-1},$$
with no conjugation by powers of $W$. The two outer factors are contractions in $\ell_\infty$ and, up to a factor $e^{\rho}\leq2$ coming from $w'$ versus $w$, in $\|\cdot\|_w$. Hence $\|B'-B\|_{\infty\to\infty}\leq4\|\Delta\mathcal L\|_{\infty\to\infty}$ and $\|B'-B\|_{w\to w}\leq4\|\Delta\mathcal L\|_{w\to w}$, and since an operator bounded by $\alpha$ in both $\|\cdot\|_\infty$ and $\|\cdot\|_w$ is bounded by $\alpha$ in $\|\cdot\|_{w+\infty}$, it suffices to bound $\Delta\mathcal L:=\mathcal L'-\mathcal L$ in those two norms.

\emph{The $\ell_\infty$ norm.} The projection is that of $\mathbf W^{\vartheta/2}\mathrm{Diag}(c)A$, so the two hypotheses combine into a single drift of its rescaling,
$$\rho_*:=(1+\tfrac{|\vartheta|}{2})\rho\leq2\rho\leq\tfrac18,$$
using $|\vartheta|=\frac2p-1\leq2$ for $p\geq\frac23$; it is $\rho_*$, not $\rho$, with which Lemma~\ref{lem:relative_drift} must be applied. Row by row,
$$\sum_{j=1}^m|\Delta\mathcal L_{ij}|\leq\frac1{w'_i}\sum_{j=1}^m|\Lambda'_{ij}-\Lambda_{ij}|+|\frac{w_i}{w'_i}-1|\frac1{w_i}\sum_{j=1}^m|\Lambda_{ij}| .$$
The exact identity $\sum_{j=1}^m|\Lambda_{ij}|=2\sigma_i(1-\sigma_i)\leq2w_i$ makes the second term at most $3\rho$. For the first, Lemma~\ref{lem:relative_drift} gives $\|(P'-P)e_i\|_2\leq10\rho_*\sqrt{\sigma_i}$ and $\sigma'_i=e^{\pm4\rho_*}\sigma_i$, whence $\|P'e_i\|_2+\|Pe_i\|_2\leq(1+e^{2\rho_*})\sqrt{\sigma_i}\leq2.3\sqrt{\sigma_i}$ and
$$\sum_{j=1}^m|P'^2_{ij}-P^2_{ij}|\leq\|(P'-P)e_i\|_2(\|P'e_i\|_2+\|Pe_i\|_2)\leq23\rho_*\,\sigma_i\leq46\rho\,\sigma_i;$$
the diagonal contributes $|\sigma'_i-\sigma_i|+|\sigma_i'^2-\sigma_i^2|\leq(e^{4\rho_*}-1)(2+e^{4\rho_*})\sigma_i\leq38\rho\sigma_i$, using $\sigma_i\leq1$. Dividing by $w'_i\geq e^{-\rho}w_i\geq e^{-\rho}\sigma_i$ leaves at most $90\rho$, so $\|\Delta\mathcal L\|_{\infty\to\infty}\leq93\rho$.

\emph{The $w$ norm.} Here reversibility does all the work, and no spectral perturbation argument is needed. Since $W\mathcal L=\Lambda$ is symmetric, $w_i\mathcal L_{ij}=w_j\mathcal L_{ji}$, and likewise for the primed quantities. Hence
$$\frac{w_i}{w_j}\,\Delta\mathcal L_{ij}=\frac{w_iw'_j}{w'_iw_j}\,\mathcal L'_{ji}-\mathcal L_{ji}, \qquad \frac{w_iw'_j}{w'_iw_j}=e^{\pm2\rho},$$
so the weighted column sums of $\Delta\mathcal L$ are controlled by its row sums. Writing $\alpha_r:=\|\Delta\mathcal L\|_{\infty\to\infty}$ and $\alpha_c$ for that weighted column sum,
$$\alpha_c\leq\alpha_r+(e^{2\rho}-1)\|\mathcal L'\|_{\infty\to\infty}\leq93\rho+5\rho\leq98\rho,$$
using $\|\mathcal L\|_{\infty\to\infty}=\max_{i\in[m]}2\sigma_i(1-\sigma_i)/w_i\leq2$. The Schur test in $\ell_2(w)$ states that $\|X\|_{w\to w}^2$ is at most the product of the two, and follows from Cauchy--Schwarz applied to $\sum_{i=1}^mw_i(\sum_{j=1}^mX_{ij}d_j)^2$; it gives $\|\Delta\mathcal L\|_{w\to w}\leq\sqrt{93\cdot98}\,\rho\leq96\rho$. Collecting, $\|B'-B\|_{w+\infty\to w+\infty}\leq4\max\{93\rho,96\rho\}\leq384\rho$.
\end{proof}

The predictor is evaluated over an entire Newton segment rather than at one
endpoint.  The next lemma integrates the Jacobian stability estimate and
bounds the error of replacing the varying Jacobian by its initial value.
\begin{lemma}[Finite-segment Jacobian error]
\label{lem:finite_segment_jacobian}
Under the Newton-segment hypotheses of
Lemma~\ref{lem:sharper_parameters}, define
\[
 x_s:=x+s\Delta,\qquad
 \psi_s:=\frac{\d}{\d s}\log s(x_s),\qquad
 \psi:=\int_0^1\psi_s\,\d s,
\]
and let $B_s$ be the Jacobian of $\log g(x_s)$ with respect to the
logarithmic row scaling.  With $B:=B_0$ and
$u:=\log g(x_1)-\log g(x_0)$, one has
\[
 N_w(\psi)\leq(1+8R)c_\gamma\delta,
 \qquad
 N_w(u-B\psi)\leq a_1\delta^2
\]
for an absolute constant $a_1$.
\end{lemma}
\begin{proof}
Keep the notation $h=\sqrt{\Phi''(x)}\Delta$ from the proof of
Lemma~\ref{lem:sharper_parameters}.  Coordinatewise self-concordance and
Lemma~\ref{lem:ls_path_facts} give, for every $s\in[0,1]$,
\[
 |\psi_{s,i}|\leq\frac{|h_i|}{1-s|h_i|},
 \qquad
 N_w(\psi_s)\leq\frac{N_w(h)}{1-\|h\|_\infty}
 \leq(1+8R)c_\gamma\delta.
\]
Integration proves the first claim and also
$\|\log s(x_s)-\log s(x)\|_\infty\leq3\delta$ after
decreasing $r_0$.  Target stability then gives
$\|\log g(x_s)-\log g(x)\|_\infty\leq12\delta$.
Apply Lemma~\ref{lem:jacobian_stability} to the endpoints $0$ and $s$.
The exact $g(x)$-mixed norm and the maintained $w$-mixed norm differ by at
most the absolute factor $e^{K/2}$, so, uniformly in $s$,
\[
 \|B_s-B_0\|_{w+\infty\to w+\infty}\leq A_J\delta
\]
for an absolute $A_J$.  The differentiability conclusion of
Lemma~\ref{lem:regularized_weight_wellposed} gives
\[
 u=\int_0^1B_s\psi_s\,\d s.
\]
Therefore
\[
 N_w(u-B_0\psi)
 \leq\int_0^1
 \|B_s-B_0\|_{w+\infty\to w+\infty}N_w(\psi_s)\,\d s
 \leq A_J(1+8R)c_\gamma\delta^2
 \leq a_1\delta^2,
\]
where the first step follows from the integral representations of $u$ and
$\psi$, the triangle inequality, and the definition of the operator norm,
the second step follows from the uniform bounds on $B_s-B_0$ and $\psi_s$
above, and the third step follows from $R\leq r_0/c_k$ and
$c_\gamma\leq1+1/(128c_k)$ by choosing a sufficiently large absolute
constant $a_1$.
\end{proof}
\subsection{Large-radius defect geometry}
\label{sec:defect_adaptive}
\label{subsec:large_radius_refinement}
\label{subsec:large_radius_defect_geometry}

This section combines the sharper geometry and stochastic residual game above. The objective is to use the
negative term discarded in Lemma~\ref{lem:robust_consistency} to pay for the
part of the Jacobian--vector product that is costly to estimate.  Throughout
this section, $N(z):=\|z\|_\infty+C\|z\|_w$, where $w$ is the maintained weight
at the beginning of the centering step.  All conditional expectations are
with respect to the algorithmic filtration immediately before the fresh
sketches named in the statement are drawn.

Let $g$ be the exact regularized Lewis vector at the beginning of a centering
step and let $B$ be the Jacobian of $\log g$ with respect to the logarithmic row
scaling.  For a direction $h$, define
\[
 r:=\|h\|_g,\qquad b:=Bh,\qquad s:=\|b\|_g,\qquad
 d(h)^2:=p^2r^2-s^2.
\]
Lemma~\ref{lem:jacobian_defect}-\hyperref[item:jacobian_defect_a]{(a)} makes $d(h)$ real.

The basic contraction estimate discards the slack measured by $d(h)$, but the
predictor must retain that slack to pay for its approximation error.  The next
lemma makes the negative defect term explicit.
\begin{lemma}[Retained Jacobian defect]
\label{lem:retained_jacobian_defect}
Suppose $\|\log w-\log g\|_\infty\leq H\leq1/4$.  Then
\[
 N(Bh)
 \leq p(1+e^{H/2}(\tfrac{2}{pC}+L_H)^2)N(h)
 -\frac{C^2}{4pe^{H/2}}\frac{d(h)^2}{N(h)},
 \]
where $L_H=e^{3H/2}-1$.  The last quotient is interpreted as zero when
$h=0$.
\end{lemma}
\begin{proof}
The proof of Lemma~\ref{lem:robust_consistency}, before its last maximization,
gives
\[
 N(Bh)-pN(h)\leq A_Hd+C(s-pr),
 \qquad A_H:=\frac2p+CL_H,
 \]
where $d=d(h)$.  If $r>0$, then
\[
 s-pr=-\frac{d^2}{pr+s}\leq-\frac{d^2}{2pr}.
\]
Young's inequality gives
\[
 A_Hd\leq\frac{A_H^2pr}{C}+\frac{Cd^2}{4pr},
\]
and hence
\[
 N(Bh)\leq pN(h)+\frac{A_H^2pr}{C}-\frac{Cd^2}{4pr}.
\]
Since $N(h)\geq C\|h\|_w\geq Ce^{-H/2}r$, we have
$pr\leq pe^{H/2}N(h)/C$.  Applying this bound to both terms in the last
display proves the claim.  The case $r=0$ is immediate.
\end{proof}

We enlarge the centrality radius while retaining logarithmic cost per
step.  Keep
\[
 E=\frac{K}{A_E L_{\rm ch}}=\Theta(q^{-1/2}\ell^{-1})
\]
for a sufficiently large absolute $A_E$, and set
\[
 R:=c_RCE=\Theta(\ell^{-1})
\]
for a sufficiently small absolute $c_R>0$.  These choices give
\[
 R=O(1/c_k),\qquad E=O(1/c_k),\qquad R^2=O(E).
\]
The last relation follows from $C^2E=O(\sqrt q/\ell)=O(1)$.

The next lemma extends the retained-movement estimate to a centrality radius
$R$ that may be larger than the score accuracy $E$.  It uses only the three
scale inequalities above, not the stronger relation $R=\Theta(E)$.

\begin{lemma}[Large-radius retained movement]
\label{lem:large_radius_retained_movement}
There are absolute constants $r_0,\gamma>0$ such that, with the parameters of
Lemma~\ref{lem:sharper_parameters}, the following holds.  Suppose
$x$ is feasible and lies in the barrier domain, $w>0$,
$\|\log w-\log g(x)\|_\infty\leq K$,
$E\leq r_0/c_k$, $R\leq r_0/c_k$, and
$\delta:=\delta_t(x,w)\leq R$.  Let $x^+$ be the projected Newton step
at fixed $t,w$.  For
\[
 \psi=\log s(x^+)-\log s(x),\qquad
 u=\log g(x^+)-\log g(x),\qquad d=d(\psi),
\]
one has
\[
 N(u)\leq(1-\frac{4}{3c_k})\delta
 -\gamma\frac{C^2d^2}{\delta},
 \qquad d^2\leq\frac{5\delta^2}{C^2}.
\]
When $\delta=0$, the quotient $d^2/\delta$ is defined to be zero.
\end{lemma}
\begin{proof}
If $\delta=0$, the projected residual vanishes, so $x^+=x$ and
$\psi=u=d=0$.  If $N(\psi)=0$, then $s(x^+)=s(x)$, so again
$u=d=0$.  Thus both claims hold in these cases.  Assume henceforth
$\delta>0$ and $N(\psi)>0$.
Lemma~\ref{lem:sharper_parameters} keeps the entire Newton segment in the
band $H\leq2K$.  Its numerical estimate gives
\[
 e^{H/2}(\tfrac{2}{pC}+L_H)^2<\frac1{4c_k}.
\]
Since $p=1-2/c_k$, Lemma~\ref{lem:retained_jacobian_defect}, applied directly
at the base point to $h=\psi$, therefore gives
\[
 N(B\psi)\leq(1-\frac{7}{4c_k})N(\psi)
 -\gamma_1\frac{C^2d^2}{N(\psi)}
\]
for an absolute $\gamma_1>0$.  This derivation uses only the segment band and
the parameter inequalities, not any separate small-radius specialization.
Lemma~\ref{lem:finite_segment_jacobian} gives
\[
 N(\psi)\leq(1+8R)c_\gamma\delta,
 \qquad
 N(u-B\psi)\leq a_1\delta^2
\]
for an absolute $a_1$.  Since
$c_\gamma\leq1+1/(128c_k)$ and $R,\delta\leq r_0/c_k$, the total positive
error in substituting these two displays is at most
\[
 (1/(128c_k)+8r_0/c_k+O(r_0^2/c_k^2))\delta
 +a_1r_0\delta/c_k.
\]
Choose $r_0$ so that this is at most $\delta/(6c_k)$.  Also
$N(\psi)\leq2\delta$, after the same choice, so the negative term is at most
$-(\gamma_1/2)C^2d^2/\delta$.  Decreasing its absolute coefficient gives
the first claim.  Finally metric equivalence gives
\[
 C\|\psi\|_g\leq e^{K}N(\psi)\leq2e^K\delta,
\]
and therefore
$d^2\leq p^2\|\psi\|_g^2\leq5\delta^2/C^2$.
\end{proof}

The retained defect is algorithmically useful only if the Jacobian action can
be estimated at a cost proportional to that defect.  The following lemma
separates the weighted and coordinatewise pieces needed for this purpose.
\begin{lemma}[Large-radius defect split predictor]
\label{lem:large_radius_split_predictor}
Under the hypotheses of Lemma~\ref{lem:large_radius_retained_movement}, define
$\Delta:=\overline\delta_t(x,w)$.  There is a coarse sigma-field
$\mathcal G_k$ and estimates
$p^\infty,p^w$ such that, except with conditional probability $\zeta$,
\[
 \|p^\infty-u\|_\infty
 \leq\epsilon_1(\overline d+\frac{\Delta}{c_k}),
\]
and, conditionally on $\mathcal G_k$,
\[
 \E[\|p^w-u\|_w^2\mid\mathcal G_k]
 \leq\epsilon_1^2(\frac{\overline d^2}{c_k}
 +\frac{\Delta^2}{C^2c_k^2}).
\]
The first estimate and the certificate cost $O(\log(m/\zeta))$ systems,
the second costs $O(c_k)$ systems, and $\epsilon_1$ may be made an
arbitrarily small absolute constant.
Algorithm~\ref{alg:coarse_defect_data} implements these estimates with the
displayed failure budget; all block sizes and copy counts are fixed before
their randomness is drawn.
\end{lemma}
\begin{proof}
Construct the certified surrogate of Lemma~\ref{lem:certified_surrogate}
at the fixed query.  Algorithm~\ref{alg:coarse_defect_data}
uses the explicit failure split and copy counts stated there.  Its Jacobian has the form
$W^{-1/2}(D+K)W^{1/2}$.  Add the exactly computable diagonal term $D$ to
two independent estimators of $K$ from
Lemma~\ref{lem:defect_fixed_point}: a coordinate chain and, drawn only
after the coarse sigma-field is fixed, a weighted chain.  This estimator
construction and the certificate of Lemma~\ref{lem:defect_certificate}
are algebraic statements about the fixed surrogate and have no
small-radius hypothesis.

Write $J^+:=(W^+)^{-1/2}(D^++K^+)(W^+)^{1/2}$ for the surrogate
Jacobian and define
\[
 e_{\rm det}:=J^+\psi-u.
\]
This is precisely the sum of the surrogate-operator perturbation and the
finite-segment remainder.  Lemmas~\ref{lem:certified_surrogate} and
\ref{lem:finite_segment_jacobian}, together with the constant equivalence of
$W^+$, $g$, and the maintained $w$, give
\[
 N(e_{\rm det})\leq a_2(E\delta+\delta^2)
\]
for an absolute $a_2$.  Since $E,R\leq r_0/c_k$,
\[
 N(e_{\rm det})\leq2a_2r_0\delta/c_k.
\]
Choose $r_0\leq\epsilon_1/(8a_2)$.  Since
$\delta\leq\Delta$, this gives
\[
 \|e_{\rm det}\|_\infty
 \leq\frac{\epsilon_1}{4}\frac{\Delta}{c_k},
 \qquad
 \|e_{\rm det}\|_w
 \leq\frac{\epsilon_1}{4C}\frac{\Delta}{c_k}.
\]

Draw the coordinate chain and the certificate first.  If
$\xi^\infty:=(W^+)^{-1/2}(\widehat{Kx}_\infty-K^+x^+)$ denotes the
physical-coordinate stochastic error, then the normalized coordinate
conclusion of Lemma~\ref{lem:defect_fixed_point} is exactly
\[
 \|\xi^\infty\|_\infty\leq\epsilon_\infty d_+.
\]
The explicit construction in Lemma~\ref{lem:defect_certificate} gives
$d_+^2\leq\overline d^2$.  Choose its absolute block size so that
$\epsilon_\infty\leq\epsilon_1/2$.  Hence, on the coarse event,
\[
 \|p^\infty-u\|_\infty
 \leq\|\xi^\infty\|_\infty+\|e_{\rm det}\|_\infty
 \leq\frac{\epsilon_1}{2}\overline d
      +\frac{\epsilon_1}{4}\frac{\Delta}{c_k},
\]
where the first step follows from
$p^\infty-u=\xi^\infty+e_{\rm det}$ and the triangle inequality, and the
second step follows from
$\|\xi^\infty\|_\infty\leq\epsilon_\infty d_+$,
$d_+\leq\overline d$, $\epsilon_\infty\leq\epsilon_1/2$, and the
deterministic bound on $\|e_{\rm det}\|_\infty$ above.
This is stronger than the first displayed conclusion.

Let $\mathcal G_k$ be generated by the past, the fixed surrogate-score
sketch, the coordinate chain, and the certificate.  Draw the weighted chain
only afterward.  Set
$\xi^w:=(W^+)^{-1/2}(\widehat{Kx}_w-K^+x^+)$.

Conditional on
$\mathcal G_k$, we can show
\[
 \E[\|\xi^w\|_w^2\mid\mathcal G_k]
 \leq c_{\rm eq}\frac{a_4c_w^2d_+^2}{c_k}
 \leq\frac{\epsilon_1^2}{8}\frac{\overline d^2}{c_k},
\]
where the first step follows from the conditional Euclidean guarantee of
Lemma~\ref{lem:defect_fixed_point} and the norm equivalence between $W^+$
and $w$, and the second step follows from $d_+^2\leq\overline d^2$ and the
choice $c_{\rm eq}a_4c_w^2\leq\epsilon_1^2/8$ of the absolute accuracy
multiplier $c_w$ in Algorithm~\ref{alg:coarse_defect_data}.
The conditional statement is legitimate because $S_w$ and every query
operator are $\mathcal G_k$-measurable before this independent chain is
drawn.

We can show
\[
 \E[\|p^w-u\|_w^2\mid\mathcal G_k]
 \leq\frac{\epsilon_1^2}{4}\frac{\overline d^2}{c_k}
 +\frac{\epsilon_1^2}{8}\frac{\Delta^2}{C^2c_k^2}
 \leq\epsilon_1^2(\frac{\overline d^2}{c_k}
 +\frac{\Delta^2}{C^2c_k^2}),
\]
where the first step follows from
$p^w-u=\xi^w+e_{\rm det}$, Young's inequality in the form
$(a+b)^2\leq2a^2+2b^2$, the conditional bound on $\xi^w$, and the
deterministic bound on $\|e_{\rm det}\|_w$ above, and the second step follows
from $1/4\leq1$ and $1/8\leq1$.  The coordinate
chain and certificate use $O(\log(m/\zeta))$ systems, and the independent
weighted chain uses $O(c_k)$ systems.  This proves every assertion directly
at the large radius.
\end{proof}

\subsection{Projection onto the defect body}
\label{subsec:defect_body_projection}

The key point is to use the computable proxy
$\Delta_k:=\overline\delta_{t_k}(x_k,w_k)$ and to measure the residual with
a radius linear in the certified defect:
\[
 s_k:=\eta_0(\overline d_k+\frac{\Delta_k}{c_k}),
\]
where $\eta_0>0$ is a sufficiently small absolute constant.  Set
\[
 r_k:=(1-\frac{59}{48c_k})\Delta_k
 -B\frac{C^2\overline d_k^2}{\Delta_k}.
\]
The branch $\Delta_k=0$ is handled before these quantities are formed.

For a fixed $t$ below, write
$\Call{FixedSlice}{t,w,\overline p,l,h,r,C}=(x,g)$ for the exact
active-set calculation in the proof of
Lemma~\ref{lem:computable_defect_projection}: it minimizes the weighted
quadratic on the slice
$|x_i|\leq t$, $l_i\leq x_i\leq h_i$, and
$\|x\|_w\leq(r-t)/C$, and returns the displayed KKT subgradient $g$.
At ties it uses the strict moving-face indicators
$\mathbf1_{\{l_i<x_i=t<h_i\}}$ and
$\mathbf1_{\{l_i<x_i=-t<h_i\}}$.
To make the auxiliary routine explicit, define
$I_i(t):=[\max\{l_i,-t\},\min\{h_i,t\}]$ and $A(t):=(r-t)/C$.
If $y_i:=\operatorname{clip}(\overline p_i,I_i(t))$ satisfies
$\|y\|_w\leq A(t)$, set $x:=y$ and $\lambda:=0$.
Otherwise, sort the values at which
$\overline p_i/(1+\lambda)$ meets an endpoint of $I_i(t)$.
On a region with free set $F$ and clipped endpoint values $b_i$, set
\[
 1+\lambda:=
 (\frac{\sum_{\substack{i=1\\i\in F}}^m w_i\overline p_i^2}
 {A(t)^2-\sum_{\substack{i=1\\i\notin F}}^m w_ib_i^2})^{1/2},
 \qquad
 x_i:=\begin{cases}
 \overline p_i/(1+\lambda),&i\in F,\\
 b_i,&i\notin F.
 \end{cases}
\]
The scan selects the region satisfying its clipping inequalities.
For either branch, define
\[
 \alpha_i^+:=w_i(\overline p_i-(1+\lambda)t)
 \mathbf1_{\{l_i<x_i=t<h_i\}},
 \qquad
 \alpha_i^-:=w_i(-\overline p_i-(1+\lambda)t)
 \mathbf1_{\{l_i<x_i=-t<h_i\}},
\]
and return
$g:=\lambda A(t)/C-\sum_{i=1}^m(\alpha_i^++\alpha_i^-)$.

\begin{algorithm}[!ht]
\caption{Approximate projection onto the defect body}
\label{alg:defect_projection}
\begin{algorithmic}[1]
\Procedure{DefectProjection}{$w,C,r,s,p^\infty,p^w$}
\Comment{Lemma~\ref{lem:computable_defect_projection}}
\State $\overline p\gets\min\{1,r/(C\|p^w\|_w)\}p^w$, with multiplier
$1$ when $p^w=0$.
\State $l_i\gets p_i^\infty-s/64$, $h_i\gets p_i^\infty+s/64$, and
$z_i\gets\operatorname{clip}(0,[l_i,h_i])$ for every $i$.
\If{$\|z\|_\infty+C\|z\|_w>r$}
  \State \Return \textnormal{\textsc{Fail}}.
\EndIf
\State $\tau_{\rm P}\gets\epsilon_{\rm P}s/C$,
$W_1\gets\sum_{i=1}^mw_i$, and
$\Lambda_{\rm P}\gets32(r/(C\tau_{\rm P}))^2$.
\State $G_{\rm P}\gets\Lambda_{\rm P}r/C^2+
\sqrt{W_1}\|\overline p\|_w+(1+\Lambda_{\rm P})rW_1$.
\State $J\gets\lceil\log_2(2+8G_{\rm P}r/\tau_{\rm P}^2)\rceil$,
$t_-\gets\|z\|_\infty$, $t_+\gets r-C\|z\|_w$, and
$t_{\rm edge}\gets r-C\sqrt{\|z\|_w^2+\tau_{\rm P}^2/16}$.
\If{$t_{\rm edge}\leq t_-$}
  \State \Return $z$.
\EndIf
\State $(x_{\rm edge},g_{\rm edge})\gets
\Call{FixedSlice}{t_{\rm edge},w,\overline p,l,h,r,C}$.
\If{$g_{\rm edge}<0$}
  \State \Return $z$.
\ElsIf{$g_{\rm edge}=0$}
  \State \Return $x_{\rm edge}$.
\EndIf
\State $t_+\gets t_{\rm edge}$ and $v_{\rm best}\gets x_{\rm edge}$.
\For{$j=1,\ldots,J$}
  \State $t\gets(t_-+t_+)/2$ and
  $(x,g)\gets\Call{FixedSlice}{t,w,\overline p,l,h,r,C}$.
  \State Retain $x$ if it improves
  $\frac12\|x-\overline p\|_w^2$.
  \If{$g>0$}
    \State $t_+\gets t$.
  \ElsIf{$g<0$}
    \State $t_-\gets t$.
  \Else
    \State \Return $x$.
  \EndIf
\EndFor
\State \Return $v_{\rm best}$.
\EndProcedure
\end{algorithmic}
\end{algorithm}

The split predictor produces two approximate descriptions of the same motion,
which must be reconciled without an implicit continuous optimization oracle.
The next lemma certifies the finite projection procedure above.
\begin{lemma}[Computable defect projection]
\label{lem:computable_defect_projection}
Let $w>0$, $C,r,s>0$, and $p^\infty,p^w\in\R^m$.  Define
\[
 \mathcal C(p^\infty;r,s)
 :=\{v:\|v\|_\infty+C\|v\|_w\leq r,\
          \|v-p^\infty\|_\infty\leq s/64\}.
\]
Fix an absolute $\epsilon_{\rm P}>0$.  If this set is nonempty,
Algorithm~\ref{alg:defect_projection} returns a feasible $\widetilde v$ such that
\[
 \|\widetilde v-\Pi_{\mathcal C}^w(\overline p)\|_w
 \leq\tau_{\rm P}:=\epsilon_{\rm P}s/C,
 \qquad
 \overline p:=\min\{1,r/(C\|p^w\|_w)\}p^w,
\]
where $\Pi_{\mathcal C}^w$ is exact metric projection in $\|\cdot\|_w$.
When $p^w=0$, the minimum factor is interpreted as $1$.
If there is a scale $\Delta>0$ such that $r\leq\Delta$ and
$s\geq\eta_0\Delta/c_k$, the routine uses no linear systems.  In the
tracked-weight regime used below, where
$\sum_{i=1}^mw_i\leq2e^Kn$, $C^2=O(q)$, and $q\leq L_*$, it has
\[
 O(mJ\log m)\ \text{scalar work},
 \qquad O(J\log m)\ \text{arithmetic depth},
 \qquad J=O(L_*+\log(2+c_k/(\eta_0\epsilon_{\rm P}))).
\]
It also detects an empty set, in which case the calling algorithm returns
\textnormal{\textsc{Fail}}.
\end{lemma}
\begin{proof}
Write $l_i=p_i^\infty-s/64$, $h_i=p_i^\infty+s/64$, and
$z_i=\operatorname{clip}(0,[l_i,h_i])$.  The vector $z$ has the smallest
absolute value in every coordinate among all vectors in the box.  Hence the
set is nonempty exactly when
\[
 \|z\|_\infty+C\|z\|_w\leq r.
\]
For
$t\in[\|z\|_\infty,r-C\|z\|_w]$, define
\[
 I_i(t):=[\max\{l_i,-t\},\min\{h_i,t\}],
 \qquad A(t):=(r-t)/C,
\]
and let $\varphi(t)$ be the minimum of
$\frac12\|x-\overline p\|_w^2$ over
$x_i\in I_i(t)$ and $\|x\|_w\leq A(t)$.  This is the partial minimum of a
jointly convex problem, so $\varphi$ is convex.

We first make the fixed-$t$ computation explicit.  If
$y_i=\operatorname{clip}(\overline p_i,I_i(t))$ satisfies
$\|y\|_w\leq A(t)$, it is the exact minimizer and $\lambda=0$.
Otherwise the ball is active and
\[
 x_i(\lambda,t)
 =\operatorname{clip}(\overline p_i/(1+\lambda),I_i(t)).
\]
The at most $2m$ nonnegative values at which
$\overline p_i/(1+\lambda)$ meets an interval endpoint partition
$\lambda\geq0$ into regions.  In a region, the clipped coordinates are
fixed endpoints $b_i$ and the others form a free set $F$; the equation
$\|x\|_w=A(t)$ has exactly the formula displayed before the algorithm.
Scanning the sorted regions finds the one satisfying its clipping
inequalities.  This is an exact active-set computation, so its running time
does not depend on the magnitude of $\lambda$.

At its KKT tuple, assign a multiplier at a tie between a moving face and a
fixed box face to the fixed face.  Thus a moving upper face is charged only
when $l_i<x_i=t<h_i$, and a moving lower face only when
$l_i<x_i=-t<h_i$.  The two formulas for $\alpha_i^\pm$ displayed before
the algorithm are then
the exact multipliers of the moving faces, including every degenerate tie,
and
\[
 g(t)=\frac{\lambda A(t)}C-\sum_{i=1}^m(\alpha_i^++\alpha_i^-)
 \in\partial\varphi(t).
\]
Thus every non-endpoint branch uses an exact subgradient; there are no
undefined multiplier intervals or approximate lower certificates.

It remains to justify the quantitative edge and iteration rules.  For
any $x$ in the original coordinate box,
$x_i z_i\geq z_i^2$, and hence
\[
 \|x-z\|_w^2\leq\|x\|_w^2-\|z\|_w^2.
\]
Consequently, if
$A(t)^2-\|z\|_w^2\leq\tau_{\rm P}^2/16$, every feasible point at every
parameter $u\geq t$ is within $\tau_{\rm P}/4$ of $z$.

The algorithm uses this fact without mixing two incompatible certificates.
If $t_{\rm edge}\leq t_-$, the preceding estimate applies to the whole
parameter interval and returning $z$ is valid.  Otherwise compute the exact
subgradient $g_{\rm edge}$ of the convex function $\varphi$ at
$t_{\rm edge}$.  If $g_{\rm edge}<0$, every minimizer of $\varphi$ lies to
the right of $t_{\rm edge}$ and is therefore within $\tau_{\rm P}/4$ of
$z$; if $g_{\rm edge}=0$, the fixed-slice point is the exact projection.
Only when $g_{\rm edge}>0$ does the algorithm retain the left interval.
That interval contains a minimizer and every subsequent cut is certified by
an exact subgradient.  In particular, a point certified only by the right
edge is never later overwritten by an objective comparison.

On every other branch, nonexpansiveness of coordinate clipping gives
\[
 \|x(\lambda,t)-z\|_w
 \leq\|\overline p\|_w/(1+\lambda).
\]
Since $\|\overline p\|_w\leq r/C$ and $A(t)\leq r/C$,
\[
 A(t)^2-\|z\|_w^2
 \leq2A(t)\|x(\lambda,t)-z\|_w
 \leq\frac{2(r/C)^2}{1+\lambda}.
\]
On the retained interval the reverse edge inequality therefore implies
$1+\lambda\leq\Lambda_{\rm P}$.  With
$W_1=\sum_{i=1}^mw_i$, the explicit moving-face formulas give
\[
 \sum_{i=1}^m(\alpha_i^++\alpha_i^-)
 \leq\sqrt{W_1}\|\overline p\|_w+(1+\lambda)tW_1.
\]
Since $t\leq r$, every exact subgradient used outside the endpoint branch has
magnitude at most $G_{\rm P}$.

Exact-subgradient bisection on the left interval preserves an interval
containing a minimizer.  After $J$ iterations the retained interval has length
at most $r2^{-J}$, and the best evaluated exact fixed-$t$ point has objective
gap at most $2G_{\rm P}r2^{-J}\leq\tau_{\rm P}^2/4$.
The projection variational inequality then gives
$\|\widetilde v-\Pi_{\mathcal C}^w(\overline p)\|_w\leq\tau_{\rm P}$.
Every returned point is feasible by construction.

Each outer iteration sorts at most $2m$ breakpoints and scans them once.
Moreover $r/s\leq c_k/\eta_0$ under the stated scale hypothesis,
$\tau_{\rm P}=\epsilon_{\rm P}s/C$, and in the application
$\sum_{i=1}^mw_i\leq2e^Kn$, $C^2=O(q)$, and $q\leq L_*$.  Taking logarithms in the
explicit formula for $J$ gives the stated bound.  Hence the scalar work and
depth are $O(mJ\log m)$ and $O(J\log m)$, respectively.  These costs are
charged by the global scalar
counter introduced below; they are not
attributed to one linear-system solve.
\end{proof}

We apply the computable projection to the two predictor estimates and the
defect certificate.  The following lemma verifies that the resulting motion
fits the residual game while preserving the centrality budget.
\begin{lemma}[Defect projection]
\label{lem:linear_defect_projection}
There are choices of the absolute constants in this construction for which,
on the coarse event of Lemma~\ref{lem:large_radius_split_predictor}, with
$\delta_k:=\delta_{t_k}(x_k,w_k)$ and
$\Delta_k:=\overline\delta_{t_k}(x_k,w_k)$,
\[
 r_k\geq N(u_k),\qquad
 r_k+1.1s_k\leq(1-\frac1{c_k})\delta_k,\qquad
 s_k\leq E/100.
\]
Let
\[
 \mathcal C_k:=\{v:N(v)\leq r_k,\
 \|v-p^\infty\|_\infty\leq s_k/64\}
\]
and let $\widehat u_k$ be the feasible output of
Lemma~\ref{lem:computable_defect_projection}, applied with scale $\Delta_k$
(indeed, $r_k\leq\Delta_k$ and $s_k\geq\eta_0\Delta_k/c_k$), with
$\tau_{\rm P}=\epsilon_{\rm P}s_k/C$.  Then, with
\[
 e_k:=u_k-\widehat u_k,\qquad
 b_k:=\E[e_k\mid\mathcal G_k],\qquad Z_k:=e_k-b_k,
\]
one has
\[
 N(b_k)\leq s_k/8,\qquad
 \|Z_k\|_\infty\leq s_k/2,\qquad
 \E[N(Z_k)^2\mid\mathcal G_k]\leq\nu^2s_k^2,
\]
with $\nu=1/5$, as required by Lemma~\ref{lem:residual_chasing}.
\end{lemma}
\begin{proof}
The one-sided certificate gives
\[
 d_k^2\leq\overline d_k^2
 \leq A_d(d_k^2+\frac{\Delta_k^2}{C^2c_k}).
\]
Choose $B$ so that
$BA_d\leq\min\{1/12,\gamma/2\}$. The certificate then gives
\[
 B\frac{C^2\overline d_k^2}{\Delta_k}
 \leq\frac{\gamma}{2}\frac{C^2d_k^2}{\Delta_k}
 +\frac{\Delta_k}{12c_k}.
\]
Lemma~\ref{lem:large_radius_retained_movement}, together with
$\delta_k\leq\Delta_k$, gives
\[
 N(u_k)\leq(1-\frac{4}{3c_k})\Delta_k
 -\gamma\frac{C^2d_k^2}{\Delta_k},
\]
so the definition with coefficient $59/48$ gives $r_k\geq N(u_k)$.

Define $a_0:=1.1\eta_0$.  Young's inequality gives
\[
 a_0\overline d_k
 \leq B\frac{C^2\overline d_k^2}{\Delta_k}
 +\frac{a_0^2}{4BC^2}\Delta_k.
\]
Since $C^2=128\overline c_sc_k\geq128c_k$, choosing $\eta_0$ so that
\[
 a_0+\frac{a_0^2}{512B}\leq\frac16
\]
yields
\[
 r_k+1.1s_k\leq(1-\frac{17}{16c_k})\Delta_k.
\]
The explicit bound
$\overline c_\gamma=1+1/(128c_k)$ gives
\[
 (1-\frac{17}{16c_k})\overline c_\gamma
 \leq1-\frac1{c_k}.
\]
Therefore the stronger, directly testable inequality
\[
 r_k+1.1s_k\leq(1-\frac1{c_k})
 \frac{\Delta_k}{\overline c_\gamma}
 \leq(1-\frac1{c_k})\delta_k
\]
holds, because Eq.~\eqref{eq:computable_centrality} gives
$\Delta_k\leq c_\gamma\delta_k\leq\overline c_\gamma\delta_k$.
Thus the body is nonempty and the literal failure test in
Algorithm~\ref{alg:fresh_predicted_center} is not triggered.  The certificate and
$d_k^2\leq5\delta_k^2/C^2\leq5\Delta_k^2/C^2$ also imply
$\overline d_k=O(\Delta_k/C)$.  Since
$\Delta_k\leq c_\gamma R=c_R c_\gamma CE$ and $C/c_k=O(1)$,
\[
 s_k\leq\eta_0c_RE(O(1)+C/c_k).
\]
Decreasing $c_R\eta_0$ proves $s_k\leq E/100$.

Next choose the coordinate sample constant so that
\[
 \|p^\infty-u_k\|_\infty
 \leq\epsilon_1(\overline d_k+\Delta_k/c_k)\leq s_k/64.
\]
Thus $u_k\in\mathcal C_k$.  Let $\overline p^w$ be the radial preprocessing
of Lemma~\ref{lem:computable_defect_projection}, and let $\pi_k$ be its exact
metric projection onto $\mathcal C_k$.  Because the radial ball contains
$u_k$, the radial projection and the projection variational inequality give
\[
 \|\pi_k-u_k\|_{w_k}
 \leq\|\overline p^w-u_k\|_{w_k}
 \leq\|p^w-u_k\|_{w_k}.
\]
The approximation guarantee and feasibility give
\[
 \|\widehat u_k-\pi_k\|_{w_k}\leq\epsilon_{\rm P}s_k/C,
 \qquad
 \|\widehat u_k-u_k\|_\infty\leq s_k/32.
\]
Moreover $C^2/c_k\leq256$ and
$s_k^2\geq\eta_0^2(\overline d_k^2+\delta_k^2/c_k^2)$, so
Lemma~\ref{lem:large_radius_split_predictor} gives
\[
 \E[C^2\|p^w-u_k\|_{w_k}^2\mid\mathcal G_k]
 \leq256\epsilon_1^2s_k^2/\eta_0^2.
\]
Define
\[
 \alpha_{\rm P}:=16\epsilon_1/\eta_0+\epsilon_{\rm P}
\]
and choose the two absolute accuracy constants so that
$\alpha_{\rm P}\leq1/16$.  Minkowski's inequality and the preceding displays
give
\[
 C(\E[\|e_k\|_{w_k}^2\mid\mathcal G_k])^{1/2}
 \leq\alpha_{\rm P}s_k.
\]
Conditional Jensen and the coordinate bound imply
\[
 N(b_k)\leq s_k/32+\alpha_{\rm P}s_k<s_k/8,
 \qquad \|Z_k\|_\infty\leq s_k/16.
\]
Finally conditional centering can only decrease the second moment, so
\[
 \E[N(Z_k)^2\mid\mathcal G_k]
 \leq2s_k^2/16^2+2\alpha_{\rm P}^2s_k^2
 \leq s_k^2/64<s_k^2/25.
\]
This proves the claims with $\nu=1/5$.
\end{proof}

\section{Initialization, path composition, and defect certification}
\label{sec:initialization_path_certification}

Section~\ref{subsec:warm_start_initialization} constructs the predicted warm
start and the fresh fixed-point initialization needed at restarts.
Section~\ref{subsec:path_composition_outer_reduction} composes centering and
path-parameter updates and converts the terminal barrier state into an LP
solution.  Section~\ref{subsec:defect_estimation_certification} develops the
adaptive residual game and the randomized certificates used to measure the
retained Jacobian defect.

\subsection{Warm starts and fixed-point initialization}
\label{subsec:warm_start_initialization}

The first leverage-score query after a large-radius step has drift $O(R)$, so
recomputing it to accuracy $E$ would cost $(R/E)^2$.  The next lemma avoids
that loss by advancing the auxiliary weight with the projected predictor.

\begin{lemma}[Large-radius warm start and query drift]
\label{lem:large_radius_warm_start}
Suppose the maintained auxiliary weight is $E/2$-accurate before a centering
step and advance its logarithm by the projected increment $\widehat u_k$ of
Lemma~\ref{lem:linear_defect_projection}.  Then the advanced vector is within
$E$ of the new regularized Lewis vector.  An absolute number of rounds from
Lemma~\ref{lem:linf_contraction} restores $E/2$ accuracy.  The first query
transition has logarithmic row drift $O(R)$, and every subsequent query in
that refinement has drift $O(E)$.
\end{lemma}
\begin{proof}
The target moves by $u_k$, while
$\|u_k-\widehat u_k\|_\infty\leq s_k/32\leq E/3200$.
Thus the advanced vector has error at most $E/2+E/3200<E$. The final part of
Lemma~\ref{lem:linf_contraction} restores $E/2$ accuracy in an absolute number
of rounds. Across the centering step the barrier scaling changes by
$e^{\pm O(R)}$, and $\|\widehat u_k\|_\infty=O(R)$, proving the first drift
bound.  During refinement both the current and next auxiliary vectors are
$O(E)$ from the same target, while the barrier point is fixed, proving the
second.
\end{proof}

Fix absolute constants $A_{\rm q}\geq1$ and $c_{\rm anc}>0$, with
$A_{\rm q}$ large enough to dominate the companion-error constants in
Lemmas~\ref{lem:large_radius_warm_start} and
\ref{lem:certified_surrogate}, and then $c_{\rm anc}$ sufficiently small
for every score consumer below.  The fixed-point routine uses either fresh
score sketches or exact score diagonals.  Its final argument is a mode flag
$\mathsf{auto}$ or $\mathsf{exact}$; an omitted flag means
$\mathsf{auto}$.  Exact mode forces every score query onto the deterministic
exact branch, even when the Gaussian row count would be smaller.  Its input
also includes a certified initial logarithmic error, which makes the
cold-start and predicted warm-start costs explicit.
The outer solver uses the hybrid score routine of
Algorithm~\ref{alg:candidate_hybrid_score} throughout.  Its fresh-state
handoff stores the companion of the actual final query, as made explicit
below.

\begin{algorithm}[!ht]
\caption{Fixed-point computation with hybrid scores}
\label{alg:fresh_fixed_point}
\begin{algorithmic}[1]
\Procedure{FixedPoint}{$x,z^{(0)},D,\epsilon_{\rm fp},\zeta,L_{\rm conf},
\mathcal S,\mathsf{mode}$}
\Comment{Lemma~\ref{lem:floor_homotopy_initialization}}
\State $D_x\gets\Phi''(x)^{-1/2}$.
\State $(J,(k_j)_{j=0}^{J-1})\gets
\Call{FixedPointSchedule}{D,\epsilon_{\rm fp},\zeta}$.
\Comment{Algorithm~\ref{alg:fixed_point_schedule}}
\If{$\mathcal S=\varnothing$}
  \State $S_{\rm sketch}\gets\sum_{j=0}^{J-1}k_j$.
  \State If $\mathsf{mode}=\mathsf{exact}$, use exact score diagonals;
  otherwise use Gaussian mode if $S_{\rm sketch}\leq nJ$ and exact score
  diagonals if not.
\EndIf
\For{$j=0,\ldots,J-1$}
  \State $D_j\gets\operatorname{Diag}(z^{(j)})^{\vartheta/2}D_x$
  and $M_j\gets D_jA$.
  \If{$\mathcal S=\varnothing$}
    \If{the global mode is Gaussian}
      \State Draw a fresh $G_j\in\R^{k_j\times m}$ with independent
      $N(0,1/k_j)$ entries after $D_j$ is fixed.
      \State $\widetilde\sigma_i^{(j)}
      \gets\|G_jP(D_j)e_i\|_2^2$ for every $i\in[m]$.
    \Else
      \State $\widetilde\sigma^{(j)}\gets\sigma(D_j)$ exactly.
    \EndIf
  \Else
    \State $(\widetilde\sigma^{(j)},\mathcal S)
    \gets\Call{HybridScore}{D_j,z^{(j)},L_{\rm conf},\mathcal S}$.
    \Comment{Algorithm~\ref{alg:candidate_hybrid_score}}
  \EndIf
  \State $\widetilde\tau^{(j)}\gets\widetilde\sigma^{(j)}+v$.
  \For{$i\in[m]$}
    \State $z_i^{(j+1)}\gets(z_i^{(j)})^{1-p/2}
    (\widetilde\tau_i^{(j)})^{p/2}$.
  \EndFor
\EndFor
\If{$\mathcal S=\varnothing$}
  \State $\mathcal S\gets
  (D_{J-1},z^{(J-1)},\widetilde\sigma^{(J-1)},0)$.
  \Comment{Lemma~\ref{lem:candidate_raw_initialization}}
\EndIf
\State \Return $(z^{(J)},\mathcal S)$.
\EndProcedure
\end{algorithmic}
\end{algorithm}

Algorithm~\ref{alg:fixed_point_schedule} fixes the round count and sample sizes before any
score samples are drawn.

\begin{algorithm}[!ht]
\caption{Deterministic schedule for fixed-point computation}
\label{alg:fixed_point_schedule}
\begin{algorithmic}[1]
\Procedure{FixedPointSchedule}{$D,\epsilon_{\rm fp},\zeta$}
\State $\varrho\gets1-p/2$, $e_0\gets D$, and $j\gets0$.
\While{$e_j>\epsilon_{\rm fp}$ or $j=0$}
  \If{$e_j\geq1$}
    \State $\delta_j^{\rm map}\gets(1-\varrho)/8$.
  \Else
    \State $\delta_j^{\rm map}\gets(1-\varrho)e_j/4$.
  \EndIf
  \State $e_{j+1}\gets\varrho e_j+\delta_j^{\rm map}$ and $j\gets j+1$.
\EndWhile
\State $J\gets j$.
\For{$j=0,\ldots,J-1$}
  \State $\theta_j\gets\min\{1/8,\delta_j^{\rm map}/(2p)\}$.
  \If{$j=J-1$}
    \State $\theta_j\gets\min\{\theta_j,c_{\rm anc}E/16\}$.
  \EndIf
  \State $k_j\gets
  \lceil C_{\rm fp}\theta_j^{-2}\log(2mJ/\zeta)\rceil$.
\EndFor
\State \Return $(J,(k_j)_{j=0}^{J-1})$.
\EndProcedure
\end{algorithmic}
\end{algorithm}

The warm-start argument handles continuation between nearby queries, but a
new epoch still needs an absolute fixed-point construction.  The next lemma
proves both modes of \textsc{FixedPoint} and its initialization cost.
\begin{lemma}[Fixed-point computation and initialization]
\label{lem:floor_homotopy_initialization}
Suppose
$\|\log z^{(0)}-\log g(x)\|_\infty\leq D$ and
$0<\epsilon_{\rm fp}\leq\min\{D,1/4\}$.  In fresh mode assume additionally
$\epsilon_{\rm fp}\leq E/4$.  If $\mathcal S=\varnothing$, then with
probability at least $1-\zeta$,
Algorithm~\ref{alg:fresh_fixed_point} returns $(z^{(J)},\mathcal S)$ with
$\|\log z^{(J)}-\log g(x)\|_\infty\leq\epsilon_{\rm fp}$.
The returned hybrid score state is
$\mathcal S=(D_{J-1},z^{(J-1)},\widetilde\sigma^{(J-1)},0)$ and satisfies
the raw-relative guarantee of
Lemma~\ref{lem:candidate_raw_initialization}; in particular, the companion
stored with $D_{J-1}$ is $z^{(J-1)}$, not $z^{(J)}$.
Here
$J=O(\log(2+D/\epsilon_{\rm fp}))$, the permitted-system cost is
\[
 O(\min\{(J+\epsilon_{\rm fp}^{-2})
 \log(m(1+J)/\zeta),\,nJ\}),
\]
and the linear-system depth is $O(J)$.  In hybrid maintained mode, if
$D\leq E$, $\epsilon_{\rm fp}=E/2$, and every queried score obeys
\[
 |\widetilde\sigma_i^{(j)}-\sigma_i(D_j)|
 \leq c_{\rm anc}E\sigma_i(D_j),
\]
then $J=O(1)$, the same output guarantee holds, and the query sequence
contributes $O(E)$ effective drift and $O(1)$ transitions to the score
process.
In $\mathsf{exact}$ mode the fresh-mode conclusions hold deterministically,
the exact branch costs $O(nJ)$ permitted systems, and the score vector stored
in the returned state is exact.
\end{lemma}
\begin{proof}
Let $e_j$ be a deterministic upper bound on the error before round $j$, with
$e_0=D$, and define $\varrho=1-p/2<1$.  Algorithm~\ref{alg:fixed_point_schedule} makes the schedule
literal.  In the coarse regime $e_j\geq1$, it requests map error
$(1-\varrho)/8$.  In the fine regime it requests map error
\[
 \delta_j^{\rm map}=\frac{1-\varrho}{4}e_j.
\]
Lemma~\ref{lem:linf_contraction} then gives
\[
 e_{j+1}\leq\varrho e_j+\delta_j^{\rm map}
 \leq(\varrho+\frac{1-\varrho}{4})e_j
\]
in the fine regime, while the coarse regime also contracts after changing an
absolute constant.  The first index $J$ with $e_J\leq\epsilon_{\rm fp}$ is
therefore deterministic and satisfies
$J=O(\log(2+D/\epsilon_{\rm fp}))$.
The last pre-update bound is
$e_{J-1}\leq\epsilon_{\rm fp}/
(\varrho+(1-\varrho)/4)=O(\epsilon_{\rm fp})$ in the fine regime.
Thus, in fresh mode, $e_{J-1}\leq A_{\rm q}E$ for an absolute $A_{\rm q}$.

At round $j$, condition on all preceding iterates and use the fresh matrix
$G_j$.  For every nonzero score,
\[
 \frac{\widetilde\sigma_i^{(j)}}{\sigma_i(D_j)}
 \stackrel{d}{=}\frac{\chi_{k_j}^2}{k_j}.
\]
For $0<\theta_j\leq1$, the standard chi-square bound and the definition of
$k_j$ give, after fixing one sufficiently large absolute $C_{\rm fp}$,
\[
 \Pr[\exists i:\ |\widetilde\sigma_i^{(j)}-\sigma_i(D_j)|
 >\theta_j\sigma_i(D_j)\mid D_j]\leq\frac\zeta J.
\]
Zero-score coordinates are returned exactly.  A tower argument and a union
bound over the adaptive rounds show that all these events hold with
probability at least $1-\zeta$.  Since the floor is nonnegative, relative
accuracy for the leverage scores gives the same or better additive accuracy
for the regularized score.  The algorithm's definition of $\theta_j$ gives
\[
 |\log(\widetilde\sigma_i^{(j)}+v_i)
       -\log(\sigma_i(D_j)+v_i)|
 \leq2\theta_j\leq\delta_j^{\rm map}/p.
\]
After multiplication by $p/2$, the logarithmic fixed-point-map error is at
most $\delta_j^{\rm map}/2$, and hence is certainly within the recurrence
budget above.  The explicit row count satisfies
\[
 k_j=O((\delta_j^{\rm map})^{-2}\log(2mJ/\zeta)),
\]
with the final
raw-relative restriction changing only its absolute constant.  The requested
accuracies decrease geometrically in the fine regime, so those squared
sample counts form a geometric series dominated by the final fine round.
Each coarse round uses only
$O(\log(2mJ/\zeta))$ rows, but there can be $J$ such rounds.  Thus the
Gaussian branch uses
\[
 O((J+\epsilon_{\rm fp}^{-2})
 \log(2m(1+J)/\zeta))
\]
systems.  Alternatively,
compute the complete score vector in every round using $O(n)$ systems, for a
total of $O(nJ)$.  Concretely, the exact branch solves the
$n$ unit right-hand sides of $A^\top D_j^2A$, forms
$D_jA(A^\top D_j^2A)^{-1}$ by $n$ sparse matrix products, and takes the
rowwise inner products with $D_jA$.  It therefore uses $O(n)$ permitted
systems and $O(n\nnz(A)+mn)$ scalar operations per round, which are charged
to the global scalar counter.
Algorithm~\ref{alg:fresh_fixed_point} selects the cheaper branch, and all
systems inside one round are parallel.  In the last round
$\theta_{J-1}\leq c_{\rm anc}E/16$.  The same chi-square event used for that
round therefore gives the raw-relative handoff, with no additional failure
event or asymptotic cost, and the algorithm stores its actual query companion
$z^{(J-1)}$.

In hybrid maintained mode, the displayed raw-relative guarantee makes every
inexact-map error at most $cE$ for an arbitrarily small absolute $c$.
Starting from $D\leq E$, the same contraction gives error $E/2$ after an
absolute number of rounds.  Consecutive companions are all $O(E)$-accurate
to the same fixed point, so their row scalings have total effective drift
$O(E)$.
\end{proof}

\subsection{Path composition and outer reduction}
\label{subsec:path_composition_outer_reduction}

The preceding ingredients now supply one certified centering move and one
controlled change of the path parameter.  The next lemma composes them and
shows that all centrality and weight-tracking invariants persist.
\begin{lemma}[Centrality and path-update composition]
\label{lem:centrality_path_composition}
Suppose
\[
 \delta:=\delta_t(x,w)\leq R\leq\frac1{160c_k},
 \qquad
 \|\log w-\log g(x)\|_\infty\leq K.
\]
Let $x^+$ be the Lee--Sidford Newton step at fixed $t,w$, and suppose
\[
 z:=\|\log w^+-\log w\|_{w+\infty}
 \leq(1-\tfrac1{c_k})\delta,
 \qquad
 \|\log w^+-\log g(x^+)\|_\infty\leq K.
\]
Then
\[
 \delta_t(x^+,w^+)\leq(1-\tfrac1{4c_k})\delta.
\]
Moreover, define
\[
 \beta:=1+C\sqrt{2\overline c_1},
 \qquad \alpha:=\frac{R}{16c_k\beta}.
\]
For every $\xi$ with $|\xi|\leq\alpha$ and $1+\xi>0$,
\[
 \delta_{(1+\xi)t}(x^+,w^+)\leq R.
\]
\end{lemma}
\begin{proof}
The parameter choice gives $K\leq1/128$, hence
$4g(x)/5\leq w\leq5g(x)/4$.  Lemma 15 of \cite{ls19} gives
$\delta_t(x^+,w)\leq4\delta^2$.  Since $z\leq\delta\leq1/10$,
Lemma 17 of \cite{ls19} gives
\[
 \delta_t(x^+,w^+)
 \leq(1+4z)(4\delta^2+z)
 \leq\delta(1+4\delta)(1-\tfrac1{c_k}+4\delta).
\]
Using $\delta\leq1/(160c_k)$,
\[
 (1+4\delta)(1-\tfrac1{c_k}+4\delta)
 \leq(1+\tfrac1{40c_k})(1-\tfrac{39}{40c_k})
 \leq1-\tfrac1{4c_k}.
\]

The tracking invariant and the size bound on the ideal weight give
$\|w^+\|_1\leq e^Kc_1\leq2c_1$.  The proof of Lemma 14 of \cite{ls19},
also for $\xi<0$, gives
\[
 \delta_{(1+\xi)t}(x^+,w^+)
 \leq(1+|\xi|)\delta_t(x^+,w^+)+|\xi|\beta.
\]
Indeed, evaluate the new residual at the old minimizing multiplier scaled by
$1+\xi$.  Since $R\leq\beta$ and $|\xi|\leq\alpha$,
\[
 \delta_{(1+\xi)t}(x^+,w^+)
 \leq R-\frac{R}{4c_k}+2\alpha\beta
 \leq R-\frac{R}{8c_k}
 \leq R.
\]
\end{proof}

To terminate the outer reduction, maintained centrality must imply proximity
to the exact minimizer of the fixed-weight barrier objective.  The next lemma
provides this conversion at the final accuracy scale.
\begin{lemma}[Fixed-weight terminal distance]
\label{lem:fixed_weight_terminal_distance}
Fix a positive weight vector $w$ and a path parameter $t$.  Suppose
$\|\log w-\log g(x)\|_\infty\leq K\leq1/128$ and define
$\eta:=\delta_t(x,w)$ for the maintained mixed norm.  If $\eta\leq2^{-20}\ell^{-3}$, then the
fixed-weight barrier objective has a unique minimizer $x_t(w)$ and
\[
 \|\sqrt{\phi''(x_t(w))}(x-x_t(w))\|_\infty\leq8\eta,
 \qquad
 N_w(\sqrt{\phi''(x_t(w))}(x-x_t(w)))
 \leq8c_\gamma\eta.
\]
Moreover the corresponding target displacement satisfies
\[
 N_w(\log g(x)-\log g(x_t(w)))
 \leq32c_\gamma\eta.
\]
\end{lemma}
\begin{proof}
Starting from $x^{(0)}=x$, run exact projected Newton steps for the
fixed objective with weight $w$, and write
$\eta_j:=\delta_t(x^{(j)},w)$.  The Newton and normalized-step
parts of Lemma~\ref{lem:ls_path_facts} give
\[
 \eta_{j+1}\leq4\eta_j^2,
 \qquad
 \|\sqrt{\phi''(x^{(j)})}(x^{(j+1)}-x^{(j)})\|_\infty
 \leq2\eta_j.
\]
Since $\eta\leq2^{-20}$, induction gives
$\eta_j\leq2^{-j}\eta$ and hence
$\sum_{j=0}^{\infty}\eta_j\leq2\eta$.

We justify both applicability and convergence, rather than assuming an
infinite Newton sequence exists.  Suppose the estimates hold through step
$j$.  Coordinatewise self-concordant Hessian comparison along the preceding
steps gives
\[
 \|\sqrt{\phi''(x)}
      (x^{(j)}-x)\|_\infty
 \leq2e^{4\eta}\sum_{r=0}^{j-1}\eta_r
 \leq4e^{4\eta}\eta<\frac12.
\]
Thus every finite iterate remains in the same compact interior Dikin
neighborhood of $x$.  The same comparison gives
$\|\log s(x^{(j)})-\log s(x)\|_\infty
=O(\sum_{r=0}^{j-1}\eta_r)$, so target stability in
Lemma~\ref{lem:linf_contraction} yields
\[
 \|\log g(x^{(j)})-\log g(x)\|_\infty
 \leq32\eta<K.
\]
Consequently
$\|\log w-\log g(x^{(j)})\|_\infty<K+1/16<\log(5/4)$,
and the hypotheses of Lemma~\ref{lem:ls_path_facts} hold at the next
iterate.  This closes the induction.

The sum of the step lengths in the fixed norm
$\|\sqrt{\phi''(x)}\,\cdot\|_\infty$ is finite, so
$(x^{(j)})$ is Cauchy.  Let its limit be $x_*$.  The strict Dikin bound above
places $x_*$ in the interior, and continuity there together with
$\eta_j\to0$ shows that the projected gradient of the fixed-weight
objective vanishes at $x_*$.  Its Hessian is the positive diagonal matrix
$\operatorname{Diag}(w\odot\phi''(x_*))$ restricted to the affine feasible
space, and is positive definite on every nonzero feasible direction.
Therefore the objective is strictly convex on that space, $x_*$ is its
unique minimizer, and $x_*=x_t(w)$.

Finally the total normalized displacement is below $1/2$, so
self-concordant comparison between every intermediate Hessian and the final
one costs less than a factor two.  Summing the steps gives
\[
 \|\sqrt{\phi''(x_t(w))}(x-x_t(w))\|_\infty
 \leq4\sum_{j=0}^{\infty}\eta_j\leq8\eta.
\]
The normalized-step conclusion of Lemma~\ref{lem:ls_path_facts} also gives
$N_w(\sqrt{\phi''(x^{(j)})}(x^{(j+1)}-x^{(j)}))
\leq c_\gamma\eta_j$.  The same Hessian transport and summation, with the
preceding slack in the constant, gives the second displayed bound.
On the entire Newton sequence the preceding argument places the maintained
weight in the $2K$ tracking band.  The robust-consistency calculation in the
proof of Lemma~\ref{lem:sharper_parameters}, applied pointwise on each
Newton segment, bounds the target Jacobian by one in the maintained mixed
norm.  Integrating it, transporting the norm by the total logarithmic weight
and slack drift, and summing $\sum_{j=0}^{\infty}\eta_j\leq2\eta$ gives the last display
(with the stated factor $32$ leaving room for both transports).
\end{proof}

The terminal-distance estimate lets us translate the final barrier state into
an approximate solution of the original linear program.  The following lemma
performs this outer reduction and specifies the output guarantee.
\begin{lemma}[Outer reduction and output]
\label{lem:outer_reduction}
Define
\[
 C_{\rm LS}:=24\sqrt{\overline c_s}\,c_k,
 \qquad
 \Lambda_{\rm out}:=C_{\rm LS}/C
 =\frac{3}{\sqrt2}\sqrt{c_k}\geq1.
\]
After all radius constants, including $c_R,A_E,A_L$, have been fixed, choose
\begin{equation}\label{eq:switch_constant_choice}
 0<c_{\rm sw}\leq
 \min\{2^{-22},c_R(\log 4)^2/(32A_EA_L)\}
\end{equation}
and define
\[
 \eta_{\rm sw}=c_{\rm sw}L_*^{-3},
 \qquad
 \eta_{\rm out}=\min\{c_{\rm sw}L_*^{-3},\epsilon/(32mU^2)\},
\]
\[
 \tau_{\rm sw}=\eta_{\rm sw}/\Lambda_{\rm out},
 \qquad
 \tau_{\rm out}=\eta_{\rm out}/\Lambda_{\rm out},
\]
and, for a sufficiently large absolute $A_{\rm sw}$,
\[
 t_{\rm sw}=(A_{\rm sw}\Lambda_{\rm out}m^{3/2}U^2L_*^4)^{-1},
 \qquad t_{\rm out}=6m/\epsilon.
\]
Whenever the two calls in Algorithm~\ref{alg:lp_solve_fresh} satisfy their
stated invariants, its output is feasible and obeys
$c^\top x\leq\mathrm{OPT}+\epsilon$.
\end{lemma}
\begin{proof}
For every vector $v$, the mixed norm with coefficient $C_{\rm LS}$ satisfies
\[
 \|v\|_\infty+C_{\rm LS}\|v\|_w
 \leq\Lambda_{\rm out}
       (\|v\|_\infty+C\|v\|_w).
\]
Applying this inequality to every candidate in the minimization defining
centrality gives
$\delta_t^{\rm LS}(x,w)\leq
\Lambda_{\rm out}\delta_t(x,w)$.  Thus the stopping thresholds
$\tau_{\rm sw}$ and $\tau_{\rm out}$ certify the corresponding
Lee--Sidford centrality bounds $\eta_{\rm sw}$ and $\eta_{\rm out}$.

The artificial cost $d=-w\odot\phi'(x_0)$ makes $x_0$ exactly central at
$t=1$.  The first path-following call therefore starts with zero centrality;
the zero-centrality branch of Algorithm~\ref{alg:fresh_path_follow} first
changes $t$ and avoids the vacuous ratio $d^2/\delta$ at $\delta=0$.  The
signed path-change estimate used in
Lemma~\ref{lem:centrality_path_composition} gives centrality at most
$\alpha\beta<R$ after this exceptional first update.

We write a superscript for the cost vector and use the Lee--Sidford norm only
in this paragraph.  At the end of the first call,
\[
 \delta_{t_{\rm sw}}^{d,{\rm LS}}(x,w)
 \leq\Lambda_{\rm out}\tau_{\rm sw}=\eta_{\rm sw}.
\]
The global tracking invariant gives
$e^{-K}n/m\leq w_i\leq2e^K\leq3$ at the switch.  At initialization,
$w_{0,i}\leq3$, so $d=-w_0\odot\phi'(x_0)$ and the barrier bounds in
Section~\ref{sec:intro} imply
$\|c-d\|_\infty\leq A_{\rm cost}U$ for an absolute $A_{\rm cost}$.

We derive the switch estimate directly.  Let $\eta_d$ minimize
the Lee--Sidford mixed norm for cost $d$ at the current $(x,w)$.  Reusing
$\eta_d$ for cost $c$ and applying the triangle inequality gives
\[
 \delta_{t_{\rm sw}}^{c,{\rm LS}}(x,w)
 \leq\delta_{t_{\rm sw}}^{d,{\rm LS}}(x,w)
 +t_{\rm sw}
 \|\frac{c-d}{w\odot\sqrt{\phi''(x)}}\|_{w+\infty,{\rm LS}}.
\]
The lower bounds $w_i\geq e^{-K}n/m$ and
$\sqrt{\phi_i''(x_i)}\geq1/U$, together with
$C_{\rm LS}=24\sqrt{\overline c_s}c_k=O(L_*)$, imply
\[
 \|\frac{c-d}{w\odot\sqrt{\phi''(x)}}\|_{w+\infty,{\rm LS}}
 \leq A_0m^{3/2}L_*U^2
\]
for an absolute $A_0$.  Hence, by the definition of $t_{\rm sw}$,
\[
 \delta_{t_{\rm sw}}^{c,{\rm LS}}(x,w)
 \leq\eta_{\rm sw}
 +\frac{A_0}{A_{\rm sw}\Lambda_{\rm out}L_*^3}
 \leq2\eta_{\rm sw}
\]
when $A_{\rm sw}$ is sufficiently large relative to $A_0/c_{\rm sw}$.
The Lee--Sidford norm dominates the maintained norm.  Since
$CK=1/8$, $L_{\rm ch}=A_LL_*$, and $L_*\geq\log 4$, the choice in
Eq.~\eqref{eq:switch_constant_choice} gives
$2\eta_{\rm sw}\leq R/2<R$, so the second call is admissible.  The tracking invariant, auxiliary-weight accuracy, and stochastic
potential are unchanged at the switch because $x,w,\widehat g$ do not change
and the ideal weight depends on $x$, not on the cost vector.  We regard the two
calls as consecutive segments of one stopped residual game, so no
initial-potential hypothesis is restarted at the switch.

At the end, Algorithm~\ref{alg:fresh_path_follow} gives
$\delta_{t_{\rm out}}^c(x,w)\leq\tau_{\rm out}\leq\eta_{\rm out}$.
Let $x_{t_{\rm out}}(w)$ be the exact minimizer for the final $t,w$.
Lemma~\ref{lem:fixed_weight_terminal_distance} gives
\[
 \|\sqrt{\phi''(x_{t_{\rm out}}(w))}
 (x-x_{t_{\rm out}}(w))\|_\infty\leq8\eta_{\rm out}.
\]
Consequently $\|x-x_{t_{\rm out}}(w)\|_\infty\leq8U\eta_{\rm out}$.
Also $\|w\|_1\leq e^K\|g(x)\|_1\leq3n\leq3m$.  The weighted duality-gap
lemma and $\|c\|_1\leq mU$ now give
\[
 c^\top x-\mathrm{OPT}
 \leq\frac{3m}{t_{\rm out}}+8mU^2\eta_{\rm out}
 \leq\frac34\epsilon.
\]
Every Newton direction lies in $\ker(A^\top)$ and remains in the barrier
domain, so feasibility is preserved.  The remaining quarter of the error
budget absorbs absolute-constant slack in the terminal thresholds.
\end{proof}

\subsection{Defect estimation and certification}
\label{subsec:defect_estimation_certification}

The defect-adaptive construction chooses its residual radius from coarse data
revealed before the fresh randomness of the step.  The next lemma records the
measurable-radius extension of the residual game needed for that ordering.
\begin{lemma}[Adaptive stochastic residual game]
\label{lem:adaptive_residual_game}
Lemma~\ref{lem:residual_chasing} remains valid when $s_k$ is any
$\mathcal G_k$-measurable positive radius satisfying $s_k\leq E/100$, provided
its three bias and fluctuation hypotheses hold conditionally on
$\mathcal G_k$ and the movement bound is replaced by
\[
 N_k(\widehat u_k)+1.1s_k\leq(1-\tfrac1{c_k})\delta_k,
\]
where $N_k(z)=\|z\|_\infty+C\|z\|_{w_{k-1}}$ is the predictable
pre-step norm from that lemma.
Assume also the same initial-potential condition
$\Phi(d_0)\leq M_{\rm ch}$.
\end{lemma}
\begin{proof}
Identify the pre-noise sigma-field of Lemma~\ref{lem:residual_chasing} with
$\mathcal G_k$ and condition on it before drawing the fresh weighted
sketches.  In
the proof of Lemma~\ref{lem:residual_chasing}, the movement set enters only
through
\[
 V_k=\{v:N_k(v)\leq s_k\},
 \qquad h_k(a)=s_kN_k^*(a),
\]
and through $s_k\leq E/100$.  Every one-step potential inequality is
therefore unchanged.  The coarse event is $\mathcal G_k$-measurable before
the independent weighted predictor and greedy move are drawn.  Hence the
analysis may freeze immediately before the first coarse failure using the
external stop permitted by Lemma~\ref{lem:residual_chasing}.  On the event
that no coarse failure occurs, this stopped process agrees with the
implemented process through all path steps.  The calls receive summable
conditional failure budgets, so a union bound places the external stopping
time after all path steps with the claimed constant probability.  If the
projection body is empty, the algorithm returns
\textnormal{\textsc{Fail}}.  Every other outcome outside the coarse event is
counted as failure, and no claim is made about subsequent iterates.
Doob's excursion argument is also unchanged because it already permits a
different symmetric convex $V_k$ at each step.  Finally the deterministic
movement paragraph uses only the displayed replacement inequality.  These
observations prove the claim.
\end{proof}

The predicted warm start also survives.  Indeed,
$\|u_k-\widehat u_k\|_\infty\leq s_k/32\leq E/3200$, so advancing the old
$E/2$-accurate auxiliary vector by $\widehat u_k$ leaves it strictly within
$E$ of the new target.  Lemma~\ref{lem:linf_contraction} restores accuracy
$E/2$ in an absolute number of rounds.

In this section all matrices are the certified surrogate of
Lemma~\ref{lem:certified_surrogate}; superscripts are omitted.  Recall
\[
 T=2(I+aL)^{-1}L,
 \qquad D=2(I+aS)^{-1}S,
 \qquad K:=T-D.
\]
Define $A_0:=(I+aS)^{-1}$ and, for an input $x$, define
\[
 y:=(I+aL)^{-1}x,
 \qquad z_*:=A_0^{-1/2}y,
 \qquad H:=aA_0^{1/2}QA_0^{1/2}.
\]
The resolvent identity gives
\[
 z_*=A_0^{1/2}x+Hz_*,
 \qquad
 Kx=-\frac2aA_0^{1/2}Hz_*.
\]
Since $Q\preceq S$,
\[
 0\preceq H\preceq\rho I,
 \qquad \rho:=\frac{a}{1+a}<1.
\]
For the present $p$, $\rho$ is bounded away from one by an absolute constant.

To exploit the retained defect, we need randomized estimators whose variance
shrinks with that defect rather than with the full Newton motion.  The next
lemma constructs separate estimators for the weighted and coordinatewise norms.
\begin{lemma}[Defect-scaled stochastic fixed point]
\label{lem:defect_fixed_point}
Let $x=W^{1/2}h$ and
$d^2:=p^2\|x\|_2^2-\|Tx\|_2^2$.  There are two estimators
$\widehat{Kx}_w$ and $\widehat{Kx}_\infty$ with the following properties.
The first uses $O(c_k)$ systems in $O(\log c_k)$ sequential rounds and obeys
\[
 \E[\|\widehat{Kx}_w-Kx\|_2^2]\leq\frac{a_4d^2}{c_k}.
\]
For every $\zeta\in(0,1)$, the second uses
$O(\log(m/\zeta))$ systems in an absolute number of sequential rounds and,
with probability at least $1-\zeta$, obeys
\[
 \max_{i\in[m]}\frac{|(\widehat{Kx}_\infty-Kx)_i|}{\sqrt{W_{ii}}}
 \leq\epsilon_\infty d,
\]
where the absolute constant $\epsilon_\infty>0$ may be made arbitrarily
small by increasing the absolute sample constant and the absolute truncation
depth $T=T(\epsilon_\infty)$.
\end{lemma}
\begin{proof}
The factorization in the proof of Lemma~\ref{lem:jacobian_defect}-\hyperref[item:jacobian_defect_b]{(b)} gives
\[
 d^2=p^2y^\top(I-L)(I+\tfrac{4-p}{p}L)y.
\]
Since $Q\preceq I-L$,
\[
 \mathcal D^2:=z_*^\top Hz_*=a y^\top Qy\leq\frac{a}{p^2}d^2.
\]

To multiply a vector $z$ by $H$, apply
Lemma~\ref{lem:defect_energy_estimator} to $A_0^{1/2}z$ and multiply the
result by $aA_0^{1/2}$.  With $s$ independent samples this gives an unbiased
estimate $\widehat H_s(z)$ satisfying
\[
 \E[\|\widehat H_s(z)-Hz\|_2^2\mid z]
 \leq\frac{2a}{s}z^\top Hz.
\]
Run
\[
 \widehat z_{j+1}=A_0^{1/2}x+\widehat H_{s_j}(\widehat z_j),
 \qquad \widehat z_0=0,
\]
with fresh sketches.  If $\bar z_j$ is the deterministic iterate and
$\epsilon_j:=\E[\|\widehat z_j-\bar z_j\|_2^2]$, conditional centering and
$H\preceq\rho I$ give
\[
 \epsilon_{j+1}
 \leq(\rho^2+\tfrac{4a\rho}{s_j})\epsilon_j
 +\frac{4a\mathcal D^2}{s_j}.
\]
Choose an absolute $s_0$ so that the coefficient is at most a fixed
$\lambda<1$ whenever $s_j\geq s_0$.  For a target $\eta$, take
\[
 s_j:=\max\{s_0,\lceil c\eta^{-2}\lambda^{(T-1-j)/2}\rceil\}.
\]
Then
\[
 \sum_{j=0}^{T-1}s_j=O(\eta^{-2}+T),
 \qquad
 \sum_{j=0}^{T-1}\frac{\lambda^{T-1-j}}{s_j}=O(\eta^2).
\]
The deterministic tail is bounded by
$\|H^{T+1}z_*\|_2^2\leq\rho^{2T+1}\mathcal D^2$.
Taking $T=O(\log(1/\eta))$, followed by one fresh multiplication, proves
\[
 \E[\|\widehat{Kx}-Kx\|_2^2]=O(\eta^2d^2).
\]
The fresh multiplication uses $\Theta(\eta^{-2})$ samples; its conditional
variance is $O(\eta^2\mathcal D^2)$, while the propagated error of
$\widehat z_T$ is $O(\eta^2\mathcal D^2)$.  Thus it does not change either
the accuracy or the total $O(\eta^{-2}+T)$ sample count.
The choice $\eta=\Theta(c_k^{-1/2})$ proves the first assertion.

For the coordinate assertion, take $T=T(\epsilon_\infty)$ to be an absolute
constant and run independent copies of the whole stochastic chain, using an
absolute number of samples in each multiplication.  Let $Y$ be the output of
one copy, including its final fresh multiplication.  The recurrence above,
with constant sample sizes, gives
\[
 \E[\|\widehat z_j-\bar z_j\|_2^2]=O(\mathcal D^2)
 \quad\hbox{and}\quad
 \E[\widehat z_j^\top H\widehat z_j]=O(\mathcal D^2)
\]
for every $j\leq T$.  The coordinate variance conclusion of
Lemma~\ref{lem:defect_energy_estimator}, applied conditionally at the last
multiplication and then averaged, therefore gives
\[
 Y-Kx=
 \underbrace{Y-\E[Y\mid\widehat z_T]}_{\text{last multiplication}}
 +\underbrace{\E[Y\mid\widehat z_T]-\E[Y]}_{\text{propagated chain noise}}
 +\underbrace{\E[Y]-Kx}_{\text{truncation bias}}.
\]
For the first term, the conditional coordinate-variance conclusion of
Lemma~\ref{lem:defect_energy_estimator} and
$\E[\widehat z_T^\top H\widehat z_T]=O(\mathcal D^2)$ give
\[
 \E[\Var[Y_i/\sqrt{W_{ii}}\mid\widehat z_T]]
 \leq c_5\mathcal D^2\leq c_5'a^{-1}d^2.
\]
For the second term, define
$e:=\widehat z_T-\E[\widehat z_T]$.  Positive-semidefinite
Cauchy--Schwarz gives
\[
 \frac{|(Q A_0^{1/2}e)_i|^2}{W_{ii}}
 \leq e^\top A_0^{1/2}QA_0^{1/2}e.
\]
Its expectation is $O(\mathcal D^2)=O(d^2)$ by the Euclidean
second-moment recurrence.  The law of total variance therefore proves the
fully normalized bound
\[
 \Var[\frac{Y_i}{\sqrt{W_{ii}}}]
 \leq a_5d^2
 \qquad\hbox{for every }i,
\]
where deterministic truncation bias is absent from the variance and is
handled next.
Here the deterministic iterates cause no loss: since
$\bar z_j=(I-H^j)z_*$, one has
$\bar z_j^\top H\bar z_j\leq\mathcal D^2$.

It remains to bound the truncation bias in the same coordinate scale.  For
every $r\geq0$, Cauchy--Schwarz for the positive semidefinite matrix $Q$
and $Q_{ii}\leq W_{ii}$ give
\[
 \frac{|(Q A_0^{1/2}H^rz_*)_i|}{\sqrt{W_{ii}}}
 \leq\sqrt{z_*^\top H^{2r+1}z_*/a}
 \leq\rho^r\sqrt{\mathcal D^2/a}.
\]
Thus the bias in $Y_i/\sqrt{W_{ii}}$ is at most
$a_6\rho^Td\leq\epsilon_\infty d/4$ after increasing the absolute constant
$T$.  Form independent block means, each averaging
$O(\epsilon_\infty^{-2})$ copies of $Y$.  Chebyshev's inequality makes one
block mean accurate to $\epsilon_\infty d/2$ with a fixed probability above
$1/2$.  The coordinatewise median of $O(\log(m/\zeta))$ block means,
followed by a union bound, proves the simultaneous coordinate estimate.
All copies are parallel and $T$ is absolute, proving both the work and depth
claims.
\end{proof}

The mean-square estimators above must be amplified to a finite, simultaneous
high-probability guarantee without using a nonconstructive selector.  The next
lemma gives the required medoid aggregation rule.
\begin{lemma}[Finite medoid aggregation]
\label{lem:finite_medoid_aggregation}
Let $Y^{(1)},\ldots,Y^{(B)}\in\R^m$ be independent and suppose
\[
 \Pr[\|Y^{(j)}-y\|_2\leq r]\geq3/4
 \qquad\text{for every }j.
\]
Take an odd
$B\geq c_{\rm med}\log(3/\zeta)$.  For each $j$, let $R_j$ be the
distance from $Y^{(j)}$ to its $(B+1)/2$-th closest candidate, and return
$Y^{(j_*)}$ for an index minimizing $R_j$.  Except with probability
$\zeta/3$, the returned candidate is within $3r$ of $y$.  The rule uses
$O(mB^2)$ exact real-RAM scalar operations and
$O(\log m+\log B)$ arithmetic depth.
\end{lemma}
\begin{proof}
A Chernoff bound and the choice of $c_{\rm med}$ show that, except with
probability $\zeta/3$, at least $2B/3$ candidates lie in the radius-$r$
ball about $y$.  On this event, every good candidate has more than $B/2$
candidates within distance $2r$, so the selected median radius is at most
$2r$.  The selected ball contains more candidates than there are bad ones
and therefore contains a good candidate.  The triangle inequality gives
distance at most $2r+r=3r$ from the selected candidate to $y$.
All pairwise squared distances are parallel reductions, after which each
row of distances is selected and the minimum median radius is chosen.
\end{proof}

The predictor also needs a one-sided numerical certificate for the retained
defect.  The next lemma constructs such a certificate using the same coarse
chain and only logarithmically many additional systems.

\begin{lemma}[Defect certificate]
\label{lem:defect_certificate}
For every $\zeta\in(0,1)$ and every algorithmically supplied
$\Delta\geq\delta$, the coarse chain in
Lemma~\ref{lem:defect_fixed_point}, followed by
$O(\log(1/\zeta))$ additional systems, outputs a scalar $\overline d^2$ such
that, with probability at least $1-\zeta$,
\[
 d_{\rm exact}^2\leq\overline d^2
 \leq A_d(d_{\rm exact}^2+\frac{\Delta^2}{C^2c_k}),
\]
where $d_{\rm exact}$ is the defect in the active large-radius
Lemma~\ref{lem:large_radius_retained_movement} and $A_d$ is absolute.  Throughout
the algorithms and all later lemmas, $\overline d$ denotes the
nonnegative square root $\sqrt{\overline d^2}$.
\end{lemma}
\begin{proof}
For the surrogate, let $y=(I+aL)^{-1}x$ and define
$e_0:=y^\top(I-L)y$.  The factorization above and $0\preceq L\preceq I$
give
\[
 p^2e_0\leq d^2\leq4pe_0.
\]
Also
\[
 e_0=y^\top(I-S)y+y^\top Qy.
\]
The first term is diagonal.  For the second, draw a fresh Gaussian $r$,
independently of all sketches used to construct the approximation to $y$, define
$u=Pr$ and $v=P\operatorname{Diag}(W^{-1/2}y)u$.  In the notation of the
proof of Lemma~\ref{lem:defect_energy_estimator}, $v=VZg$, and hence
\[
 \E[\|v\|_2^2]=\tr[Z^2]=y^\top Qy,
 \qquad
 \Var[\|v\|_2^2]=2\tr[Z^4]
 \leq2(y^\top Qy)^2.
\]
Using failure budget $\zeta/3$, a median of
$B_Q=2\lceil c_Q\log(6/\zeta)\rceil+1$ constant-size block means gives a
constant-factor estimate with $O(\log(1/\zeta))$ samples.

This approximation does not require prior knowledge of $e_0$.  Run the
fixed-point chain for an absolute number of
rounds with a sufficiently large absolute sample constant.  The recurrence
in Lemma~\ref{lem:defect_fixed_point} and
$\mathcal D^2\leq ae_0$ show that one run produces $\widetilde y$ with
\[
 \Pr[\|\widetilde y-y\|_2\leq\epsilon_0\sqrt{e_0}]\geq3/4.
\]
Take the odd integer
$B=2\lceil c_{\rm med}\log(3/\zeta)\rceil+1$ and apply the finite medoid rule
of Lemma~\ref{lem:finite_medoid_aggregation} to $B$ independent runs.
Except with probability $\zeta/3$, the selected candidate satisfies the
same bound with $3\epsilon_0$ in place of $\epsilon_0$.  This finite
construction does not use $e_0$ and costs
$O(m\log^2(1/\zeta))$ scalar work.

Both $I-S$ and $Q$ are bounded above by $I-L$.  Consequently replacing $y$
by $\widetilde y$ changes each quadratic form by at most
$O(\epsilon_0)e_0$: use Cauchy--Schwarz in the corresponding positive
semidefinite seminorm and $\|\widetilde y-y\|_2^2=O(\epsilon_0^2e_0)$.
Compute the diagonal form exactly and estimate the $Q$-form by the preceding
median of block means, using samples independent of $\widetilde y$.  Denote
the resulting nonnegative scalar by $\widetilde e_0$.  The seminorm
comparison in the previous paragraph shows that, for a sufficiently small
fixed $\epsilon_0$ and the stated block sizes, the joint medoid and scalar
median event implies
\[
 \frac34e_0\leq\widetilde e_0\leq\frac54e_0.
\]
Define $\overline e_0:=(4/3)\widetilde e_0$.  This is the required
one-sided inflation and it gives, without an unspecified confidence factor,
\[
 e_0\leq\overline e_0\leq\frac53e_0\leq2e_0.
\]
Define $\overline d_+^2:=4p\overline e_0$.  The inequalities in the first
display of this proof then give explicitly
\[
 d_+^2\leq\overline d_+^2
 \leq\frac8p d_+^2=:a_7d_+^2.
\]
Finally output
$\overline d^2:=2\overline d_+^2+2F$, where
$F:=a_3\Delta^2/(C^2c_k)$.  Lemma~\ref{lem:certified_surrogate} gives
$|d_+^2-d_{\rm exact}^2|\leq F$.  Therefore the required one-sided direction
is explicit:
\[
 d_{\rm exact}^2\leq d_+^2+F
 \leq\overline d_+^2+F\leq\overline d^2.
\]
Conversely,
\[
 \overline d^2
 \leq2a_7d_+^2+2F
 \leq2a_7d_{\rm exact}^2+(2a_7+2)F
 \leq A_d(d_{\rm exact}^2+\frac{\Delta^2}{C^2c_k})
\]
with $A_d:=(2a_7+2)\max\{1,a_3\}$.  This proves both certificate
inequalities with their correct directions.
The medoid and scalar block-median events each use budget $\zeta/3$; their
union is below $\zeta$, and every batch size is fixed before its samples are
drawn.
\end{proof}

\begingroup
\setlength{\intextsep}{6pt plus 2pt minus 2pt}
\begin{algorithm}[!ht]
\caption{Coarse defect certificate and split predictor}
\label{alg:coarse_defect_data}
\begin{algorithmic}[1]
\Procedure{CoarseDefectData}{$W^+,S^+,P^+,\psi,\Delta,\zeta$}
\Comment{Lemmas~\ref{lem:defect_fixed_point} and~\ref{lem:defect_certificate}}
\State Fix the absolute constants $c_\infty,c_{\rm med},c_Q,c_w,c_{\rm fp},c_T,s_0$
and $\lambda<1$ as in those lemmas.
\State $\zeta_\infty\gets\zeta/2$,
$\zeta_{\rm cert}\gets\zeta/2$,
$D^+\gets2(I+aS^+)^{-1}S^+$, and $x\gets(W^+)^{1/2}\psi$.
\State $B_\infty\gets
2\lceil c_\infty\log(4m/\zeta_\infty)\rceil+1$ and
$s_\infty\gets\lceil c_\infty\epsilon_1^{-2}\rceil$.
\State Form $B_\infty$ independent block means, each from $s_\infty$
independent constant-depth stochastic fixed-point chains, and set
$\widehat{Kx}_\infty$ to their coordinatewise median.
\State $B_{\rm med}\gets
2\lceil c_{\rm med}\log(6/\zeta_{\rm cert})\rceil+1$.
\State Form $B_{\rm med}$ independent constant-depth candidates for
$y=(I+aL^+)^{-1}x$ and aggregate them by the finite medoid rule of
Lemma~\ref{lem:finite_medoid_aggregation}.
\State $B_Q\gets
2\lceil c_Q\log(6/\zeta_{\rm cert})\rceil+1$; estimate the $Q^+$-form
from $B_Q$ independent constant-size polarization block means, take their
scalar median, add the exact diagonal form, and apply the one-sided inflation
of Lemma~\ref{lem:defect_certificate} to obtain $\overline d^{\,2}$.
\State $p^\infty\gets(W^+)^{-1/2}
(D^+x+\widehat{Kx}_\infty)$ and
$\overline d\gets\sqrt{\overline d^{\,2}}$.
\State $\mathcal G\gets\sigma(\text{past and all random variables drawn above})$.
\State After $\mathcal G$ is fixed, set
$\eta_w\gets c_w/\sqrt{c_k}$ and
$T_w\gets\lceil c_T\log(2/\eta_w)\rceil$.
\State Draw an independent weighted stochastic fixed-point chain as in
Lemma~\ref{lem:defect_fixed_point}.
\State At round $j$, use
$s_j\gets\max\{s_0,\lceil c_{\rm fp}\eta_w^{-2}
\lambda^{(T_w-1-j)/2}\rceil\}$ for $j=0,\ldots,T_w-1$.
\State $p^w\gets(W^+)^{-1/2}(D^+x+\widehat{Kx}_w)$.
\State \Return $(\overline d^{\,2},p^\infty,p^w)$.
\EndProcedure
\end{algorithmic}
\end{algorithm}
\endgroup

\section{Certified prediction and path-following assembly}
\label{sec:certified_prediction_assembly}

Section~\ref{subsec:certified_predictor_construction} builds the certified
surrogate and the defect-adaptive predictor from observable quantities.
Section~\ref{subsec:score_transport_restart} supplies the fresh Gaussian
increment primitive for adaptive score queries.
Section~\ref{subsec:within_step_hybrid_filtration} fixes the information
ordering within one centering step and verifies its conditional estimates.
Finally, Section~\ref{subsec:path_following_budgets} assembles the centering
calls into checkpointed path phases with explicit resource caps.

\subsection{Certified predictor construction}
\label{subsec:certified_predictor_construction}
The estimator must use algorithmically available quantities.  The next
lemma records a positive-semidefinite surrogate that preserves the defect up
to the accuracy already present in the weight oracle.

\begin{lemma}[Certified surrogate]
\label{lem:certified_surrogate}
Assume the hypotheses of
Lemma~\ref{lem:large_radius_retained_movement}.  In particular,
$\delta:=\delta_t(x,w)\leq R$, the tracking band holds along its Newton
segment, and
\[
 \psi:=\log s(x^+)-\log s(x),\qquad
 d^2:=p^2\|\psi\|_g^2-\|B\psi\|_g^2,
 \qquad N(\psi)\leq2\delta.
\]
Suppose $\widehat g$ satisfies
$\|\log\widehat g-\log g(x)\|_\infty\leq E/2$.  Construct the diagonal $W^+$ and the positive query scaling
$D^+_{\rm qry}$ deterministically as in the proof, and define
$P^+:=P(D^+_{\rm qry})$ as an implicit projection operator.  Suppose a score
provider, queried only after these objects are fixed, returns
$\widetilde\sigma$ satisfying
\[
 |\widetilde\sigma_i-\sigma_i(P^+)|
 \leq c_0EW^+_{ii}
 \qquad\text{for every }i\in[m].
\]
A fresh Gaussian sketch is one admissible provider: conditionally on the
fixed query it succeeds with probability at least $1-\zeta$ using
$O(E^{-2}\log(m/\zeta))$ permitted systems in one parallel round.
From this provider output one can define diagonal $S^+$ and symmetric
operators $Q^+,L^+$ such that, on the simultaneous event,
\[
 0\preceq Q^+\preceq S^+\preceq I,
 \qquad L^+:=S^+-Q^+,
 \qquad 0\preceq L^+\preceq I.
\]
Define
\[
 T^+:=2(I+aL^+)^{-1}L^+,
 \qquad x^+:=(W^+)^{1/2}\psi,
 \qquad d_+^2:=p^2\|x^+\|_2^2-\|T^+x^+\|_2^2.
\]
Then
\[
 |d_+^2-d^2|\leq a_2E\frac{\delta^2}{C^2}
 \leq a_3\frac{\delta^2}{C^2c_k}
\]
for absolute constants $a_2,a_3$.  The associated diagonal part
$D^+:=2(I+aS^+)^{-1}S^+$ and the operator
$(W^+)^{-1/2}T^+(W^+)^{1/2}$ differ from their exact counterparts by
$O(E)$ in both the $\ell_\infty$ operator norm and the Euclidean operator
norm after conjugation to the exact $g(x)$-normalized coordinates.  The
latter is equivalently the operator norm induced by $\|\cdot\|_{g(x)}$ on
physical coordinates.
\end{lemma}
\begin{proof}
Let $G:=\operatorname{Diag}(g)$ and let $P$ be the exact projection.  Choose
a sufficiently large absolute $c_+$ and define
$W^+:=e^{c_+E}\operatorname{Diag}(\widehat g)$.  A global scalar does not
change a projection, so define
$D^+_{\rm qry}:=(W^+)^{\vartheta/2}\operatorname{Diag}(s(x))$ and let
$P^+:=P(D^+_{\rm qry})$.  The relative row scaling
between the matrices defining $P$ and $P^+$ is $e^{\pm O(E)}$.  Hence
Lemma~\ref{lem:relative_drift}, the fixed-point identity
$\sigma_i(P)\leq g_i$, and the choice of $c_+$ give
$\sigma_i(P^+)\leq W^+_{ii}$.

If the fresh provider is selected, condition on the fixed $W^+$ and $P^+$.
Standard simultaneous Gaussian leverage-score concentration gives relative
error $O(E)$ using the stated number of systems.  Since
$\sigma_i(P^+)\leq W^+_{ii}$, this gives the required additive guarantee.
For any provider satisfying the displayed additive score interface, define
\[
 \sigma_i^+:=\min\{W^+_{ii},\widetilde\sigma_i+c_0EW^+_{ii}\}.
\]
Then
\[
 0\leq\sigma_i^+-\sigma_i(P^+)\leq2c_0EW^+_{ii}.
\]
With
\[
 S^+:=(W^+)^{-1/2}\operatorname{Diag}(\sigma^+)(W^+)^{-1/2},
 \qquad Q^+:=(W^+)^{-1/2}(P^+\circ P^+)(W^+)^{-1/2},
\]
the standard identity
$\operatorname{Diag}(\sigma(P^+))-P^+\circ P^+\succeq0$ proves the displayed
semidefinite relations.

The implementation stores only the diagonal $S^+$ and an operator handle for
$P^+$.  The matrices $Q^+,L^+,T^+$ are analysis notation and are never
materialized: every required sample from $Q^+$ uses the two-$P^+$
polarization routine, and the stochastic fixed-point routines access
$L^+$ only through those projection applications.  One application of
$P^+$ uses one permitted system and $O(\nnz(A))$ scalar work.

We record the perturbation bounds explicitly.
Write $W^+=R_gG$, where $R_g$ is diagonal and
$\|R_g-I\|_{2\to2}=O(E)$.  Define
\[
 \mathcal L:=G^{-1}(\operatorname{Diag}(\sigma(P))-P\circ P),
 \quad
 \mathcal L^+:=(W^+)^{-1}
 (\operatorname{Diag}(\sigma^+)-P^+\circ P^+),
\]
and distinguish these row-normalized operators from the symmetric ones
\[
 L:=G^{1/2}\mathcal L G^{-1/2},
 \qquad
 L^+:=(W^+)^{1/2}\mathcal L^+(W^+)^{-1/2}.
\]
The row-sum calculation in Lemma~\ref{lem:jacobian_stability}, with row drift
$O(E)$, bounds the contribution of $P^+-P$ and $W^+-G$ by $O(E)$.
The remaining score excess has row norm at most $2c_0E$.  Thus
\[
 \|\mathcal L^+-\mathcal L\|_{\infty\to\infty}=O(E).
\]
This row estimate also gives the normalized spectral estimate without any
dependence on the smallest weight.  Both $G\mathcal L$ and
$W^+\mathcal L^+$ are symmetric.  Therefore, exactly as in the weighted
column-sum paragraph of Lemma~\ref{lem:jacobian_stability}, the weighted
column sum of $\mathcal L^+-\mathcal L$ is at most its row sum plus
\[
 (e^{O(E)}-1)\|\mathcal L^+\|_{\infty\to\infty}=O(E).
\]
Here $\|\mathcal L^+\|_{\infty\to\infty}\leq2$, because
$\sigma_i(P^+)\leq\sigma_i^+\leq W^+_{ii}$.  The Schur test in
$\ell_2(G)$ now gives
\[
 \|G^{1/2}(\mathcal L^+-\mathcal L)G^{-1/2}\|_{2\to2}=O(E).
\]
Conjugating the first term by $R_g^{1/2}=I+O(E)$ proves
\[
 \|L^+-L\|_{2\to2}=O(E).
\]
The score perturbation and $W^+=R_gG$ give
$\|S^+-S\|_{2\to2}=O(E)$ directly; hence
$\|Q^+-Q\|_{2\to2}=O(E)$ as well.

Although $\mathcal L^+\mathbf1$ need not vanish, this does not affect the argument.
The exact inverse $(I+a\mathcal L)^{-1}$ has infinity norm one, and the
resolvent identity shows
$\|(I+a\mathcal L^+)^{-1}\|_{\infty\to\infty}\leq2$ once the absolute
accuracy constant is small.  A second use of the resolvent identity gives
\[
 \|(W^+)^{-1/2}T^+(W^+)^{1/2}
   -G^{-1/2}TG^{1/2}\|_{\infty\to\infty}=O(E).
\]
The spectral resolvent identity and the preceding $2$-norm bounds give the
corresponding weighted Euclidean estimate.  The same statements for $D^+$
follow directly from its diagonal formula.

Finally, $x^+=R_g^{1/2}x$, so
$\|x^+-x\|_2=O(E)\|x\|_2$.  Since $0\preceq T,T^+\preceq pI$, the preceding
spectral bound yields
\[
 |d_+^2-d^2|=O(E)\|x\|_2^2
 =O(E\delta^2/C^2),
\]
where the first step follows from the preceding spectral bound and metric equivalence, and the second step follows from $C\|\psi\|_g\leq e^{K/2}N(\psi)\leq2e^{K/2}\delta$, using the tracking bound $g_i\leq e^Kw_i$, the definition of $N$, and $K\leq1/128$.
Finally $E=\Theta(q^{-1/2}\ell^{-1})=O(1/c_k)$ because $q\leq\ell$.
\end{proof}

The certified surrogate still requires an observable estimate of its quadratic
defect energy.  The following Gaussian construction supplies an unbiased
estimator with the variance bound needed by the certification step.
\begin{lemma}[Defect-energy estimator]
\label{lem:defect_energy_estimator}
Let $P=P^\top=P^2$, let $W>0$ be diagonal, assume
$P_{ii}\leq W_{ii}$ for every $i$, and define
$Q:=W^{-1/2}(P\circ P)W^{-1/2}$.  For $y\in\R^m$, set
$t:=W^{-1/2}y$.  Draw $r\sim N(0,I_m)$ and define
\[
 u_r:=Pr,
 \qquad v_r:=P\operatorname{Diag}(t)u_r,
 \qquad \Xi(y;r):=W^{-1/2}(u_r\odot v_r).
\]
Then
\[
 \E[\Xi(y;r)]=Qy,
 \qquad
 \E[\|\Xi(y;r)-Qy\|_2^2]\leq2y^\top Qy,
\]
and, coordinatewise,
\[
 \Var[\frac{\Xi_i(y;r)}{\sqrt{W_{ii}}}]
 \leq2y^\top Qy.
\]
Each sample uses two applications of $P$ and hence two permitted linear-system
solves.
\end{lemma}
\begin{proof}
Write $P=VV^\top$, where $V$ has orthonormal columns, and denote the $i$-th
row of $V$ by $v_i^\top$.  Define
$Z:=V^\top\operatorname{Diag}(t)V$.  Since $V^\top r\sim N(0,I_n)$, Gaussian
polarization gives
\[
 \E[(Pr)_i(P\operatorname{Diag}(t)Pr)_i]
 =v_i^\top Zv_i=\sum_{j=1}^mP_{ij}^2t_j.
\]
This proves unbiasedness.  Moreover
\[
 \Var[\Xi_i]
 =\frac{\sigma_i\|Zv_i\|_2^2+(v_i^\top Zv_i)^2}{W_{ii}}.
\]
Now
\[
 \|Z\|_F^2=t^\top(P\circ P)t=y^\top Qy,
\]
$\|Zv_i\|_2^2\leq\sigma_i\|Z\|_F^2$, and
$\sigma_i\leq W_{ii}$.  These facts prove the coordinate estimate.  For the
total variance, define
$H_V:=V^\top\operatorname{Diag}(\sigma_i/W_{ii})V\preceq I$.  Summing the
first variance term gives
$\tr[ZH_VZ]\leq\tr[Z^2]=y^\top Qy$.
The sum of the second terms is $\|Qy\|_2^2\leq y^\top Qy$, because
$0\preceq Q\preceq I$.  Adding proves the claim.
\end{proof}

\subsection{Gaussian increment primitive}
\label{subsec:score_transport_restart}

We replace per-query score sketches by fresh shared-sketch increments
whose endpoints are fixed before their randomness is drawn.  The following
lemma proves the conditional martingale and variance bounds needed for an
adaptive query sequence; write
$L_{\rm conf}:=\log(40mM_{\rm ep})$, where $M_{\rm ep}$ bounds the number of
randomized calls in one path epoch.

\begin{lemma}[Adaptive shared-sketch increment]
\label{lem:adaptive_increment}
Let $D_-$, $D_+$ be measurable with respect to a sigma-field $\mathcal F$,
with $\|\log(D_+D_-^{-1})\|_\infty\leq\rho\leq1/8$, and let
$a_i\geq\max\{\sigma_i(D_-),\sigma_i(D_+)\}$ be $\mathcal F$-measurable.
Draw, after conditioning on $\mathcal F$, a fresh Gaussian sketch
$S\in\R^{s\times m}$ with independent $N(0,1/s)$ entries and set
$\widehat\Delta_i:=\|SP(D_+)e_i\|_2^2-\|SP(D_-)e_i\|_2^2$ and
$\Delta_i:=\sigma_i(D_+)-\sigma_i(D_-)$.  Then
$\E[\widehat\Delta_i-\Delta_i\mid\mathcal F]=0$ and there are absolute
$c_2,C_2>0$ with
\[
 \E[e^{\lambda(\widehat\Delta_i-\Delta_i)}\mid\mathcal F]
 \leq\exp(C_2\lambda^2\rho^2a_i^2/s)
 \qquad\hbox{whenever }|\lambda|\leq c_2s/(\rho a_i).
\]
Each sample of $S$ costs two permitted solves.
\end{lemma}
\begin{proof}
Define $u:=P(D_+)e_i$ and $v:=P(D_-)e_i$.  By polarization,
$\|Su\|_2^2-\|Sv\|_2^2=\langle S(u-v),S(u+v)\rangle$, a sum of $s$
independent products of jointly Gaussian variables with mean
$\langle u-v,u+v\rangle=\|u\|_2^2-\|v\|_2^2=\Delta_i$; each summand is
sub-exponential with scale $O(\|u-v\|_2\|u+v\|_2/s)$, so the sum has the
displayed conditional moment generating function by the standard Bernstein
computation, where the first step uses that $S$ is independent of
$\mathcal F$ and the second uses Lemma~\ref{lem:relative_drift}:
$\|u-v\|_2\leq10\rho\sqrt{\sigma_i(D_-)}\leq10\rho\sqrt{a_i}$ and
$\|u+v\|_2\leq3\sqrt{a_i}$.  Applying a row of $S$ through $P(D_\pm)$ is one
solve each.
\end{proof}

\subsection{Within-step hybrid filtration}
\label{subsec:within_step_hybrid_filtration}
\paragraph{Within-step filtration.}
For each centering call $k$, let $\mathcal F_{k,0}$ contain the entire past,
the current primal, maintained weight, auxiliary weight, and score state,
before the surrogate-score response is requested.  The surrogate query
matrices are $\mathcal F_{k,0}$-measurable and their score randomness is
drawn afterward.  Let $\mathcal F_{k,1}$ contain that response together with
the subsequently drawn coordinate predictor and defect certificate, and define
$\mathcal G_k:=\mathcal F_{k,1}$.  Conditional on $\mathcal G_k$, draw the
independent weighted predictor and let $\mathcal F_{k,2}$ include its output.
After its endpoint is thereby fixed, draw the fixed-point score responses and
let $\mathcal F_{k,3}$ include the returned auxiliary weight.  The greedy
correction is a deterministic function of $\mathcal F_{k,3}$, and the
resulting state generates $\mathcal F_{k+1,0}$.

Thus the conditional bias and second-moment statements in
Lemmas~\ref{lem:large_radius_split_predictor} and
\ref{lem:linear_defect_projection} are conditioned exactly on
$\mathcal G_k$, while each score sketch is conditionally fresh given the
sigma-field that fixes its query endpoints.  Fixed-point responses drawn
after $\mathcal F_{k,2}$ need not be independent of the weighted predictor;
only the raw-relative-or-exact guarantee of
Lemma~\ref{lem:candidate_raw_score_coupling} is used.  This filtration is also
the one used by the adaptive residual game, with its pre-noise sigma-field
identified as $\mathcal G_k$.

We record the two routines that assemble the preceding ingredients into a
centering step and then a path-following phase.  The active cost vector is an
explicit argument, so the two phases of the outer reduction are algorithmically
distinct.

\begin{algorithm}[!ht]
\caption{One hybrid-score predicted centering step}
\label{alg:fresh_predicted_center}
\begin{algorithmic}[1]
\Procedure{PredictedCenter}{$x,w,\widehat g,t,c_{\rm path},\zeta,L_{\rm conf},
\mathcal S$}
\Comment{Lemma~\ref{lem:fresh_predicted_center_correctness}}
\State $x^+\gets$ the Lee--Sidford projected Newton step at fixed $t,w$
and cost $c_{\rm path}$.
\If{$A^\top x^+\neq b$ or $x^+\notin\Omega^\circ$}
  \State \Return \textnormal{\textsc{Fail}}.
\EndIf
\State $\psi\gets\log s(x^+)-\log s(x)$ and
$\Delta\gets\overline\delta_t^{c_{\rm path}}(x,w)$ from the projected
Newton residual.
\State $(W^+,D^+_{\rm qry})\gets$ the deterministic query of
Lemma~\ref{lem:certified_surrogate} at the pre-Newton base point $x$, and
let $P^+$ denote its implicit projection operator.
\State $(\widetilde\sigma^+,\mathcal S)
\gets\Call{HybridScore}
{D^+_{\rm qry},\operatorname{diag}(W^+),L_{\rm conf},\mathcal S}$.
\Comment{Algorithm~\ref{alg:candidate_hybrid_score}}
\State $S^+\gets$ the certified diagonal from
$\widetilde\sigma^+$; record only implicit operator access to $P^+$ as in
Lemma~\ref{lem:certified_surrogate}.
\State $(\overline d^{\,2},p^\infty,p^w)
\gets\Call{CoarseDefectData}{W^+,S^+,P^+,\psi,\Delta,\zeta}$.
\Comment{Algorithm~\ref{alg:coarse_defect_data}}
\State $\overline d\gets\sqrt{\overline d^{\,2}}$.
\State $s\gets\eta_0(\overline d+\Delta/c_k)$ and
$r\gets(1-59/(48c_k))\Delta-BC^2\overline d^2/\Delta$.
\State $\mathcal C\gets\{u:N(u)\leq r,\
\|u-p^\infty\|_\infty\leq s/64\}$.
\If{$\mathcal C$ is empty or
$r+1.1s>(1-1/c_k)\Delta/\overline c_\gamma$}
  \State \Return \textnormal{\textsc{Fail}}.
\EndIf
\State $\widehat u\gets\Call{DefectProjection}{w,C,r,s,p^\infty,p^w}$.
\Comment{Algorithm~\ref{alg:defect_projection}}
\State $z\gets\exp(\log\widehat g+\widehat u)$ and
$(\widehat g^+,\mathcal S)
\gets\Call{FixedPoint}{x^+,z,E,E/2,\zeta,L_{\rm conf},\mathcal S,
\mathsf{auto}}$.
\Comment{Algorithm~\ref{alg:fresh_fixed_point}}
\State $d_{\rm obs}\gets\log w+\widehat u-\log\widehat g^+$.
\State $\chi\gets\Call{MixedBallLinearOracle}
{w,C,1.1s,\nabla\Phi(d_{\rm obs}),\epsilon_{\rm g}}$.
\Comment{Algorithm~\ref{alg:mixed_ball_linear_oracle}}
\State $w^+\gets\exp(\log w+\widehat u+\chi)$.
\State \Return $(x^+,w^+,\widehat g^+,\mathcal S)$.
\EndProcedure
\end{algorithmic}
\end{algorithm}

With the surrogate, defect certificate, and fresh score increments available,
we can state a complete predicted centering call.  The next lemma verifies the
pseudocode above and accounts for every stopping event and resource charge.
\begin{lemma}[Hybrid-score predicted centering step]
\label{lem:fresh_predicted_center_correctness}
Suppose
\[
 \begin{aligned}
 0<\delta:=\delta_t^{c_{\rm path}}(x,w)&\leq R,
 &\Delta&:=\overline\delta_t^{c_{\rm path}}(x,w),\\
 \|\log w-\log g(x)\|_\infty&\leq K,
 &\|\log\widehat g-\log g(x)\|_\infty&\leq E/2.
 \end{aligned}
\]
and let $\zeta\in(0,1/10)$.  Suppose the raw-relative initialization event
holds, and use the coupled process of
Lemma~\ref{lem:candidate_raw_score_coupling} up to the existing auxiliary
stop.  Lemma~\ref{lem:candidate_score_interface} supplies both score
consumers without conditioning on future score success: a response before
the first violation is raw-relative, and every replacement response is
exact.  Given that past, the simultaneous coarse good event has conditional
failure probability at most $\zeta$.  On that event,
Algorithm~\ref{alg:fresh_predicted_center} does not return
\textnormal{\textsc{Fail}} and returns $(x^+,w^+,\widehat g^+,\mathcal S)$ such that
\[
 \|\log\widehat g^+-\log g(x^+)\|_\infty\leq E/2,
 \qquad
 \|\log w^+-\log w\|_{w+\infty}
 \leq(1-\tfrac1{c_k})\delta_t(x,w).
\]
Conditionally on the coarse sigma-field, its residual error satisfies the
three bias and fluctuation hypotheses of
Lemma~\ref{lem:adaptive_residual_game}.  It contributes $O(E)$ hybrid
effective drift and $O(1)$ transitions to the score process.  Outside that process, one call
uses $O(L_*+q)$ permitted systems and $O(\log q)$ linear-system depth.
\end{lemma}
\begin{proof}
Lemma~\ref{lem:certified_surrogate} makes the query matrices measurable before
their fresh sketches are drawn.  Lemmas~\ref{lem:large_radius_split_predictor}
and~\ref{lem:linear_defect_projection}, including the computable projection of
Lemma~\ref{lem:computable_defect_projection}, imply that $\mathcal C$ is
nonempty and prove the literal proxy inequality tested by the algorithm,
\[
 r+1.1s\leq(1-\tfrac1{c_k})\Delta/\overline c_\gamma
 \leq(1-\tfrac1{c_k})\delta.
\]
They also give the conditional bias and fluctuation bounds required by
Lemma~\ref{lem:adaptive_residual_game}.  To connect the actual observation
to that lemma, write
\[
 d_{k-1}:=\log w-\log g(x),\qquad
 e_k:=u-\widehat u=b_k+Z_k,\qquad q_k:=d_{k-1}-Z_k.
\]
The line defining $d_{\rm obs}$ in the algorithm then gives the exact identity
\[
 d_{\rm obs}-q_k
 =\log g(x^+)-\log\widehat g^+-b_k.
\]
Since $\|\log\widehat g^+-\log g(x^+)\|_\infty\leq E/2$ and
$N(b_k)\leq s/8$, its infinity norm is at most
$E/2+s/8<E$.  Thus the observed-softmax hypothesis and the greedy correction
in Lemma~\ref{lem:adaptive_residual_game} apply to the value actually passed
by Algorithm~\ref{alg:fresh_predicted_center}.  Finally,
$N(\widehat u)\leq r$ and $N(\chi)\leq1.1s$, so the displayed movement
bound follows from the triangle inequality.

Lemma~\ref{lem:large_radius_warm_start} gives
$\|\log z-\log g(x^+)\|_\infty\leq E$.  Thus
Algorithm~\ref{alg:fresh_fixed_point}, invoked in hybrid maintained mode with $D=E$
and desired accuracy $E/2$, uses an absolute number of rounds and proves the
auxiliary-weight claim.  Lemma~\ref{lem:candidate_step_effective_drift}
accounts for all score transitions and their effective drift globally.  The coordinate predictor,
certificate, and weighted predictor cost $O(L_*+q)$ systems outside that
process, and the weighted predictor has $O(\log q)$ sequential rounds.
This proves the cost and depth claims.
\end{proof}

\subsection{Path following and resource budgets}
\label{subsec:path_following_budgets}

\begin{algorithm}[!ht]
\caption{Hybrid-score path-following phase}
\label{alg:fresh_path_follow}
\begin{algorithmic}[1]
\Procedure{PathFollow}{$x,w,\widehat g,c_{\rm path},t_0,t_1,\tau,\zeta_{\rm step},
L_{\rm conf},\mathcal S$}
\Comment{Lemma~\ref{lem:fresh_path_follow_correctness}}
\State $\beta\gets1+C\sqrt{2\overline c_1}$, $\alpha\gets R/(16c_k\beta)$, and
$t\gets t_0$.
\State $M_{\rm path}\gets\lceil A_{\rm path}(c_k\beta|\log(t_1/t_0)|/R
+c_k\log(R/\tau)+1)\rceil$, including terminal-centering calls,
$J\gets\lceil8c_k\log(R/\tau)\rceil$, and $j_{\rm path}\gets0$.
\While{$t\neq t_1$}
  \State $j_{\rm path}\gets j_{\rm path}+1$.
  \If{$j_{\rm path}>M_{\rm path}-J$}
    \State \Return \textnormal{\textsc{Fail}}.
  \EndIf
  \State $\Delta\gets\overline\delta_t^{c_{\rm path}}(x,w)$ from one
  projected-residual solve.
  \If{$\Delta=0$}
    \State $t\gets\operatorname{median}\{(1-\alpha)t,t_1,(1+\alpha)t\}$.
  \Else
    \State $Y\gets\Call{PredictedCenter}
    {x,w,\widehat g,t,c_{\rm path},\zeta_{\rm step},L_{\rm conf},\mathcal S}$.
    \Comment{Algorithm~\ref{alg:fresh_predicted_center}}
    \If{$Y=\textnormal{\textsc{Fail}}$}
      \State \Return \textnormal{\textsc{Fail}}.
    \EndIf
    \State $(x,w,\widehat g,\mathcal S)\gets Y$.
    \State $t\gets\operatorname{median}\{(1-\alpha)t,t_1,(1+\alpha)t\}$.
  \EndIf
\EndWhile
\For{$j=1,\ldots,J$}
  \State $\Delta\gets\overline\delta_{t_1}^{c_{\rm path}}(x,w)$ from one
  projected-residual solve.
  \If{$\Delta>0$}
    \State $Y\gets\Call{PredictedCenter}
    {x,w,\widehat g,t_1,c_{\rm path},\zeta_{\rm step},L_{\rm conf},\mathcal S}$.
    \Comment{Algorithm~\ref{alg:fresh_predicted_center}}
    \If{$Y=\textnormal{\textsc{Fail}}$}
      \State \Return \textnormal{\textsc{Fail}}.
    \EndIf
    \State $(x,w,\widehat g,\mathcal S)\gets Y$.
  \EndIf
\EndFor
\State \Return $(x,w,\widehat g,\mathcal S)$.
\EndProcedure
\end{algorithmic}
\end{algorithm}

The centering call is iterated while the path parameter moves toward
its target.  The following lemma proves that the phase routine preserves all
invariants and gives its total number of calls.
\begin{lemma}[Fresh path-following phase]
\label{lem:fresh_path_follow_correctness}
Suppose $t_0,t_1>0$, $0<\tau\leq R$, and the input satisfies
\[
 \delta_{t_0}^{c_{\rm path}}(x,w)\leq R,
 \qquad
 \|\log w-\log g(x)\|_\infty\leq K,
 \qquad
 \|\log\widehat g-\log g(x)\|_\infty\leq E/2,
\]
together with the initial-potential hypothesis of
Lemma~\ref{lem:adaptive_residual_game}; for a segment immediately continuing
a preceding call inside the same attempt, it is enough that both segments
use the same stopped residual game.  Assume that the executed prefix belongs
to an epoch of at most $B_{\rm ep}$ calls, with
$L_{\rm ch}\geq\log(A_0B_{\rm ep}m\beta)$ and
$\zeta_{\rm step}\leq(100B_{\rm ep})^{-1}$.  On the raw-relative initialization
event, analyze the coupled hybrid score process of
Lemma~\ref{lem:candidate_raw_score_coupling}.  Let $\mathcal E_{\rm sc}$ be
the event that no exact score replacement occurs before completion of this
phase, and let $\mathcal E_{\rm coarse}$ be the intersection of the
per-call coarse events and $\mathcal E_{\rm ch}$ the stopped tracking event
of Lemma~\ref{lem:adaptive_residual_game}, both defined without conditioning
on future score success.  On $\mathcal E_{\rm coarse}\cap
\mathcal E_{\rm ch}\cap\mathcal E_{\rm sc}$,
Algorithm~\ref{alg:fresh_path_follow} does not return
\textnormal{\textsc{Fail}} and returns an
iterate at $t_1$ with
\[
 \delta_{t_1}(x,w)\leq\tau,
\]
while preserving the tracking and auxiliary-weight invariants.  It makes
\[
 O(\frac{c_k\beta}{R}|\log(t_1/t_0)|
   +c_k\log(R/\tau)+1)
\]
calls to Algorithm~\ref{alg:fresh_predicted_center}.
\end{lemma}
\begin{proof}
On the intersection of the coarse, chasing, and score good events,
Lemma~\ref{lem:fresh_predicted_center_correctness}
supplies the movement and tracking hypotheses of
Lemma~\ref{lem:centrality_path_composition}.  Thus every nonzero-centrality
iteration contracts at fixed $t$ and permits the signed median update of
magnitude at most $\alpha$.  If centrality is zero, the proof of the signed
path-update part of Lemma~\ref{lem:centrality_path_composition}, without its
contraction term, gives centrality at most $\alpha\beta\leq R$ after the same
median update.  A full update changes $|\log t|$ by $\Omega(\alpha)$, and the
last median update reaches $t_1$ exactly.  This proves the first term in the
call count.

At $t_1$, each nonvacuous centering call contracts centrality by
$1-1/(4c_k)$.  Therefore the chosen value of $J$ reduces centrality from at
most $R$ to at most $\tau$.  This contraction remains valid even when
$\tau<E$: the hypotheses of
Lemmas~\ref{lem:large_radius_retained_movement},
\ref{lem:linear_defect_projection}, and~\ref{lem:adaptive_residual_game}
impose no lower bound on centrality relative to $E$, and
Lemma~\ref{lem:linear_defect_projection} keeps the residual radius below
the fixed observation scale $E$.  Lemma~\ref{lem:fresh_predicted_center_correctness}
also preserves the auxiliary invariant, while the stochastic game preserves
tracking simultaneously.  Taking $A_{\rm path}$ sufficiently large makes the
a priori call bound dominate both loops.  Every zero test in the pseudocode
is therefore a test of the computable projected-residual proxy, not of the
mixed-norm minimization defining exact centrality.  The pseudocode explicitly
charges one projected-residual solve for that test.  There is at most one such
solve per loop iteration, so these solves are absorbed by the displayed call
count and the $O(L_*+q)$ non-score work per iteration.  This proves the lemma.
\end{proof}

A checkpoint can interrupt terminal centering, so the stopping rule must be
compatible with normalization.  The next lemma specifies how to restart
the first terminal loop and execute the final fixed-weight tail only once.

\begin{lemma}[Checkpoint-safe terminal completion]
\label{lem:checkpoint_terminal_completion}
In the checkpointed two-phase execution, restart the first phase's terminal
counter whenever a checkpoint normalization occurs inside that loop.
Choose $B_{\rm ep}>4J_{\rm sw}$, where
$J_{\rm sw}=\lceil8c_k\log(R/\tau_{\rm sw})\rceil$.
Omit the second phase's terminal loop from the stochastic stream.
After that stream returns its full state at $t_{\rm out}$, run
\textnormal{\textsc{FixedWeightTail}} once: keep $w,t$ fixed and take
projected Newton steps until
$\overline\delta_{t_{\rm out}}^c(x,w)\leq\tau_{\rm out}$.
Test feasibility and interiority before each barrier evaluation and announce
the deterministic cap
$J_{\rm tail}:=\lceil A_{\rm tail}(1+\log(2+\log(R/\tau_{\rm out})))\rceil$;
return \textnormal{\textsc{Fail}} if the cap is reached without success.
On the good tracking event these conventions preserve the outer reduction.
The successful stream gains at most $J_{\rm sw}=O(c_k\log\ell)$
instructions, and the once-only tail uses $O(\kappa+\ell)$ systems.
\end{lemma}
\begin{proof}
The first terminal loop contracts centrality by $1-1/(4c_k)$ at each
nonvacuous call.  Normalization need not preserve an already smaller
residual, so its earlier counter cannot be reused to certify consecutive
contractions.  Resetting the counter restores this implication.  After
such an accepted checkpoint the entire loop fits before the next checkpoint,
since $J_{\rm sw}<B_{\rm ep}$.  There is therefore at most one successful
reset.  The switch threshold is independent of $U,\epsilon$, giving the
stated inverse-polylogarithmic terminal length.

The final tail starts with centrality at most $R$ and tracking error at
most $K/4$.  The fixed-weight Newton estimate and the proxy comparison give
$\overline\delta_{j+1}\leq8\overline\delta_j^2$.
For sufficiently small $R$, the total normalized movement is $O(R)$;
target stability changes $\log g$ by at most $O(R)\leq K/4$ after
the final radius choice.  Induction therefore preserves the tracking band,
interiority, and quadratic contraction.  The displayed cap follows by
iterating the quadratic bound.  At its endpoint the terminal-distance and
duality-gap argument of Lemma~\ref{lem:outer_reduction} applies.
This deterministic tail makes no score or predictor calls and is outside
the retry loop, so its charge is paid only once.
\end{proof}

\begin{algorithm}[!ht]
\caption{Deterministic normalization at a path-epoch boundary}
\label{alg:epoch_normalize}
\begin{algorithmic}[1]
\Procedure{EpochNormalize}{$x,w,t,c_{\rm path}$}
\Comment{Lemma~\ref{lem:epoch_normalization}}
\State Using exact score diagonals, run fixed-weight projected Newton steps
until the computable proxy satisfies
$\overline\delta_t^{c_{\rm path}}(x,w)\leq\epsilon_{\rm rec}$.
\State $(\widetilde g,\mathcal S_{\rm tmp})\gets
\Call{FixedPoint}{x,\mathbf1,q,\tau_{\rm norm},1/2,L_{\rm conf},
\varnothing,\mathsf{exact}}$.
\While{$e^{\mu\tau_{\rm norm}}
\Phi(\log w-\log\widetilde g)>M_{\rm ch}/2$}
  \State $\widetilde d\gets\log w-\log\widetilde g$ and
  $\chi\gets\Call{MixedBallLinearOracle}
  {w,C,s_{\rm norm},\nabla\Phi(\widetilde d),\epsilon_{\rm norm}}$.
  \State $w\gets w\odot e^\chi$ and run fixed-weight projected Newton
  steps until
  $\overline\delta_t^{c_{\rm path}}(x,w)\leq\epsilon_{\rm rec}$ at the
  same $t,c_{\rm path}$.
  \State $(\widetilde g,\mathcal S_{\rm tmp})\gets
  \Call{FixedPoint}{x,\mathbf1,q,\tau_{\rm norm},1/2,L_{\rm conf},
  \varnothing,\mathsf{exact}}$.
\EndWhile
\State $(\widehat g,\mathcal S)\gets
\Call{FixedPoint}{x,\mathbf1,q,E/8,1/2,L_{\rm conf},
\varnothing,\mathsf{exact}}$; retain the exact score vector of its actual
final query in $\mathcal S$.
\State \Return $(x,w,\widehat g,\mathcal S)$.
\EndProcedure
\end{algorithmic}
\end{algorithm}

Checkpointing requires all numerical scales to remain polynomially bounded
between exact restarts.  The next lemma shows that the normalization routine
achieves this without changing the asymptotic path length.
\begin{lemma}[Deterministic epoch normalization]
\label{lem:epoch_normalization}
There are absolute constants $c_{\rm norm}>0$, $D_{\rm norm}\geq1$, and an
integer $d_{\rm norm}\geq1$ with the following property, uniformly in every
sufficiently large $A_L$.  Define
\[
\begin{aligned}
 s_{\rm norm}&:=c_{\rm norm}E/c_k,
 &\epsilon_{\rm norm}&:=1/(16c_k),\\
 \tau_{\rm norm}&:=\frac{s_{\rm norm}}{2^{14}c_k\beta},
 &\epsilon_{\rm rec}&:=\min\{2^{-22}\ell^{-3},
                 \tau_{\rm norm}/2^{10}\}.
\end{aligned}
\]
Suppose that $x$ is feasible and interior,
\[
 \delta_t^{c_{\rm path}}(x,w)\leq R,
 \qquad
 \|\log w-\log g(x)\|_\infty\leq K/2.
\]
Then Algorithm~\ref{alg:epoch_normalize} is deterministic, preserves
feasibility and the path parameter, and returns a state satisfying
\[
 \delta_t^{c_{\rm path}}(x,w)\leq\epsilon_{\rm rec},\qquad
 \Phi(\log w-\log g(x))\leq M_{\rm ch}/2,\qquad
 \|\log w-\log g(x)\|_\infty\leq K/4,
\]
together with the stated auxiliary-weight and exact raw-score base
invariants.  Each of its permitted-system, sequential-round, and exact
scalar-operation counts is at most
\begin{equation}\label{eq:normalization_uniform_cost}
 D_{\rm norm}(1+A_L)^{d_{\rm norm}}(4m)^{d_{\rm norm}}.
\end{equation}
The constants $D_{\rm norm},d_{\rm norm}$ do not depend on $A_L$.
The joint choice in Lemma~\ref{lem:candidate_hybrid_constant_hierarchy}
makes Eq.~\eqref{eq:normalization_uniform_cost} at most
$P_{\rm norm}:=(4m)^{A_{\rm norm}}$.
\end{lemma}
\begin{proof}
We first analyze exact fixed-weight centers; the finite stopping tolerance is
inserted at the end.  For a positive weight vector $\widetilde w$, let
$x(\widetilde w)$ be the minimizer of the fixed-weight barrier objective at
$t,c_{\rm path}$.  Its existence and the convergence of projected Newton
iteration follow as in Lemma~\ref{lem:fixed_weight_terminal_distance}; from
an input of centrality at most $R$ the total normalized displacement is
$O(R)$.  Target stability and the choice of $c_R$ therefore show that the
initial recentering changes $\log g$ by at most $K/16$.  In particular the
recentered state remains strictly inside the tracking band.

Fix a recentered state and a direction $\chi$ with
$N_w(\chi)\leq s_{\rm norm}$.  Along
$w_\theta=w\odot e^{\theta\chi}$ and
$x_\theta=x(w_\theta)$, differentiate the fixed-weight stationarity
condition.  With
\[
 a_i(x):=\frac{\phi_i'(x_i)}{\sqrt{\phi_i''(x_i)}},
 \qquad |a_i(x)|\leq1,
\]
the normalized central response is
\[
 h_\theta:=\sqrt{\Phi''(x_\theta)}\,\dot x_\theta
 =-P_{x_\theta,w_\theta}
       (\chi\odot a(x_\theta)).
\]
Lemma~\ref{lem:pythagorean_projector}, applied along this segment, gives
$N_{w_\theta}(h_\theta)\leq c_\gamma N_{w_\theta}(\chi)$.
At an exact center the fixed-weight centrality is zero, so
Lemma~\ref{lem:sharper_parameters} applies pointwise as a degenerate Newton
segment whenever the strict tracking band holds.  On any such prefix it
gives
\[
 N_{w_\theta}(\tfrac{\d}{\d\theta}\log g(x_\theta))
 \leq(1-\tfrac{15}{8c_k})c_\gamma
       N_{w_\theta}(\chi).
\]
Define
$\lambda:=(1-15/(8c_k))c_\gamma$.  Since the weighted component of
$N_{w_\theta}$ differs from that of $N_w$ by at most
$e^{s_{\rm norm}/2}$, integration gives, for
\[
 u:=\log g(x(w\odot e^\chi))-\log g(x(w))
\]
the bound $N_w(u)\leq\lambda e^{s_{\rm norm}}N_w(\chi)$.
Because $c_\gamma\leq1+1/(128c_k)$, decreasing $c_{\rm norm}$ makes
$\lambda e^{s_{\rm norm}}\leq1-3/(2c_k)$, and hence
\begin{equation}\label{eq:epoch_response_contraction}
 N_w(u)\leq(1-\tfrac{3}{2c_k})N_w(\chi).
\end{equation}
Thus the exact centered update changes
$d:=\log w-\log g(x)$ by $d^+=d+\chi-u$.

Let $\widetilde d=\log w-\log\widetilde g$, so that
$\|\widetilde d-d\|_\infty\leq\tau_{\rm norm}$.  The support guarantee of
Algorithm~\ref{alg:mixed_ball_linear_oracle},
Eq.~\eqref{eq:epoch_response_contraction}, duality, and the coordinate estimate
\[
 |\phi'(\widetilde d_i)-\phi'(d_i)|
 \leq(e^{\mu\tau_{\rm norm}}-1)
       (|\phi'(d_i)|+2\mu)
\]
show, for $a=\nabla\Phi(d)$, that evaluating the oracle at the certified
residual changes the exact support calculation by at most
\[
 O(\mu\tau_{\rm norm})s_{\rm norm}N_w^*(a)
 +O(\mu^2m\tau_{\rm norm}s_{\rm norm}).
\]
To pass from exact centers to the finite stopping rule, let $x^\circ(w)$
denote the exact fixed-weight center.  The unconditional lower side of
Eq.~\eqref{eq:computable_centrality} gives
$\delta\leq\overline\delta\leq\epsilon_{\rm rec}$.  The strengthened
fixed-weight distance bound of
Lemma~\ref{lem:fixed_weight_terminal_distance}, self-concordant transport,
and the target-Jacobian estimate give, at each endpoint,
\[
 N_w(\log g(x)-\log g(x^\circ(w)))
 \leq32c_\gamma\epsilon_{\rm rec}.
\]
Consequently
\[
 N_w(
 [\log g(x^+)-\log g(x)]-
 [\log g(x^\circ(w e^\chi))-\log g(x^\circ(w))]
 )
 \leq64c_\gamma\epsilon_{\rm rec}.
\]
Its contribution to the support calculation is at most
$64c_\gamma\epsilon_{\rm rec}N_w^*(a)$ and is absorbed because
$\epsilon_{\rm rec}\leq\tau_{\rm norm}/2^{10}$.  Hence
\[
 \langle a,\chi-u\rangle
 \leq-\frac{3s_{\rm norm}}{4c_k}N_w^*(a)
       +O(\mu^2m\tau_{\rm norm}s_{\rm norm}).
\]
Taylor's theorem, $\phi''=\mu^2\phi$, and the same mixed-norm quadratic
estimate used in Lemma~\ref{lem:residual_chasing} give
\[
 \Phi(d^+)\leq\Phi(d)
 -\frac{s_{\rm norm}}{2c_k}N_w^*(a)
 +C_{\rm norm}\mu^2m s_{\rm norm}^2.
\]
The gradient-dependent Taylor remainder is absorbed by the preceding linear
term because
\[
 \mu s_{\rm norm}=c_{\rm norm}/(100c_k).
\]
This is the reason for using the smaller normalization radius.  The additive
term $C_{\rm norm}\mu^2m s_{\rm norm}^2$ is absorbed only above the threshold
fixed next.  Moreover
\[
 N_w^*(a)\geq\frac1\beta\|a\|_1,
 \qquad
 \|a\|_1\geq\mu(\Phi(d)-2m).
\]
The choice
$\tau_{\rm norm}=s_{\rm norm}/(2^{14}c_k\beta)$ absorbs the certified-gradient
and finite-centering perturbations as well.  Taking $A_{\rm ch}$ large after
the absolute Taylor constants are fixed, every iteration with
$\Phi(d)>M_{\rm ch}/4$, where $M_{\rm ch}=A_{\rm ch}m\beta$, therefore
decreases $\Phi$ by a factor
$1-\Omega((c_k^2\beta)^{-1})$.  Consequently at most
\[
 O(c_k^2\beta\log(2+\Phi(d_{\rm start})/M_{\rm ch}))
\]
normalization iterations are executed.  If the outward loop test fails,
then
$\Phi(d)\leq e^{\mu\tau_{\rm norm}}\Phi(\widetilde d)\leq M_{\rm ch}/2$.
If it has not yet failed, the true potential is larger than
$e^{-2\mu\tau_{\rm norm}}M_{\rm ch}/2>M_{\rm ch}/4$, so the descent estimate
continues to apply.

At a certified epoch endpoint the starting tracking error is at most $K/2$.
The initial fixed-weight recentering moves the target by $O(R)\leq K/16$,
so
\[
 \log\Phi(d_{\rm start})
 \leq\log(2m)+9\mu K/16
 =\log(2m)+9A_EA_L\ell/1600,
\]
where the first step follows from
$\Phi(d)\leq2m e^{\mu\|d\|_\infty}$ and
$\|d_{\rm start}\|_\infty\leq9K/16$, and the second step follows from
$\mu=1/(100E)$ and $E=K/(A_EA_L\ell)$.
Thus the number of normalization iterations is
$O(c_k^2\beta(1+A_L)\ell)$, with an implicit constant independent of
$A_L$.  We close the tracking hypothesis used above by a first-exit
argument.  Suppose that the exact-center
interpolation has a first point at which $\|d\|_\infty=K$.  On the preceding
prefix the pointwise invocation of Lemma~\ref{lem:sharper_parameters} is
valid, so Eq.~\eqref{eq:epoch_response_contraction} and the potential descent
derived from it hold at every completed update on that prefix.  Hence
$\Phi$ is nonincreasing there.  More quantitatively, the initial margin and
$\|d\|_\infty\leq\mu^{-1}\log\Phi(d)$ give, after increasing $A_E$,
$\|d\|_\infty\leq5K/8$ at every completed exact-center update.  During one
interpolation,
\[
 \|d_\theta-d_0\|_\infty
 \leq N_w(\theta\chi)
      +N_w(\log g(x_\theta)-\log g(x_0))
 \leq(2+o(1))s_{\rm norm}<K/8.
\]
The two finite recentering errors contribute at most
$64c_\gamma\epsilon_{\rm rec}<K/16$.  The intervening fixed-weight Newton
path has no larger total target displacement: apply the Newton-series
estimate in Lemma~\ref{lem:fixed_weight_terminal_distance} to each of its
prefixes.  Thus the putative first exit is still strictly below $K$, a
contradiction.  This justifies the operator estimate on the entire
normalization path without circularity.  Finally the stopping
threshold, rather than the starting potential, gives
$\|d\|_\infty\leq K/4$, because $\mu^{-1}=100E$,
$L_{\rm ch}=A_L\ell$, and $E=K/(A_EL_{\rm ch})$.

It remains to justify finite execution.  After changing the weight by
$s_{\rm norm}<1/10$, the weight-response inequality of
Lemma~\ref{lem:ls_path_facts} gives fixed-weight centrality
$O(s_{\rm norm})$.  Projected Newton iteration reduces this quadratically,
so $O(\log\log(2+s_{\rm norm}/\epsilon_{\rm rec}))$ systems suffice to
reach $\epsilon_{\rm rec}$.  The exact-score branch of
Algorithm~\ref{alg:fresh_fixed_point} computes each required certified
regularized Lewis vector in
$O(n\log(2+q/\epsilon_{\rm rec}))$ permitted systems.  To keep the
dependence on $A_L$ explicit, the tolerance definitions give
\[
 \epsilon_{\rm rec}^{-1}
 \leq2^{22}\ell^3+
 \frac{2^{31}A_E}{c_{\rm norm}}A_L\ell c_k^{5/2}\beta.
\]
Here $n\leq m$, $c_k\leq8\ell$, and
$\beta=1+32\sqrt{c_kn}$, so the reciprocal tolerances and the number
of iterations are bounded by fixed polynomials in $m$ and $1+A_L$.
The initial recentering, each subsequent recentering, and each exact
fixed-point computation have a number of rounds polynomial in the
logarithms of these reciprocal tolerances.  The mixed-ball oracle has
the same property, since the tracking band gives
$n/(2m)\leq w_i\leq3$.  Each round uses polynomially many scalar
operations in $m$ and at most polynomially many permitted systems.
Multiplying these bounds proves Eq.~\eqref{eq:normalization_uniform_cost}
with $D_{\rm norm},d_{\rm norm}$ independent of $A_L$.  The finite
recentering errors obey $\epsilon_{\rm rec}\leq\tau_{\rm norm}/2^{10}$
by definition, not by a later choice of a work exponent.
The final exact-score response supplies a
zero-error raw relative base, so it also satisfies every score-state
convention.  All Newton directions lie in $\ker(A^\top)$ and all iterates
remain in the barrier domain, proving feasibility and the remaining claims.
\end{proof}

The preceding lemmas introduced several absolute constants whose choices must
be mutually compatible and independent of the input.  The next lemma fixes
the structural subset without circular dependence.
\begin{lemma}[Compatible structural constants]
\label{lem:constant_hierarchy}
All structural absolute constants used by
Algorithms~\ref{alg:defect_projection}--\ref{alg:fresh_path_follow} can be
fixed once, independently of $m,n,U,\epsilon$.  The remaining scale,
hybrid-provider, outer-switch, and resource-cap constants are then fixed,
without changing these structural choices, in
Lemma~\ref{lem:candidate_hybrid_constant_hierarchy}.
\end{lemma}
\begin{proof}
First fix the numerical constants in the
Lee--Sidford facts, the choices $\epsilon_0=1/10$ and $\nu=1/5$, and the
structural definitions of $p,c_k,C,K$.  Next fix the constant-depth
stochastic-chain, medoid, and scalar-median parameters; this fixes
$a_2,a_3,a_4,a_7,A_d$ and the absolute equivalence constants.  Choose
$B>0$ so that $BA_d\leq\min\{1/12,\gamma/2\}$, then choose $\eta_0>0$ so
that Eq.~\eqref{eq:constant_eta_condition} below holds.  Choose
$\epsilon_1,\epsilon_{\rm P},\epsilon_{\rm g},c_0>0$ in that order, small
enough that
\begin{equation}\label{eq:constant_eta_condition}
 1.1\eta_0+\frac{(1.1\eta_0)^2}{512B}\leq\frac16,
 \qquad 16\epsilon_1/\eta_0+\epsilon_{\rm P}\leq1/16.
\end{equation}
and so that the fixed support-oracle and surrogate perturbation inequalities
hold.

After those choices, decrease $r_0$ to satisfy every finite-segment estimate,
including $r_0\leq\epsilon_1/(8a_2)$, and fix $A_{\rm q}$ large enough to
dominate the deterministic companion-error constants.  Every remaining
requirement encountered above is now a finite monotone upper bound on
$c_R,E,R,c_{\rm anc}$ or a lower bound on a sample, path, horizon, or budget
constant.  Lemma~\ref{lem:candidate_hybrid_constant_hierarchy} resolves
those bounds in an order that never changes a previously frozen radius.
\end{proof}

\section{Score maintenance and certified checkpoints}
\label{sec:score_maintenance_checkpoints}

Section~\ref{subsec:exact_diagonal_self_transport} derives the exact
coordinatewise transport identity and bounds the nonlocal endpoint error.
Section~\ref{subsec:hybrid_maintenance_stopping} combines transported and
ordinary transitions with block stopping to maintain all scores uniformly.
Finally, Section~\ref{subsec:hybrid_path_compatibility} verifies the companion
bands, initialization, and coupling needed by the path algorithm.
The final resource caps are assembled in Lemma~\ref{lem:four_log_ledger}.

Lemma~\ref{lem:ls_path_facts} records the Lee--Sidford Newton contraction,
normalized movement, response to a weight change, signed path-parameter
change, and weighted duality-gap estimate used below.  The outer reduction
also uses their interval-barrier derivative bounds, which are built into the
definition of $U$ in Section~\ref{sec:intro}; the cost-switch estimate and the
terminal-distance conversion are derived explicitly in
Lemma~\ref{lem:outer_reduction}.  Every other ingredient used for the solve
bound---well-posed regularized weights, the computable centrality proxy,
linear-defect projection, defect-adaptive predictor, observed softmax step,
hybrid relative score maintenance, residual chasing, and unconditional
initialization---is proved in the preceding sections and below.  The hybrid
provider transports the own-row response exactly and sketches only the
nonlocal residual.  Every random sketch is drawn after its query endpoints
are fixed.

Algorithms~\ref{alg:fresh_fixed_point}--\ref{alg:fresh_path_follow} implement
fixed-point computation, one hybrid-score centering step, and one
path phase.  The final block applies the two-phase outer reduction.
In the budgeted execution, every descendant call is made through the same
resource wrapper.  Inside Algorithm~\ref{alg:checkpointed_path}, a descendant
\textnormal{\textsc{Fail}} or an attempt-local-cap failure is caught and
turns the current attempt into a rejection.  A shared-global-cap failure is
not caught: it propagates immediately to Algorithm~\ref{alg:lp_solve_fresh}.
Outside this retry boundary every descendant failure likewise propagates
before the enclosing procedure executes its next line, and tuple assignment
is performed only on a successful return.

Regard the two path loops, the first terminal-centering loop, and the
cost-vector switch as one resumable instruction stream $\mathfrak P$.
Use Lemma~\ref{lem:checkpoint_terminal_completion}: the first terminal
counter restarts after normalization, and the final deterministic tail is
outside this stream and the retry loop.
Its state records the active phase, $t$, both deterministic loop counters,
and $(x,w,\widehat g,\mathcal S)$.  One \emph{atomic path instruction} is a
complete iteration of either loop of
Algorithm~\ref{alg:fresh_path_follow}: it includes the projected-residual
test, the optional centering call, and, in the path loop, the ensuing median
update of $t$.  The cost-vector switch is atomic as well.  Thus a checkpoint
never splits a centering call from the path update licensed by it.  We count
every atomic loop iteration, including a zero-centrality iteration, against
the epoch and global path caps.  Enlarging the absolute constant $A_N$
absorbs this equivalent count.  Pausing and resuming only at these boundaries
preserves each complete path update.  The terminal-counter reset and
once-only final tail follow Lemma~\ref{lem:checkpoint_terminal_completion}.

For $1\leq b\leq B_{\rm ep}$ define
$P_{\rm norm}:=(4m)^{A_{\rm norm}}$ and fix a sufficiently large absolute
$A_{\rm loc}$.  Every attempt uses the deterministic local caps
\begin{equation}\label{eq:attempt_resource_caps}
\begin{aligned}
 M_{\rm att}(b)&:=A_{\rm loc}
  (b+q\ell+1+\mathbf1_{\{b=B_{\rm ep}\}}P_{\rm norm}),\\
 S_{\rm att}(b)&:=A_{\rm loc}(
  b(E^{-1}\ell+\ell+q+1)+E^{-2}\ell
  +\mathbf1_{\{b=B_{\rm ep}\}}P_{\rm norm}),\\
 D_{\rm att}(b)&:=A_{\rm loc}(
  b(1+\log q)+1
  +\mathbf1_{\{b=B_{\rm ep}\}}P_{\rm norm}),\\
 W_{\rm att}(b)&:=A_{\rm loc}(\nnz(A)+m)
  (S_{\rm att}(b)+b\ell^2
  +\mathbf1_{\{b=B_{\rm ep}\}}P_{\rm norm}).
\end{aligned}
\end{equation}
The first cap counts score queries and exact-score queries; the other three
count permitted systems, sequential system rounds, and scalar work.  A
routine announces its full charge before work or randomness, and an attempt
is rejected before exceeding any local cap.

\begin{algorithm}[!ht]
\caption{Checkpointed execution of the two path phases}
\label{alg:checkpointed_path}
\small
\begin{algorithmic}[1]
\Procedure{CheckpointedPath}{$\mathfrak P,N_{\rm cap}$}
\Comment{Lemma~\ref{lem:checkpointed_path}}
\State $B_{\rm ep}\gets(4m)^{A_{\rm ep}}$,
$M_{\rm ep}\gets A_{\rm call}(B_{\rm ep}+q\ell+1)$,
$L_{\rm conf}\gets\log(40mM_{\rm ep})$, and
$\zeta_{\rm step}\gets(100B_{\rm ep})^{-1}$,
$\epsilon_{\rm chk}\gets s_{\rm norm}/(2^{14}c_k\beta)$.
\State $J_{\rm ep}\gets\lceil N_{\rm cap}/B_{\rm ep}\rceil+1$,
$A_{\rm att}\gets4J_{\rm ep}$, $a\gets0$, and $r\gets0$.
\While{$\mathfrak P$ is not complete and $a<A_{\rm att}$}
  \State $b\gets\min\{B_{\rm ep},N_{\rm cap}-r\}$; if $b=0$, return
  \textnormal{\textsc{Fail}}.
  \State Announce the four attempt-local caps in
  Eq.~\eqref{eq:attempt_resource_caps}.
  \State Charge every descendant operation to its local and shared global
  counters.
  \State Save the mathematical path and score state $Q$ of $\mathfrak P$;
  set $j\gets0$, $a\gets a+1$, $\mathrm{reject}\gets\textnormal{false}$,
  start fresh score and residual-game epochs, and begin a fresh attempt.
  Resource counters are never rolled back.
  \While{$j<b$ and $\mathfrak P$ is not complete and
  $\mathrm{reject}=\textnormal{false}$}
    \State Execute the next atomic instruction of $\mathfrak P$, using
    $\zeta_{\rm step}$ in Algorithm~\ref{alg:fresh_predicted_center}.  On a
    descendant or attempt-local \textnormal{\textsc{Fail}}, set
    $\mathrm{reject}\gets\textnormal{true}$; on a shared-global-cap
    \textnormal{\textsc{Fail}}, return \textnormal{\textsc{Fail}}.
    \If{the instruction completed} \State $j\gets j+1$. \EndIf
  \EndWhile
  \If{$\mathrm{reject}=\textnormal{true}$}
    \State Restore $Q$ and \textbf{continue} to the next attempt.
  \EndIf
  \If{$\mathfrak P$ is complete} \State \Return its full mathematical state. \EndIf
  \If{$r+j=N_{\rm cap}$}
    \State Restore $Q$ and \textbf{continue} to the next attempt.
  \EndIf
  \If{$w\not>0$, the point violates $A^\top x=b$, is not interior, or has a
  path/phase descriptor inconsistent with $\mathfrak P$}
    \State Restore $Q$ and \textbf{continue} to the next attempt.
  \EndIf
  \State $(g_{\rm chk},\mathcal S_{\rm chk})\gets
  \Call{FixedPoint}{x,\mathbf1,q,\epsilon_{\rm chk},1/2,L_{\rm conf},
  \varnothing,\mathsf{exact}}$, and compute
  $\overline\delta_t^{c_{\rm path}}(x,w)$.
  \If{$\overline\delta_t^{c_{\rm path}}(x,w)>R-R/(32c_k)$ or
  $\|\log w-\log g_{\rm chk}\|_\infty+\epsilon_{\rm chk}>K/2$}
    \State Restore $Q$ and \textbf{continue} to the next attempt.
  \Else
    \State $(x,w,\widehat g,\mathcal S)
    \gets\Call{EpochNormalize}{x,w,t,c_{\rm path}}$.
    \State $(g_{\rm chk},\mathcal S_{\rm chk})\gets
    \Call{FixedPoint}{x,\mathbf1,q,\epsilon_{\rm chk},1/2,L_{\rm conf},
    \varnothing,\mathsf{exact}}$ and certify
    $\overline\delta_t^{c_{\rm path}}(x,w)\leq R/4$,
    $\Phi(\log w-\log g_{\rm chk})\leq
      e^{-\mu\epsilon_{\rm chk}}M_{\rm ch}$, and
    $\|\log\widehat g-\log g_{\rm chk}\|_\infty\leq E/4$;
    on failure restore $Q$ and \textbf{continue}; otherwise accept this
    epoch and start the next attempt from the normalized state.
    \State $r\gets r+j$; if currently inside the first terminal loop,
    restart that terminal counter at zero.
  \EndIf
\EndWhile
\State \Return \textnormal{\textsc{Fail}}.
\EndProcedure
\end{algorithmic}
\end{algorithm}

After normalization and the structural constants are fixed, the path can be
partitioned into bounded attempts with deterministic caps.  The next lemma
assembles these attempts and controls their total failure probability.
\begin{lemma}[Checkpointed polynomial epochs]
\label{lem:checkpointed_path}
Use the joint constant choice of
Lemma~\ref{lem:candidate_hybrid_constant_hierarchy}, with
$A_{\rm ep}>A_{\rm norm}+10$, and define
$B_{\rm ep}=(4m)^{A_{\rm ep}}$.  Run all stochastic routines inside one
attempt with
\[
 L_{\rm ch}=A_L\ell,\quad
 M_{\rm ep}=A_{\rm call}(B_{\rm ep}+q\ell+1),\quad
 L_{\rm conf}=\log(40mM_{\rm ep}),\quad
 \zeta_{\rm step}=(100B_{\rm ep})^{-1}.
\]
Suppose an attempted epoch starts either on the raw-relative initialization
event or from an accepted output of Lemma~\ref{lem:epoch_normalization}.
Conditional on its entire starting state, a full $B_{\rm ep}$-instruction
attempt that does not finish the stream is accepted with probability at
least $4/5$; if it finishes the stream, the same good event implies a correct
output.  Algorithm~\ref{alg:checkpointed_path},
capped at $4J_{\rm ep}$ attempts where
$J_{\rm ep}=\lceil N_{\rm cap}/B_{\rm ep}\rceil+1$, halts on every outcome
and completes the two path phases with constant probability.  Its total
systems, system depth, and scalar work are within an absolute constant of
the uncheckpointed bounds with $L_*=\ell$.
\end{lemma}
\begin{proof}
Use Lemma~\ref{lem:checkpoint_terminal_completion} for terminal execution.
Correctness of a returned full state below means correctness of its output
after the once-only deterministic tail.  Its cost is absorbed by the
displayed bounds.  Enlarge $A_N$ to include the possible short terminal
restart and choose $B_{\rm ep}>4J_{\rm sw}$.
At an attempt boundary the algorithm snapshots only the mathematical state:
the active phase, $t$, loop counters, $x,w,\widehat g,\mathcal S$, and the
local residual-game state.  Shared resource counters are monotone and are
never rolled back.  A rejected attempt restores the snapshot, resets its
local stochastic-game and score-block counters, and draws fresh randomness;
all work already done remains charged.  Immediately after forming a proposed
Newton point, membership in $\Omega^\circ$ is tested before evaluating any
barrier quantity, and a failed test rejects the attempt.

On the local good event, the $B_{\rm ep}$ coarse events fail with total
probability at most $1/100$, the stopped residual game fails with probability
at most $1/20$, and raw score transport fails with probability at most
$1/40$.  Thus their intersection has conditional probability at least
$4/5$.  The enlarged choice of $A_E$ makes the last display in the proof of
Lemma~\ref{lem:residual_chasing} at most $K/4$.  On this event every nonzero
path instruction ends, after its path-parameter update, with
$\delta\leq R-R/(8c_k)$.  Therefore
\[
 \overline\delta
 \leq\overline c_\gamma(1-\tfrac1{8c_k})R
 \leq R-\frac{R}{32c_k}.
\]
The zero-centrality instruction has
$\overline\delta\leq\overline c_\gamma R/(16c_k)$ and obeys the same test.
An exact-score approximation
$g_{\rm chk}=e^{\pm\epsilon_{\rm chk}}g(x)$ passes
$\|\log w-\log g_{\rm chk}\|_\infty+\epsilon_{\rm chk}\leq K/2$.
Affine feasibility, interiority, weight positivity, and the active phase and
$t$ descriptor are checked exactly.

For an attempt of length cap $b$, the same good event gives $O(b)$ score
transitions and total effective drift $O(bE)$.  Hence the hybrid provider
uses
\[
 O(b+bE^{-1}\ell+E^{-2}\ell)
\]
systems, while all non-score operations use $O(b(\ell+q))$ systems,
$O(b(1+\log q))$ system rounds, and
$O((\nnz(A)+m)b\ell^2)$ scalar work.  If $b=B_{\rm ep}$, the two cold
checkpoint computations and epoch normalization are bounded by a constant
multiple of $P_{\rm norm}$.  Thus, after fixing $A_{\rm loc}$, a good
attempt never reaches any cap in Eq.~\eqref{eq:attempt_resource_caps}.
On an arbitrary bad attempt, those same caps force rejection before a larger
charge is incurred; no drift assertion is used for rejected work.

Lemma~\ref{lem:epoch_normalization} then returns the complete fresh-start
invariants, including computable centrality at most $\epsilon_{\rm rec}$,
potential at most $M_{\rm ch}/2$, an $E/8$-accurate auxiliary vector, and an
exact raw-score base.  The algorithm recomputes outward certificates after
normalization and rejects on any failure.  Hence the local good event implies
acceptance.  An accepted state contains every hypothesis of a fresh
invocation of the path lemmas: in particular the potential, not merely the
$\ell_\infty$ tracking bound, is reset.  Past sketches occur in no future
hypothesis, and the phase and $t$ encode all deterministic progress.  The
two costs share one residual game only if the switch occurs inside the same
attempt; every accepted checkpoint starts a new game.

Let $T$ be the number of accepted checkpoints strictly preceding the final
segment (which may itself contain $B_{\rm ep}$ instructions and complete
without a checkpoint).  Then $T\leq J_{\rm ep}-1$.  The conditional acceptance probabilities
stochastically dominate Bernoulli variables of mean $4/5$.  A Chernoff bound
shows that the first $3J_{\rm ep}$ attempts contain $T$ accepted full epochs
except with probability $e^{-\Omega(J_{\rm ep})}$; for bounded
$J_{\rm ep}$ the same claim follows directly after decreasing the absolute
success constant.  This leaves at least one of the total $4J_{\rm ep}$
attempts for the final segment.  That segment has at most
$B_{\rm ep}$ instructions and, conditionally on the state at which it
starts, its local good event has probability at least $4/5$ and implies a
correct output.  If its good event fails, an uncertified output is counted as
part of the theorem's constant failure probability.  The final segment is
not followed by a mandatory exact checkpoint, which preserves the cheaper
initialization branch.  Together with the one-time initialization event,
this gives constant overall success.

Every attempt announces local deterministic caps before work.  Let
$A_{\rm att}\leq4J_{\rm ep}$ be the number of attempts, and write $b_a$ for
the announced instruction cap of attempt $a$.  Because
$J_{\rm ep}=\lceil N_{\rm cap}/B_{\rm ep}\rceil+1$ and
$b_a\leq\min\{B_{\rm ep},N_{\rm cap}\}$,
\[
 \sum_{a=1}^{A_{\rm att}} b_a=O(N_{\rm cap}),\qquad
 A_{\rm att}=O(N_{\rm cap}/B_{\rm ep}+1).
\]
Summing Eq.~\eqref{eq:attempt_resource_caps} therefore gives
\[
 O(N_{\rm cap}(E^{-1}\ell+\ell+q)
 +(N_{\rm cap}/B_{\rm ep}+1)E^{-2}\ell)
\]
systems before normalization.  The choice of $A_{\rm ep}$ ensures
\[
 B_{\rm ep}\geq E^{-2}\ell,
 \qquad
 (N_{\rm cap}/B_{\rm ep})P_{\rm norm}=O(N_{\rm cap}).
\]
When
$N_{\rm cap}<B_{\rm ep}$ no normalization is invoked.  The same summation
applies to rounds and scalar work.  Every confidence logarithm is
$O(\ell)$ independently of $\kappa$.
The final partial segment pays no extra exact checkpoint.  System depth and
scalar work obey the same constant-factor accounting.
\end{proof}

\begin{algorithm}[!ht]
\caption{Two-phase linear-programming solver}
\label{alg:lp_solve_fresh}
\begin{algorithmic}[1]
\Procedure{LPSolve}{$x_0,U,\epsilon$}
\Comment{Theorem~\ref{thm:log_four_formal}}
\State $\ell\gets\log(4m)$, $q\gets\log(4m/n)$, and
$\kappa\gets A_\kappa(1+\log(emU/\epsilon))$.
\State Set $p,v,c_k,C,K$ as in Lemma~\ref{lem:sharper_parameters},
$L_*\gets\ell$,
$L_{\rm ch}\gets A_LL_*$, $E\gets K/(A_EL_{\rm ch})$, and
$R\gets c_RCE$.
\State $(N_{\rm cap},M_{\rm calls},L_{\rm conf},
B_{\rm sys},B_{\rm rnd},B_{\rm ar},H_{\rm cap})\gets$ the explicit quantities in
Lemma~\ref{lem:four_log_ledger}.
\State Initialize the five shared counters---score queries, permitted
systems, sequential system rounds, exact scalar operations, and scalar
parallel depth---at zero.  Before each operation, announce its deterministic
charges and compare them with $M_{\rm calls}$, $B_{\rm sys}$,
$B_{\rm rnd}$, $B_{\rm ar}$, and $H_{\rm cap}$, respectively; return
\textnormal{\textsc{Fail}} before work or randomness if any cap would be exceeded.
\State $\zeta\gets1/100$ and
$Y\gets\Call{FixedPoint}{x_0,\mathbf1,q,E/4,\zeta/3,L_{\rm conf},
\varnothing,\mathsf{auto}}$.
\Comment{Algorithm~\ref{alg:fresh_fixed_point}}
\If{$Y=\textnormal{\textsc{Fail}}$}
  \State \Return \textnormal{\textsc{Fail}}.
\EndIf
\State $(\widehat g,\mathcal S)\gets Y$.
\State $w\gets\widehat g$, $d\gets-w\odot\phi'(x_0)$, and
$(e,\theta)\gets(E/4,\bot)$.
\State $(t_{\rm sw},t_{\rm out},\tau_{\rm sw},\tau_{\rm out})\gets$ the
parameters of Lemma~\ref{lem:outer_reduction}.
\State Form the resumable stream $\mathfrak P$ from the two path-phase
descriptors $(d,1,t_{\rm sw},\tau_{\rm sw})$ and
$(c,t_{\rm sw},t_{\rm out},\tau_{\rm out})$.
Use the amplified Algorithm~\ref{alg:gram_center}, the path step and test
in Eq.~\eqref{eq:four_log_path_step} and Eq.~\eqref{eq:four_log_path_test},
and the terminal convention of Lemma~\ref{lem:four_log_terminal}, so the
second terminal loop is omitted from the stream.
Algorithm~\ref{alg:checkpointed_path}, with the local caps and strengthened
normalization in Lemma~\ref{lem:four_log_ledger}, supplies
$\zeta_{\rm step}$, $L_{\rm conf}$, and the current score state to each
atomic instruction.  Save and restore $(e,\theta)$ with the mathematical
state, and reset it to $(E/8,\bot)$ after normalization.
The cost switch shares the residual game only when it occurs inside one
attempt; every accepted checkpoint starts a fresh game.
\State $Y\gets\Call{CheckpointedPath}{\mathfrak P,N_{\rm cap}}$.
\Comment{Algorithm~\ref{alg:checkpointed_path}}
\If{$Y=\textnormal{\textsc{Fail}}$}
  \State \Return \textnormal{\textsc{Fail}}.
\EndIf
\State $(x,w,\ldots)\gets Y$.
\State \Return $\Call{FixedWeightTail}{x,w,t_{\rm out},c,\tau_{\rm out}}$.
\EndProcedure
\end{algorithmic}
\end{algorithm}

We prove the hybrid score-maintenance bound by transporting relative
leverage-score errors through diagonal rescalings.  The provider transports
the own-row response exactly and sketches only the nonlocal endpoint
discrepancy; no path-geometry or residual-game lemma is altered.
For a positive diagonal matrix $D$, write
\[
 P_D:=P(D),\qquad \sigma_i(D):=(P_D)_{ii}=\|P_De_i\|_2^2.
\]
All sigma-fields below are taken immediately before the fresh Gaussian
sketch of the transition under discussion.

\subsection{Exact diagonal self-transport}
\label{subsec:exact_diagonal_self_transport}

For $\gamma>0$ and $s\in[0,1]$, define
\[
 \mathcal T_\gamma(s):=
 \frac{\gamma^2s}{1+(\gamma^2-1)s}.
\]

Diagonal self-transport starts from the exact effect of rescaling a single
row, including its action on an imperfect incoming score.  The next lemma
records that identity and bounds the transported relative error.
\begin{lemma}[One-row transport and approximate anchor]
\label{lem:candidate_one_row_transport}
Let $D$ be positive diagonal, define $P=P_D$, $p_i=Pe_i$, and
$s_i=P_{ii}$.  Let $D^{(i,\gamma)}$ be obtained from $D$ by multiplying its
$i$-th diagonal entry by $\gamma$.  Then
\[
 P_{D^{(i,\gamma)}}e_i=
 \frac{\gamma}{1+(\gamma^2-1)s_i}
 (p_i+(\gamma-1)s_ie_i)
\]
and
\[
 \sigma_i(D^{(i,\gamma)})=\mathcal T_\gamma(s_i).
\]
For $t\in[0,1]$, define
\[
 q_i(t):=
 \frac{\gamma}{1+(\gamma^2-1)t}
 (p_i+(\gamma-1)te_i).
\]
Then, for every $s,t\in[0,1]$ with $s=s_i$,
\begin{equation}
\label{eq:candidate_anchor_identity}
 \mathcal T_\gamma(t)-\|q_i(t)\|_2^2
 =A_\gamma(t)(t-s),
 \qquad
 A_\gamma(t):=
 \frac{\gamma^2(1+2(\gamma-1)t)}
 {(1+(\gamma^2-1)t)^2}.
\end{equation}
If $s>0$, $|\log\gamma|\leq1/8$, $|t-s|\leq\epsilon s$, and
$0\leq\epsilon\leq1/4$, then $A_\gamma(t)\geq0$ and
\begin{equation}
\label{eq:candidate_anchor_multiplier}
 \frac{A_\gamma(t)s}{\mathcal T_\gamma(s)}
 \leq\exp(12\epsilon |\log\gamma|s).
\end{equation}
\end{lemma}
\begin{proof}
Write $B=DA$, $M=B^\top B$, and let $b_i^\top$ be the $i$-th row of
$B$.  Thus $p_i=BM^{-1}b_i$ and $s_i=b_i^\top M^{-1}b_i$.  The new Gram
matrix is $M+(\gamma^2-1)b_ib_i^\top$.  Sherman--Morrison gives
\[
 (M+(\gamma^2-1)b_ib_i^\top)^{-1}b_i
 =\frac{M^{-1}b_i}{1+(\gamma^2-1)s_i}.
\]
Multiplying the $i$-th row of $B$ by $\gamma$ now gives the first display.
The norm identity follows because $P_{D^{(i,\gamma)}}$ is an orthogonal
projection.

For Eq.~\eqref{eq:candidate_anchor_identity}, define $b=\gamma-1$ and
$d(t)=1+(\gamma^2-1)t$.  Since
$\|p_i\|_2^2=e_i^\top p_i=s$, direct expansion gives
\[
 \|q_i(t)\|_2^2
 =\frac{\gamma^2}{d(t)^2}
 (s+2bts+b^2t^2).
\]
Subtracting this quantity from $\gamma^2t/d(t)$ yields
$\gamma^2(t-s)(1+2bt)/d(t)^2$.

When $|\log\gamma|\leq1/8$, one has $\gamma>1/2$, so
$1+2(\gamma-1)t\geq2\gamma-1>0$.  Define
\[
 R_\gamma(t):=
 \frac{A_\gamma(t)s}{\mathcal T_\gamma(s)}.
\]
At $t=s$,
\[
 R_\gamma(s)=
 \frac{1+2(\gamma-1)s}{1+(\gamma^2-1)s}
 =1-\frac{(\gamma-1)^2s}{1+(\gamma^2-1)s}
 \leq1.
\]
Moreover,
\[
 \frac{\d}{\d t}\log R_\gamma(t)
 =\frac{2(\gamma-1)}{1+2(\gamma-1)t}
 -\frac{2(\gamma^2-1)}{1+(\gamma^2-1)t}.
\]
Define $a:=|\log\gamma|\leq1/8$.  Then
$7/8\leq\gamma\leq8/7$, both denominators are at least $3/4$, and
$|\gamma-1|\leq8a/7$ and $|\gamma^2-1|\leq120a/49$.  Consequently,
\[
 |\frac{\d}{\d t}\log R_\gamma(t)|
 \leq(\frac{64}{21}+\frac{320}{49})a
 <10a<12a.
\]
Integrating from $s$ to $t$ and using
$|t-s|\leq\epsilon s$ proves
Eq.~\eqref{eq:candidate_anchor_multiplier}.
\end{proof}

The one-row identity must be combined with the simultaneous motion of all
other rows between two score queries.  The following endpoint estimate shows
that the nonlocal remainder is governed by the companion-weighted drift.
\begin{lemma}[Finite nonlocal endpoint bound]
\label{lem:candidate_nonlocal_endpoint}
Let $D_-,D_+$ and a positive vector $\widetilde w$ be measurable with
respect to a sigma-field $\mathcal F$.  Define
\[
 d:=\log\operatorname{diag}(D_+D_-^{-1}),\qquad
 \lambda_w:=\frac{\sum_{j=1}^m\widetilde w_jd_j}{\sum_{j=1}^m\widetilde w_j},
\]
\[
 \delta:=d-\lambda_w\mathbf1,\qquad
 \rho_w:=(\sum_{j=1}^m\widetilde w_j\delta_j^2)^{1/2}.
\]
Suppose
\begin{equation}
\label{eq:candidate_endpoint_hypotheses}
 \|\delta\|_\infty\leq1/8,\qquad
 \sigma_j(D_-)\leq Q\widetilde w_j\quad(j\in[m]),
\end{equation}
where $Q$ is an absolute constant.  For a fixed coordinate $i$, define
$s_i=\sigma_i(D_-)$, $\gamma_i=e^{\delta_i}$, and let $q_i(t)$ be the
vector in Lemma~\ref{lem:candidate_one_row_transport}, formed from
$P_{D_-}e_i$.  If $|t_i-s_i|\leq\epsilon s_i$ with
$0\leq\epsilon\leq1/4$, then
\begin{equation}
\label{eq:candidate_finite_product}
 \|P_{D_+}e_i-q_i(t_i)\|_2
 \|P_{D_+}e_i+q_i(t_i)\|_2
 \leq K_Q\rho_w\sigma_i(D_+)
\end{equation}
for an absolute constant $K_Q$ depending only on $Q$, provided
$\rho_w\leq\rho_Q$ for a sufficiently small absolute $\rho_Q>0$.
Furthermore,
\begin{equation}
\label{eq:candidate_score_comparison}
 e^{-K_Q\rho_w}\mathcal T_{\gamma_i}(s_i)
 \leq\sigma_i(D_+)
 \leq e^{K_Q\rho_w}\mathcal T_{\gamma_i}(s_i).
\end{equation}
\end{lemma}
\begin{proof}
If $s_i=0$, then row $i$ of $A$ is zero because $D_-$ is positive.
Positive diagonal rescaling preserves this property, so both projection
columns in the statement are zero and all conclusions are immediate.
Assume below that $s_i>0$.
Multiplying a diagonal query by a scalar does not change its projection, so
replace $D_+$ by $e^{-\lambda_w}D_+$.  First multiply only row $i$ of
$D_-$ by $\gamma_i$.  By
Lemma~\ref{lem:candidate_one_row_transport}, the resulting $i$-th projection
column is $q_i(s_i)$ and its squared norm is
$\mathcal T_{\gamma_i}(s_i)$.

Connect this intermediate query to $D_+$ by multiplying every row
$j\neq i$ by $e^{t\delta_j}$, $t\in[0,1]$.  Along this path write $P(t)$
for the projection, $\sigma_j(t)=P(t)_{jj}$, and
$\Delta=\operatorname{Diag}(\delta)$ with $\delta_i=0$.  Differentiation
of the projection gives
\[
 \dot P=(I-P)\Delta P+P\Delta(I-P).
\]
Since $\Delta e_i=0$,
\[
 \dot Pe_i=(I-2P)\Delta Pe_i.
\]
The matrix $I-2P$ is an orthogonal reflection.  Consequently
\begin{equation}
\label{eq:candidate_column_energy}
 \|\dot Pe_i\|_2^2
 =\sum_{j=1}^m\delta_j^2P_{ji}^2
 \leq\sigma_i(t)\sum_{j=1}^m\sigma_j(t)\delta_j^2,
\end{equation}
where the last step uses $P_{ji}^2\leq\sigma_j(t)\sigma_i(t)$.
Similarly,
\[
 \dot\sigma_i(t)=-2\sum_{j=1}^m\delta_jP_{ij}(t)^2.
\]
Weighted Cauchy--Schwarz and
$\sum_{j=1}^mP_{ij}^4/\sigma_j(t)\leq
\sigma_i(t)\sum_{j=1}^mP_{ij}^2=\sigma_i(t)^2$ show that
\begin{equation}
\label{eq:candidate_log_score_drift}
 \lvert\frac{\d}{\d t}\log\sigma_i(t)\rvert
 \leq2(\sum_{j=1}^m\sigma_j(t)\delta_j^2)^{1/2}.
\end{equation}

Every intermediate query is within logarithmic $\ell_\infty$ distance
$1/8$ of $D_-$.  Lemma~\ref{lem:relative_drift} and
Eq.~\eqref{eq:candidate_endpoint_hypotheses} therefore give
\[
 \sum_{j=1}^m\sigma_j(t)\delta_j^2
 \leq e^{1/2}Q\rho_w^2.
\]
Integrating Eq.~\eqref{eq:candidate_log_score_drift} proves
Eq.~\eqref{eq:candidate_score_comparison}, after enlarging $K_Q$.
Integrating Eq.~\eqref{eq:candidate_column_energy} and using the same score
comparison gives
\begin{equation}
\label{eq:candidate_exact_anchor_column}
 \|P_{D_+}e_i-q_i(s_i)\|_2
 \leq K_Q\rho_w\sqrt{\sigma_i(D_+)}.
\end{equation}

It remains to replace $s_i$ by $t_i$.  Direct differentiation gives
\[
 q_i'(t)=
 \frac{\gamma_i(\gamma_i-1)}
 {(1+(\gamma_i^2-1)t)^2}
 (e_i-(\gamma_i+1)P_{D_-}e_i),
\]
and
\[
 \|e_i-(\gamma_i+1)P_{D_-}e_i\|_2^2
 =1+(\gamma_i^2-1)s_i.
\]
Thus
\[
 \|q_i(t_i)-q_i(s_i)\|_2
 \leq3\epsilon|\delta_i|s_i
 \leq3\epsilon\sqrt Q\rho_w\sqrt{s_i}.
\]
Since $|\log\gamma_i|\leq1/8$, one has
$\mathcal T_{\gamma_i}(s_i)=e^{\pm1/4}s_i$.
Eq.~\eqref{eq:candidate_score_comparison} therefore converts the last
factor to $O(\sqrt{\sigma_i(D_+)})$.  Combining this estimate with
Eq.~\eqref{eq:candidate_exact_anchor_column} proves the first-factor bound
$O(\rho_w\sqrt{\sigma_i(D_+)})$.  Both columns have norm
$O(\sqrt{\sigma_i(D_+)})$ when $\rho_Q$ is sufficiently small.  Their
product proves Eq.~\eqref{eq:candidate_finite_product}.
\end{proof}

The endpoint estimate now lets us transport an already approximate score and
use fresh randomness only for the residual discrepancy.  The next lemma gives
the conditional recurrence that drives the maintained raw-relative error.
\begin{lemma}[Conditional self-transport transition]
\label{lem:candidate_self_transport_transition}
Assume the hypotheses of
Lemma~\ref{lem:candidate_nonlocal_endpoint}.  Suppose an
$\mathcal F$-measurable incoming state $h\in\mathbb R^m$ satisfies
\[
 |h_i-\sigma_i(D_-)|\leq\epsilon\sigma_i(D_-),
 \qquad 0\leq\epsilon\leq1/4.
\]
Define $t_i=\min\{\max\{h_i,0\},1\}$ and
\[
 a_i:=\frac{\gamma_i}{1+(\gamma_i^2-1)t_i},
 \qquad c_i:=a_i(\gamma_i-1)t_i,
\]
\[
 Q_t:=P_{D_-}\operatorname{Diag}(a)+\operatorname{Diag}(c).
\]
After all these quantities and both endpoints are fixed, draw a fresh
$S\in\R^{k\times m}$ with independent $N(0,1/k)$ entries and set
\[
 h_i^+:=\mathcal T_{\gamma_i}(t_i)
 +\|SP_{D_+}e_i\|_2^2-\|SQ_te_i\|_2^2.
\]
Then
\begin{equation}
\label{eq:candidate_relative_recurrence}
 \frac{h_i^+-\sigma_i(D_+)}{\sigma_i(D_+)}
 =b_i\frac{t_i-\sigma_i(D_-)}{\sigma_i(D_-)}+\xi_i,
 \qquad 0\leq b_i\leq e^{K_Q\rho_w},
\end{equation}
and, conditionally on $\mathcal F$,
\begin{equation}
\label{eq:candidate_relative_mgf}
 \E[e^{\lambda\xi_i}\mid\mathcal F]
 \leq\exp(K_Q^2\lambda^2\rho_w^2/k),
 \qquad |\lambda|\leq k/(4K_Q\rho_w).
\end{equation}
The update for all coordinates uses one application of each endpoint
projection per sketch row because
\[
 SQ_t=(SP_{D_-})\operatorname{Diag}(a)+S\operatorname{Diag}(c).
\]
If $\rho_w=0$, the two endpoints differ only by a scalar and the algorithm
returns $t$ without drawing a sketch.
If $\sigma_i(D_-)=0$, use the exact convention
$h_i^+=0$, $b_i=0$, and $\xi_i=0$ in
Eq.~\eqref{eq:candidate_relative_recurrence}.
\end{lemma}
\begin{proof}
Clipping to $[0,1]$ cannot increase distance from the true score, so
$|t_i-\sigma_i(D_-)|\leq\epsilon\sigma_i(D_-)$.  The $i$-th column of
$Q_t$ is exactly $q_i(t_i)$.  Conditional on $\mathcal F$, Gaussian
isotropy and Eq.~\eqref{eq:candidate_anchor_identity} give
\[
 \E[h_i^+-\sigma_i(D_+)\mid\mathcal F]
 =A_{\gamma_i}(t_i)(t_i-\sigma_i(D_-)).
\]
Thus Eq.~\eqref{eq:candidate_relative_recurrence} holds with
\[
 b_i:=\frac{A_{\gamma_i}(t_i)\sigma_i(D_-)}{\sigma_i(D_+)}.
\]
It is nonnegative by Lemma~\ref{lem:candidate_one_row_transport}.
Eq.~\eqref{eq:candidate_anchor_multiplier} and
Eq.~\eqref{eq:candidate_score_comparison}, together with
\[
 |\delta_i|\sqrt{\sigma_i(D_-)}
 \leq\sqrt Q\rho_w,
\]
give $b_i\leq e^{K_Q\rho_w}$ after enlarging $K_Q$.

For the centered innovation, use the identity
\[
 \|Su\|_2^2-\|Sv\|_2^2
 =\langle S(u-v),S(u+v)\rangle
\]
with $u=P_{D_+}e_i$ and $v=q_i(t_i)$.  The centered product of two jointly
Gaussian variables is sub-exponential.  Equation
Eq.~\eqref{eq:candidate_finite_product} therefore gives
\[
 \E[e^{\lambda X_i}\mid\mathcal F]
 \leq\exp(K_Q^2\lambda^2\rho_w^2\sigma_i(D_+)^2/k)
\]
for $|\lambda|\leq k/(4K_Q\rho_w\sigma_i(D_+))$, where $X_i$ is the
centered unnormalized sketch error.  Substitute
$\xi_i=X_i/\sigma_i(D_+)$ to obtain
Eq.~\eqref{eq:candidate_relative_mgf}.

If $\rho_w=0$, positivity of $\widetilde w$ gives $\delta=0$.  Hence the
endpoints differ only by a scalar, $P_{D_+}=P_{D_-}$, and the exact
transported value is $t$.
\end{proof}

\subsection{Hybrid maintenance and stopping}
\label{subsec:hybrid_maintenance_stopping}

For a transition define the ordinary centered drift
\[
 \lambda_\infty:=\frac{\max_{i\in[m]} d_i+\min_{i\in[m]} d_i}{2},
 \qquad
 \rho_\infty:=\|d-\lambda_\infty\mathbf1\|_\infty.
\]
Not every score transition satisfies the weighted eligibility condition for
self-transport.  The following ordinary transition provides a uniform fallback
whose cost is charged to centered $\ell_\infty$ drift.
\begin{lemma}[Ordinary relative transition]
\label{lem:candidate_ordinary_relative_transition}
Let $D_-,D_+$ be fixed by a sigma-field $\mathcal F$, define
$d=\log\operatorname{diag}(D_+D_-^{-1})$, and suppose
$\rho_\infty\leq1/8$.  If an incoming raw state $h$ obeys
$|h_i-\sigma_i(D_-)|\leq\epsilon\sigma_i(D_-)$ with
$0\leq\epsilon\leq1/4$, define $t_i=\min\{\max\{h_i,0\},1\}$.
After the endpoints and $t$ are fixed, draw a fresh
$S\in\R^{k\times m}$ with independent $N(0,1/k)$ entries and set
\[
 h_i^+:=t_i+\|SP_{D_+}e_i\|_2^2-\|SP_{D_-}e_i\|_2^2.
\]
For every coordinate with nonzero score,
\[
 \frac{h_i^+-\sigma_i(D_+)}{\sigma_i(D_+)}
 =b_i\frac{t_i-\sigma_i(D_-)}{\sigma_i(D_-)}+\xi_i,
 \qquad b_i:=\frac{\sigma_i(D_-)}{\sigma_i(D_+)},
\]
where $0\leq b_i\leq e^{4\rho_\infty}$ and, for absolute constants
$c,C>0$,
\[
 \E[e^{\lambda\xi_i}\mid\mathcal F]
 \leq\exp(C\lambda^2\rho_\infty^2/k),
 \qquad |\lambda|\leq ck/\rho_\infty.
\]
If the score is zero, then the corresponding row of every projection query
is zero and the transition returns zero exactly.  If $\rho_\infty=0$, the
two projections coincide and the procedure returns $t$ without a sketch.
\end{lemma}
\begin{proof}
Projection matrices are invariant under a common scalar rescaling of $D$,
so remove $\lambda_\infty\mathbf1$ from $d$.  The relative-drift lemma gives
$\sigma_i(D_+)=e^{\pm4\rho_\infty}\sigma_i(D_-)$.  Conditional Gaussian
isotropy shows that the transition increment is unbiased.  Apply
Lemma~\ref{lem:adaptive_increment} with
$a_i=\max\{\sigma_i(D_-),\sigma_i(D_+)\}$, which is
$\mathcal F$-measurable.  After division by $\sigma_i(D_+)$, score
comparison turns its conditional moment-generating-function bound into the
displayed relative bound.  The formula for $b_i$ is the remaining
deterministic part.  Finally, clipping to $[0,1]$, an interval containing the
true score, multiplies the incoming relative error by a number in $[0,1]$.
\end{proof}

For the same transition, and for the positive companion
$\widetilde w^-$ stored with the old query, define
\[
 \lambda_w:=\frac{\sum_{i=1}^m\widetilde w_i^-d_i}
 {\sum_{i=1}^m\widetilde w_i^-},
 \qquad
 \rho_w:=(\sum_{i=1}^m\widetilde w_i^-(d_i-\lambda_w)^2)^{1/2}.
\]
Fix absolute constants $C_\infty,C_w$ large enough to dominate all constants
in the ordinary and weighted transition bounds, and then choose
\[
 h_0\leq\min\{\log(4/3),C_\infty/8,C_w\rho_Q\}.
\]
The ordinary branch has
effective drift $C_\infty\rho_\infty$ and is eligible when this quantity is
at most $h_0$.  The weighted branch is eligible only if
\[
 \|d-\lambda_w\mathbf1\|_\infty\leq1/8
 \quad\hbox{and}\quad C_w\rho_w\leq h_0;
\]
when eligible, its effective drift is $C_w\rho_w$.  Define $\eta$ as the
smaller eligible effective drift.  If no branch is eligible, charge the
transition the finite amount $h_0$ and force a fresh base at the new
endpoint; no transition recurrence is then used.  On every nonrestart
transition, both branches have a common recurrence
\begin{equation}
\label{eq:candidate_unified_recurrence}
 z_{t+1,i}=b_{t,i}r_{t,i}z_{t,i}+\xi_{t,i},
 \qquad 0\leq r_{t,i}\leq1,\qquad
 0\leq b_{t,i}\leq e^{\eta_t},
\end{equation}
and
\begin{equation}
\label{eq:candidate_unified_mgf}
 \E[e^{\lambda\xi_{t,i}}\mid\mathcal F_t]
 \leq e^{\lambda^2\eta_t^2/k_t},
 \qquad |\lambda|\leq k_t/(4\eta_t).
\end{equation}
where $z_{t,i}:=(h_{t,i}-\sigma_i(D_t))/\sigma_i(D_t)$ for a nonzero
score.  The zero-score coordinates are returned exactly and omitted from
this normalization.
Here the factor $r_{t,i}$ is the contraction caused by clipping the previous
raw state to $[0,1]$.  The ordinary branch follows from
Lemma~\ref{lem:candidate_ordinary_relative_transition}; the weighted branch
follows from Lemma~\ref{lem:candidate_self_transport_transition}.  The
constants $C_\infty,C_w$ are fixed
once so that Eq.~\eqref{eq:candidate_unified_recurrence} and
Eq.~\eqref{eq:candidate_unified_mgf} hold literally in both branches.

\begin{algorithm}[!ht]
\caption{Hybrid relative score maintenance}
\label{alg:candidate_hybrid_score}
\begin{algorithmic}[1]
\Procedure{HybridScore}{$D,\widetilde w,L_{\rm conf},\mathcal S$}
\Comment{Lemma~\ref{lem:candidate_hybrid_maintenance}}
\If{$\mathcal S=\varnothing$}
  \State Go to \textsc{FreshBase}.
\Else
  \State $(D_-,\widetilde w^-,h,H)\gets\mathcal S$ and
  $d\gets\log\operatorname{diag}(DD_-^{-1})$.
  \State Compute $\rho_\infty$ and $\rho_w$ from the two displays above.
  \State $\eta_\infty\gets C_\infty\rho_\infty$ if
  $C_\infty\rho_\infty\leq h_0$, and
  $\eta_\infty\gets+\infty$ otherwise.
  \State $\eta_w\gets C_w\rho_w$ if
  $\|d-\lambda_w\mathbf1\|_\infty\leq1/8$ and $C_w\rho_w\leq h_0$,
  and $\eta_w\gets+\infty$ otherwise.
  \State $\eta\gets\min\{\eta_\infty,\eta_w\}$, breaking ties
  deterministically.
  \If{$\eta=+\infty$}
    \State Charge effective drift $h_0$ and go to \textsc{FreshBase}.
  \ElsIf{$H+\eta>h_0$}
    \State Charge effective drift $\eta$ and go to \textsc{FreshBase}.
  \EndIf
  \State $t_i\gets\min\{\max\{h_i,0\},1\}$ for every $i\in[m]$.
  \If{$\eta=0$}
    \State $h\gets t$.
  \Else
    \State $k\gets\max\{1,\lceil C_1\eta E^{-2}L_{\rm conf}\rceil\}$
    and draw a fresh $S\in\R^{k\times m}$ with independent $N(0,1/k)$
    entries after all preceding quantities are fixed.
    \If{$\eta=\eta_\infty$}
      \State Update $h$ by
      Lemma~\ref{lem:candidate_ordinary_relative_transition}, using the
      already drawn matrix $S$.
    \Else
      \State Update $h$ by
      Lemma~\ref{lem:candidate_self_transport_transition}, using
      the already drawn matrix $S$ and $\widetilde w^-$ to define
      $\lambda_w$ and $\rho_w$.
    \EndIf
  \EndIf
  \State $H\gets H+\eta$ and go to \textsc{ReturnState}.
\EndIf
\Statex \textsc{FreshBase}:
\State $k_0\gets\lceil C_0E^{-2}L_{\rm conf}\rceil$, draw a fresh
$S\in\R^{k_0\times m}$ with independent $N(0,1/k_0)$ entries after
$D,\widetilde w$ are fixed, set
$h_i\gets\|SP_De_i\|_2^2$ for every $i$, and set $H\gets0$.
\State Go to \textsc{ReturnState}.
\Statex \textsc{ReturnState}:
\State $h_i^{\rm out}\gets\min\{\max\{h_i,0\},1\}$ for every $i$.
\State $\mathcal S\gets(D,\widetilde w,h,H)$.
\State \Return $(h^{\rm out},\mathcal S)$.
\EndProcedure
\end{algorithmic}
\end{algorithm}

The block analysis needs a maximal inequality that allows inherited error to
be predictably contracted before each fresh innovation is added.  The next
lemma supplies exactly this stopped, time-uniform concentration bound.
\begin{lemma}[Contractive maximal Bernstein]
\label{lem:candidate_contractive_bernstein}
Let $T\geq1$ and
\[
 Y_t=a_tY_{t-1}+X_t,\qquad0\leq a_t\leq1,
 \qquad t=1,\ldots,T,
\]
where $a_t$ is predictable and
\[
 \E[e^{\lambda X_t}\mid\mathcal F_{t-1}]
 \leq e^{\lambda^2v_t/2}
 \quad\hbox{for }|\lambda|\leq b_t^{-1}.
\]
If $\sum_{t=1}^Tv_t\leq V_0$, $\max_{t\in[T]}b_t\leq b_0$, and $|Y_0|\leq u_0$, then,
for $u>u_0$,
\[
 \Pr[\max_{t\in[T]}|Y_t|\geq u]
 \leq2\exp(-c\min\{
 \frac{(u-u_0)^2}{V_0},\frac{u-u_0}{b_0}
 \}).
\]
The conclusion remains valid at a first-exit stopping time, with the crossing
update included and the process frozen thereafter; equivalently, its
pre-exit activity indicator is predictable.
\end{lemma}
\begin{proof}
For every fixed admissible $\lambda>0$,
\[
 \E[\cosh(\lambda Y_t)\mid\mathcal F_{t-1}]
 \leq e^{\lambda^2v_t/2}\cosh(\lambda a_tY_{t-1})
 \leq e^{\lambda^2v_t/2}\cosh(\lambda Y_{t-1}).
\]
Thus
\[
 \cosh(\lambda Y_t)
 \exp(-\frac{\lambda^2}{2}\sum_{j=1}^tv_j)
\]
is a nonnegative supermartingale.  Stop at the first crossing of $u$ and
apply Ville's inequality.  Since
$\cosh(\lambda u)\geq e^{\lambda u}/2$ and
$\cosh(\lambda Y_0)\leq e^{\lambda u_0}$, the crossing probability is at
most
\[
 2\exp(-\lambda(u-u_0)+\lambda^2V_0/2).
\]
Take $\lambda=\min\{(u-u_0)/V_0,(2b_0)^{-1}\}$.  For a stopped process,
multiply every innovation by the predictable indicator of not having
stopped, and after stopping take coefficient one and innovation zero.
\end{proof}

The two transition recurrences and the maximal inequality can now be combined
into one adaptive score provider.  The following lemma states its uniform
accuracy event and prices bases, transitions, and restarts in a single ledger.
\begin{lemma}[Hybrid self-transport score maintenance]
\label{lem:candidate_hybrid_maintenance}
Consider an adaptive sequence of at most $M_{\rm ep}$ score queries.  Each
query $D_t$ has a positive companion $\widetilde w^{(t)}$ fixed before its
score sketch and satisfying, before the companion stopping time,
\begin{equation}
\label{eq:candidate_companion_band}
 Q^{-1}(\sigma_i(D_t)+v_i)
 \leq\widetilde w^{(t)}_i
 \leq Q(\sigma_i(D_t)+v_i)
 \qquad(i\in[m])
\end{equation}
for an absolute $Q$.  Assume the block input satisfies the raw-relative
condition
\[
 |h_{b,i}-\sigma_i(D_b)|
 \leq(c_{\rm anc}E/8)\sigma_i(D_b).
\]
Restart immediately before cumulative effective drift would exceed a fixed
small $h_0$, and on a nonrestart transition use
\[
 k_t:=\max\{1,\lceil C_1\eta_tE^{-2}L_{\rm conf}\rceil\},
 \qquad L_{\rm conf}:=\log(40mM_{\rm ep}).
\]
Every transition sketch is fresh and is drawn only after both endpoints,
the old companion, both normalizations, the branch, $\eta_t$, and $k_t$ are
fixed.  Run Algorithm~\ref{alg:candidate_hybrid_score}.  In the accounting
below, $\eta_t$ is the selected finite drift whenever a branch is eligible,
including when that transition starts a fresh block, and is $h_0$ only when
no branch is eligible.  Then, for sufficiently
large absolute base and transition constants, the raw states satisfy
\begin{equation}
\label{eq:candidate_raw_relative_accuracy}
 |h_{t,i}-\sigma_i(D_t)|
 \leq c_{\rm anc}E\sigma_i(D_t)
\end{equation}
simultaneously before the companion stopping time, except with probability
at most $1/40$.  The same inequality then holds for the clipped outputs.
If $M$ is the actual number of transitions and
$\Omega:=\sum_{t=1}^M\eta_t$, the total number of permitted systems is
\begin{equation}
\label{eq:candidate_hybrid_score_cost}
 O(M+E^{-2}L_{\rm conf}(1+\Omega)).
\end{equation}
\end{lemma}
\begin{proof}
At every fresh base, condition on the complete pre-base sigma-field and draw
\[
 k_0=\lceil C_0E^{-2}L_{\rm conf}\rceil
\]
Gaussian rows.  For $\sigma_i(D_b)>0$,
\[
 \frac{\|SP_{D_b}e_i\|_2^2}{\sigma_i(D_b)}
 \stackrel{d}{=}\frac{\chi_{k_0}^2}{k_0}.
\]
Under the standing no-zero-row assumption, every score is positive.  If zero rows
are allowed, a zero score produces the exact zero column and is returned
exactly.  The chi-square tail and a union bound give the displayed base
condition with conditional failure at most $e^{-3L_{\rm conf}}$.

Fix a block with transition indices $b,\ldots,e-1$ and a coordinate, and stop
immediately before the first incoming raw-relative violation or companion
violation.  Let
$H_t=\sum_{j=b}^{t-1}\eta_j$.  All $\eta_j,H_t,k_t$ are predictable.
By Eq.~\eqref{eq:candidate_unified_recurrence},
\[
 Y_t:=e^{-H_t}z_t
\]
satisfies
\[
 Y_{t+1}=a_tY_t+\widetilde\xi_t,\qquad0\leq a_t\leq1.
\]
Eq.~\eqref{eq:candidate_unified_mgf} remains valid after multiplication
of the innovation by $e^{-H_{t+1}}\leq1$.  The sample rule gives
\[
 \sum_{t=b}^{e-1}\frac{\eta_t^2}{k_t}
 \leq\frac{h_0E^2}{C_1L_{\rm conf}},
 \qquad
 \max_{b\leq t\leq e-1}\frac{\eta_t}{k_t}
 \leq\frac{E^2}{C_1L_{\rm conf}}.
\]
Choose $h_0$ so that $e^{-h_0}\geq3/4$.  A crossing
$|z_t|\geq c_{\rm anc}E$ implies
$|Y_t|\geq3c_{\rm anc}E/4$, whereas the base has magnitude at most
$c_{\rm anc}E/8$.  Lemma~\ref{lem:candidate_contractive_bernstein}, with
$C_1$ sufficiently large, makes the coordinate-block crossing probability
at most $2e^{-3L_{\rm conf}}$.  The first violating update is included by
freezing the stopped process after it.  Thus there is no conditioning on a
future success event.

The restart decision is predictable.  If adding a transition would exceed
$h_0$, then either the completed block has accumulated at least $h_0/2$ or
the triggering transition has effective drift at least $h_0/2$.  These
charges are disjoint.  A transition with no eligible branch contributes
$h_0$ by definition.  Thus the number of fresh bases is at most
$1+2\Omega/h_0$.  Including a possible externally supplied initial base,
there are at most $M_{\rm ep}+1$ blocks.  The base
failure is at most $e^{-3L_{\rm conf}}$ per coordinate and block, and the
crossing failure is at most $2e^{-3L_{\rm conf}}$.  Hence their union is at
most
\[
 3m(M_{\rm ep}+1)e^{-3L_{\rm conf}}<1/40.
\]

Finally,
\[
 \sum_{t=1}^Mk_t=O(M+E^{-2}L_{\rm conf}\Omega),
\]
and the fresh bases cost
$O(E^{-2}L_{\rm conf}(1+\Omega))$.  Each row uses two endpoint systems for
a transition and one system for a base; constant factors are absorbed in
Eq.~\eqref{eq:candidate_hybrid_score_cost}.
\end{proof}

\subsection{Compatibility with the path algorithm}
\label{subsec:hybrid_path_compatibility}

The maintenance theorem assumes a two-sided companion band at every query.
The next lemma checks that all fixed-point and surrogate queries generated by
the path algorithm satisfy this hypothesis.
\begin{lemma}[Uniform query companions]
\label{lem:candidate_query_companions}

Every maintained-mode fixed-point and inflated-surrogate score query made
after the raw-relative state handoff and before the existing auxiliary
stopping time has a companion satisfying
Eq.~\eqref{eq:candidate_companion_band} for an absolute $Q$.
\end{lemma}
\begin{proof}
For an exact fixed-point query $D_*(x)$,
$g_i(x)=\sigma_i(D_*(x))+v_i$.  Every fixed-point companion $z$ used by the
algorithm satisfies $z=e^{\pm A_{\rm q}E}g(x)$.  The query formed from $z$
is within logarithmic row drift $O(E)$ of $D_*(x)$, so
Lemma~\ref{lem:relative_drift} gives
$\sigma_i(D_z)=e^{\pm O(E)}\sigma_i(D_*(x))$.  Hence
$z_i=e^{\pm O(E)}(\sigma_i(D_z)+v_i)$.

For the surrogate query,
$W^+=e^{c_+E}\operatorname{Diag}(\widehat g)$ with
$\widehat g=e^{\pm E/2}g(x)$.  Its query is again within $O(E)$ row drift of
$D_*(x)$, and the same argument yields
$W^+_{ii}=e^{\pm O(E)}(\sigma_i(P^+)+v_i)$.  This gives both directions of
Eq.~\eqref{eq:candidate_companion_band}, not only the already used upper
bound $\sigma_i(P^+)\leq W^+_{ii}$.
\end{proof}

The score provider's thresholds must coexist with the earlier geometric and
path-following constants without a circular choice.  The next lemma gives an
explicit admissible ordering for all remaining constants.
\begin{lemma}[Compatible constants for the hybrid score provider]
\label{lem:candidate_hybrid_constant_hierarchy}
The constants in the hybrid provider can be fixed compatibly with
Lemma~\ref{lem:constant_hierarchy}.  In particular, they can be chosen in
the noncircular order
\[
 \begin{aligned}
 &Q,\quad (K_Q,\rho_Q),\quad (C_\infty,C_w),\quad h_0,
   \quad c_R,\quad A_E,\\
 &c_{\rm anc},\quad (C_0,C_1,C_{\rm fp}),\quad A_{\rm path},
   \quad c_{\rm norm},\\
 &(A_{\rm norm},A_{\rm ep},A_L,c_{\rm sw},A_{\rm sw},A_{\rm call}),
   \quad A_\kappa,\quad A_N,\quad A_{\rm loc},\quad A_B.
 \end{aligned}
\]
The parenthesized scale and budget constants are chosen jointly as in the
proof.  With these choices, before the auxiliary stopping time every query has the
companion band in Eq.~\eqref{eq:candidate_companion_band}, every successful
transition has an eligible branch, every transition lemma is invoked inside
its stated radius and incoming-error range, and the base and crossing
probabilities in Lemma~\ref{lem:candidate_hybrid_maintenance} have the
claimed bounds.
\end{lemma}
\begin{proof}
Retain first all structural and deterministic constants already fixed in
Lemma~\ref{lem:constant_hierarchy}.  The proof of
Lemma~\ref{lem:candidate_query_companions} gives an absolute $E_{\rm cmp}>0$
and an absolute $Q$ for every $E\leq E_{\rm cmp}$.  Fix this $Q$.
Lemma~\ref{lem:candidate_nonlocal_endpoint} then fixes absolute constants
$K_Q$ and $\rho_Q$.

Let $c_{\rm ord},C_{\rm ord}$ be the constants in the moment-generating
function bound of Lemma~\ref{lem:candidate_ordinary_relative_transition}.
Choose
\[
 C_\infty\geq
 \max\{4,\sqrt{C_{\rm ord}},1/(4c_{\rm ord})\},
 \qquad C_w\geq\max\{K_Q,1\}.
\]
These choices make Eq.~\eqref{eq:candidate_unified_recurrence} and
Eq.~\eqref{eq:candidate_unified_mgf} literal for both transition branches.
Next choose
\[
 h_0\leq
 \min\{\log(4/3),C_\infty/8,C_w\rho_Q,1/16\}.
\]
Thus eligibility implies $\rho_\infty\leq1/8$ in the ordinary branch and
$\rho_w\leq\rho_Q$ in the weighted branch, while
$e^{-h_0}\geq3/4$.

The already proved query and movement bounds give absolute constants
$A_{\rm ord},A_w,A_\infty$ such that all old-query-to-surrogate and
same-point refinement transitions satisfy
$\rho_\infty\leq A_{\rm ord}E$, and the principal transition satisfies
\[
 \rho_w\leq A_wE,
 \qquad \|d^\circ\|_\infty\leq A_\infty R.
\]
Decrease the absolute radius multiplier $c_R$ once, and then increase $A_E$
once, so
that
\[
 E\leq\min\{E_{\rm cmp},1,
 h_0/(C_\infty A_{\rm ord}),h_0/(C_wA_w)\},
 \qquad A_E\geq400,
 \qquad 2A_\infty R\leq1/8.
\]
This is compatible with all old upper-bound requirements because
$E=K/(A_EL_{\rm ch})$ and $R=c_RCE$.  The ordinary transitions are now
eligible.  For the principal transition, its weighted mean lies between the
smallest and largest coordinates, and hence
\[
 \|d^\circ-\lambda_w\mathbf1\|_\infty
 \leq2\|d^\circ\|_\infty\leq1/8,
 \qquad C_w\rho_w\leq h_0.
\]
Thus its weighted branch is eligible as well.

Choose $c_{\rm anc}$ so that
$c_{\rm anc}\leq\min\{1/8,c_0/2\}$ and all old consumer tolerances hold.
This only
strengthens the old score requirements.  Since $E\leq1$, the incoming
relative error is then below $1/4$.  Finally choose
$C_0=\Omega(c_{\rm anc}^{-2})$ for the raw chi-square base and
$C_1=\Omega(c_{\rm anc}^{-2})$ for the maximal Bernstein bound, and choose
$C_{\rm fp}$ large enough for the fixed-point chi-square bound.  Increasing
these sample multipliers changes no earlier radius or eligibility choice.

Next choose $A_{\rm path}$ to dominate both deterministic path loops and
$c_{\rm norm}$ small enough for Lemma~\ref{lem:epoch_normalization}.
We choose the remaining scales jointly, retaining the dependence on $A_L$.
Fix an absolute $b_{\rm sw}>0$ with
\[
 b_{\rm sw}\leq\min\{2^{-22},c_R(\log4)^2/(32A_E)\},
\]
and fix an absolute $D_{\rm sw}$ sufficiently large relative to
$A_0/b_{\rm sw}$.  For every $A_L\geq1$, the prescriptions
$c_{\rm sw}=b_{\rm sw}/A_L$ and $A_{\rm sw}=D_{\rm sw}A_L$
satisfy Eq.~\eqref{eq:switch_constant_choice} and the cost-switch estimate.
Any additional fixed-factor normalization or switch accuracy is imposed
by decreasing $c_{\rm norm},b_{\rm sw}$ and increasing $D_{\rm sw}$ now.
These constants will not change below.

There are fixed $D\geq1$ and an integer $d\geq1$, independent of $A_L$,
such that
\begin{equation}\label{eq:joint_epoch_polynomial_bound}
 P_{\rm chk},\quad E^{-2}\ell+4J_{\rm sw}+1
 \leq D(1+A_L)^d(4m)^d.
\end{equation}
Here $P_{\rm chk}$ bounds each resource count for the two cold checkpoint
computations and normalization together.  For its first bound use
Eq.~\eqref{eq:normalization_uniform_cost} and the exact fixed-point
schedule.  For the second bound, use
$E^{-2}\ell=(128A_E)^2A_L^2c_k\ell^3$ and
\[
 \frac{R}{\tau_{\rm sw}}
 =\frac{c_R\Lambda_{\rm out}\ell^2}{8A_Eb_{\rm sw}}.
\]
In particular $J_{\rm sw}=O(c_k(1+\log\ell))$ uniformly in $A_L$;
the same bound holds with $4R$ in its definition.  Increase $D,d$ once
so that $D(1+A_L)^d$ also dominates the fixed number of score queries
per atomic instruction and the coefficient needed for the cold
initializer's $O(\ell+\log(1+A_L))$ rounds.  Exact checkpoint queries
will be covered by $P_{\rm chk}\leq B_{\rm ep}$ below.

For an integer $H>d+11$, define
\begin{equation}\label{eq:joint_epoch_constants}
 \begin{aligned}
 A_{\rm ep}&:=H,& A_{\rm norm}&:=H-11,& A_L&:=H^2,\\
 c_{\rm sw}&:=b_{\rm sw}/H^2,&
 A_{\rm sw}&:=D_{\rm sw}H^2,&
 A_{\rm call}&:=D(1+H^2)^d.
 \end{aligned}
\end{equation}
Choose $H$ sufficiently large that
\begin{equation}\label{eq:joint_epoch_absorption}
 D(1+H^2)^d\leq4^{H-d-11},
 \qquad
 H^2\geq H+\frac{d\log(1+H^2)}{\log4}+D_{\rm hor},
\end{equation}
where $D_{\rm hor}$ is fixed independently of $H$ and dominates the
remaining tracking and confidence constants.  Such an integer exists because the
first right-hand side grows exponentially in $H$, whereas its left-hand
side grows polynomially, and the second left-hand side grows quadratically.
We can also satisfy every earlier absolute lower bound on $A_L$.
For every $m\geq1$,
\[
 D(1+H^2)^d(4m)^d\leq(4m)^{H-11}
 =P_{\rm norm}<B_{\rm ep}=(4m)^H,
\]
where the first step follows from Eq.~\eqref{eq:joint_epoch_absorption}
and $4m\geq4$, the second step follows from
Eq.~\eqref{eq:joint_epoch_constants}, and the third step follows from
$4m>1$.  Thus $B_{\rm ep}\geq E^{-2}\ell$,
$B_{\rm ep}>4J_{\rm sw}$, and all checkpoint work is absorbed.
Moreover $\beta\leq256m$ and $q\ell+1\leq2B_{\rm ep}$, so both
$\log(A_0B_{\rm ep}m\beta)$ and
$\log(40m\lceil A_{\rm call}(B_{\rm ep}+q\ell+1)\rceil)$ are at most
\[
 (H+\frac{d\log(1+H^2)}{\log4}+D_{\rm hor})\ell
 \leq A_L\ell,
\]
where the first step follows from substituting
Eq.~\eqref{eq:joint_epoch_constants} and using $\ell\geq\log4$, and
the second step follows from Eq.~\eqref{eq:joint_epoch_absorption}.
The initializer and per-instruction confidence logarithms obey the same
bound after the fixed choice of $D_{\rm hor}$.  This fixes every constant
in the parenthesized block without increasing $A_L$ afterward.

Choose $A_\kappa$ to dominate both phase log-ratios and the terminal
logarithms, including the now-fixed switch constants.  Finally choose
$A_N$ to dominate the frozen analytic atomic-path instruction count,
then $A_{\rm loc}$ to dominate the four good-attempt
charges in Eq.~\eqref{eq:attempt_resource_caps}, and then $A_B$ to dominate
the global system, round, and scalar work constants, including
$C_0,C_1,C_{\rm fp}$ and $1/h_0$.

After $A_B$ no absolute constant is changed.  In particular, no radius is
shrunk after $c_{\rm sw}$ is selected and no analytic work constant is
increased after the caps are selected.  This proves the claimed order and
fixes one algorithm rather than a separate choice for each call.

Under the standing no-zero-row assumption, every score is positive.  If
zero rows are retained instead, positive diagonal rescaling preserves them,
their projection columns and sketches are identically zero, and the exact
zero conventions in the three transition lemmas apply.
\end{proof}

The score complexity is determined by the total effective drift accumulated
across all queries in one centering call.  The following lemma charges every
such transition and proves the required $O(E)$ bound.
\begin{lemma}[Effective drift of one centering call]
\label{lem:candidate_step_effective_drift}
Before the minimum of the existing coarse, tracking, companion, and raw-score
stopping times, all score transitions in one call to
Algorithm~\ref{alg:fresh_predicted_center} have total hybrid effective drift
$O(E)$ and use $O(1)$ transitions.
\end{lemma}
\begin{proof}
There are exactly three types of transition.

First, the last old fixed-point query at $x$ moves to the inflated surrogate
query at the same $x$.  Its old companion and $W^+$ are both
$e^{\pm O(E)}g(x)$, so the ordinary centered drift is $O(E)$.  After the
absolute constants are fixed, the existing choice of sufficiently large
$A_E$ makes its effective drift at most $h_0$.

Second, the surrogate query at $x$ moves to the first predicted fixed-point
query at $x^+$.  Modulo the scalar $c_+E$, its row-log drift is
\[
 d^\circ=\psi+\frac{\vartheta}{2}\widehat u.
\]
The existing movement lemmas give
$N_w(\psi)=O(R)$ and $N_w(\widehat u)\leq r\leq\Delta=O(R)$.  Therefore
\[
 \|d^\circ\|_w=O(R/C)=O(E),
 \qquad \|d^\circ\|_\infty=O(R)\leq1/8.
\]
The path weight and $W^+$ are within an absolute coordinatewise factor before
the tracking stop.  Thus the computable weighted centered drift is at most
$\|d^\circ\|_{W^+}=O(E)$.  Since a weighted mean lies between the minimum
and maximum coordinates,
\[
\|d^\circ-\lambda_w\mathbf1\|_\infty
 \leq2\|d^\circ\|_\infty\leq1/8.
\]
These are precisely the inequalities enforced by the already frozen
$c_R,A_E$ in Lemma~\ref{lem:candidate_hybrid_constant_hierarchy}; in
particular $C_w\rho_w\leq h_0$, so the weighted branch is eligible.

Third, the remaining constant number of fixed-point refinements are all at
$x^+$.  Lemma~\ref{lem:large_radius_warm_start} gives total ordinary row
drift $O(E)$, and again every such branch is eligible.  At an ordinary
block boundary the algorithm charges the selected drift, not $h_0$; hence
adding the three contributions proves the stated $O(E)$ effective charge
even when the call starts a fresh block.  No score
query is made by the predictors, defect certificate, projection, greedy
correction, residual solve, path-parameter update, or cost-vector switch.
\end{proof}

Raw-relative maintenance also needs an exact specification of its initial
state and of the companion associated with the last actual query.  The next
lemma establishes this handoff from the fresh fixed-point routine.
\begin{lemma}[Raw-relative initialization and state handoff]
\label{lem:candidate_raw_initialization}
In the fresh branch of Algorithm~\ref{alg:fresh_fixed_point}, the final
Gaussian score sketch satisfies, on the same final-round event used in the
fixed-point proof,
\[
 |\widetilde\sigma_i^{(J-1)}-\sigma_i(D_{J-1})|
 \leq(c_{\rm anc}E/8)\sigma_i(D_{J-1})
 \qquad(i\in[m]).
\]
The exact branch satisfies the same statement with zero error.  The score
state handed to Algorithm~\ref{alg:candidate_hybrid_score} is
\[
 \mathcal S=
 (D_{J-1},z^{(J-1)},\widetilde\sigma^{(J-1)},0),
\]
where $z^{(J-1)}$ is the companion of the actual last query $D_{J-1}$;
the returned fixed-point vector $z^{(J)}$ is not substituted for this stored
companion.  In general the final Gaussian row count is
$O((\epsilon_{\rm fp}^{-2}+E^{-2})\log(2mJ/\zeta))$; in the outer solver
$\epsilon_{\rm fp}=E/4$, so this is $O(E^{-2}\log(2mJ/\zeta))$.
\end{lemma}
\begin{proof}
Conditional on $D_{J-1}$, every nonzero coordinate of a $k$-row Gaussian
score sketch obeys
\[
 \frac{\widetilde\sigma_i^{(J-1)}}
 {\sigma_i(D_{J-1})}\stackrel{d}{=}\frac{\chi_k^2}{k}.
\]
The explicit schedule has
$\theta_{J-1}\leq c_{\rm anc}E/16$ and
\[
 k_{J-1}=\lceil
 C_{\rm fp}\theta_{J-1}^{-2}\log(2mJ/\zeta)\rceil
 =O((\epsilon_{\rm fp}^{-2}+E^{-2})\log(2mJ/\zeta)).
\]
Here minimality of $J$ gives $e_{J-1}\geq\epsilon_{\rm fp}$; equality can
occur in the forced one-round case because the input assumption has
$D\geq\epsilon_{\rm fp}$.  Hence the unclamped last-round tolerance is
$\Omega(\epsilon_{\rm fp})$, which justifies the displayed upper bound
after taking the minimum with $c_{\rm anc}E/16$.
The chi-square event of the fixed-point proof therefore gives the displayed
simultaneous raw-relative event with failure probability at most $\zeta/J$;
it is not a new event and requires no extra sketch.
For the call made by Algorithm~\ref{alg:lp_solve_fresh},
$\epsilon_{\rm fp}=E/4$, giving the claimed $E^{-2}$ specialization.
Zero-score coordinates and the exact branch are exact.  By construction,
$D_{J-1}$ was formed from $z^{(J-1)}$; $z^{(J)}$ is computed only after
that query response.  Storing the former therefore gives precisely the
old-query companion needed to compute the next weighted drift.
\end{proof}

The accuracy proof must stop at the first violation without conditioning the
past on a future success event.  The following exact-continuation coupling
provides that filtration-safe comparison process.
\begin{lemma}[Raw-relative score coupling with exact continuation]
\label{lem:candidate_raw_score_coupling}
Let $\tau_{\rm aux}$ be the existing first coarse, tracking, algorithmic,
or companion failure, taken immediately before the next score query.  Let
$\tau_{\rm sc}$ be the first query before $\tau_{\rm aux}$ whose stored raw
post-update state violates Eq.~\eqref{eq:candidate_raw_relative_accuracy}.
Couple the implemented process to an analysis process that agrees with it
strictly before $\tau_{\rm sc}$.  At $\tau_{\rm sc}$ and at every later
score query, the analysis process supplies the exact score vector instead
of the maintained response.  Thus every response consumed by the analysis
process either satisfies the raw-relative guarantee or is exact, and a
score violation never stops a centering call after its predictor noise has
been drawn.  Moreover,
\[
 \Pr[\tau_{\rm sc}<\tau_{\rm aux}\mid
 \text{raw-relative initialization event}]\leq1/40.
\]
The predictor bias and fluctuation identities retain their original
conditional sigma-fields, and the coupled process may be stopped only at
the existing pre-noise auxiliary stops.
\end{lemma}
\begin{proof}
Strictly before $\tau_{\rm aux}$,
Lemma~\ref{lem:candidate_query_companions} supplies the companion band of
Lemma~\ref{lem:candidate_hybrid_maintenance}.  That lemma therefore bounds
the first raw-relative score violation by $1/40$.  All transition endpoints,
companions, branch choices, effective drifts, and row counts are functions
of the past and are fixed before the transition's fresh sketch.  At the
first bad response, switch the analysis process permanently to exact scores.
All later exact responses are deterministic functions of their queries,
while predictor randomness continues to be drawn with its original
conditional law.  Hence no current centering call is censored according to
a score event revealed after its weighted predictor.  The implemented and
analysis processes agree until the first violation and agree throughout on
the event that no violation occurs.  No argument conditions on that event,
so the existing predictor martingale identities are unchanged.  If an
independent auxiliary stop is used, it is taken only at its already specified
pre-noise boundary.
\end{proof}

The raw-relative guarantee must finally be translated into the additive score
accuracy expected by the fixed-point, surrogate, and predictor interfaces.
The next lemma performs this transfer, including clipping and zero scores.
\begin{lemma}[Transfer from raw-relative scores to the path interfaces]
\label{lem:candidate_score_interface}
On the raw-relative initialization event, consider the coupled process of
Lemma~\ref{lem:candidate_raw_score_coupling}.  At an executed query $D$, let
$\sigma_i=\sigma_i(D)$ and let
$\widehat\sigma_i=\min\{\max\{h_i,0\},1\}$.  At the first violating query
and at every later query replaced by the coupling, define
$\widehat\sigma=\sigma$.  If
$c_{\rm anc}E\leq1/2$ and $c_{\rm anc}\leq\min\{1/8,c_0\}$, then
\begin{equation}
\label{eq:candidate_anchor_transfer}
 |\widehat\sigma_i-\sigma_i|
 \leq c_{\rm anc}E\sigma_i
 \leq c_{\rm anc}E(\sigma_i+v_i).
\end{equation}
Thus every fixed-point score satisfies the hypothesis of
Lemma~\ref{lem:linf_contraction}, and every inflated-surrogate score
satisfies the additive hypothesis of
Lemma~\ref{lem:certified_surrogate}.  These are the only hybrid-score
consumers in Algorithm~\ref{alg:fresh_predicted_center}.  Every centering
call completed before the auxiliary stop incurs $O(E)$ total hybrid effective drift and
$O(1)$ score transitions, while its non-score work and depth are
$O(L_*+q)$ systems and $O(\log q)$ system rounds.
\end{lemma}
\begin{proof}
Every projection score lies in $[0,1]$.  Projection onto this interval is
nonexpansive and fixes $\sigma_i$, so the raw guarantee gives
\[
 |\widehat\sigma_i-\sigma_i|
 \leq|h_i-\sigma_i|
 \leq c_{\rm anc}E\sigma_i.
\]
At every replaced query the error is zero.  This proves
Eq.~\eqref{eq:candidate_anchor_transfer}.

At a fixed-point query, define $\tau_i=\sigma_i+v_i$ and
$\widehat\tau_i=\widehat\sigma_i+v_i$.  Then
\[
 (1-c_{\rm anc}E)\tau_i
 \leq\widehat\tau_i
 \leq(1+c_{\rm anc}E)\tau_i,
\]
and hence
\[
 |\log\widehat\tau_i-\log\tau_i|
 \leq\frac{c_{\rm anc}E}{1-c_{\rm anc}E}
 \leq2c_{\rm anc}E.
\]
After multiplication by the fixed-point exponent $p/2$, this is at most
$pE/8$ once $c_{\rm anc}\leq1/8$.  At the surrogate,
$\sigma_i(P^+)\leq W^+_{ii}$, and therefore
\[
 |\widehat\sigma_i-\sigma_i(P^+)|
 \leq c_{\rm anc}E\sigma_i(P^+)
 \leq c_{\rm anc}EW^+_{ii},
\]
which is the certified-surrogate hypothesis when
$c_{\rm anc}\leq c_0$.

Lemma~\ref{lem:candidate_raw_score_coupling} neither conditions on future
success nor changes any coarse conditional kernel.  Hence the bias and
second-moment statements of Lemmas~\ref{lem:large_radius_split_predictor}
and~\ref{lem:linear_defect_projection} retain their stated sigma-fields.
Lemma~\ref{lem:candidate_step_effective_drift} accounts for every score
transition in the call.  The coordinate predictor, certificate, projection,
Newton step, and fixed-point refinement have the non-score work and depth
claimed in Lemmas~\ref{lem:large_radius_split_predictor},
\ref{lem:computable_defect_projection}, and
\ref{lem:large_radius_warm_start}.  This proves the direct one-call interface.
\end{proof}

\section{Gram-pullback score transport and lazy repairs}
\label{sec:gram_transport_lazy_repairs}

We construct the score transport and auxiliary-weight maintenance used by
the path algorithm.  Pulling the new projector back to the old column space
charges a principal score transition to the retained defect.  We refine
the auxiliary weight only when its accumulated error requires it.
Section~\ref{sec:log_four} combines these routines with an amplified
predictor and the projected-residual potential to prove the final bound.

Section~\ref{subsec:gram_score_transport} derives the Gram-pullback
transition and its relative-score maintenance guarantee.
Section~\ref{subsec:gram_lazy_repairs} constructs the pointwise predictor
bound and lazy auxiliary repairs, and verifies the resulting centering
invariants.

Write $k:=c_k$, so that $C^2=256k$, and retain
$E=\Theta((\sqrt q\ell)^{-1})$, $R=c_RCE$, and $L_*=\ell$.
The constants in this section are fixed before the final radius, confidence,
and resource constants.  In particular, with $a_0:=1.1\eta_0$, strengthen
the choice of $\eta_0$ to
\begin{equation}\label{eq:gram_eta}
 a_0+\frac{a_0^2}{512B}\leq\frac1{16},
 \qquad c_{\rm anc}\leq\frac1{1024}.
\end{equation}
These are additional absolute restrictions, not input-dependent accuracies.

\subsection{Score transport through the old column space}
\label{subsec:gram_score_transport}

We first express a score transition in the old column space, so that its
error is controlled by the Gram perturbation.  The next lemma gives the
conditional relative-error bounds needed by the score-maintenance process.

\begin{lemma}[Gram-pullback transition]
\label{lem:gram_gram_transition}
Let $D_-,D_+$ and an incoming score state $h$ be fixed by a sigma-field
$\mathcal F$.  Define $\Gamma:=D_+D_-^{-1}$, $P_-:=P(D_-)$,
$P_+:=P(D_+)$, and $\mathcal K:=\Gamma^{-1}P_+\Gamma^{-1}$.
For analysis, write $P_-=VV^\top$ with $V^\top V=I$ and define
$F:=V^\top(\Gamma^2-I)V$ and $\nu:=\|F\|_F$.
Suppose a nonnegative $\mathcal F$-measurable number $\theta$ satisfies
$\nu\leq\theta\leq1/4$.  Define $H:=\mathcal K-P_-$ and
$t_i:=\operatorname{clip}(h_i,[0,1])$.  Using fresh independent
$r_1,\ldots,r_b\sim N(0,I_m)$, define
\begin{equation}\label{eq:gram_gram_update}
 h_i^+:=\gamma_i^2
 (t_i+\frac1b\sum_{j=1}^b(P_-r_j)_i(Hr_j)_i).
\end{equation}
This update satisfies the relative recurrence and MGF interface of
Eq.~\eqref{eq:candidate_unified_recurrence} and
Eq.~\eqref{eq:candidate_unified_mgf} with effective drift $8\theta$.
It uses two permitted systems per sample and $O(\operatorname{nnz}(A)+m)$
additional scalar work per sample.
\end{lemma}
\begin{proof}
The projection formula gives
\begin{equation}\label{eq:gram_pullback}
 \mathcal K=V(I+F)^{-1}V^\top,
 \qquad P_-H=HP_-=H.
\end{equation}
Define $s_i:=(P_-)_{ii}$.  Gaussian isotropy gives
\[
 \E[(P_-r)_i(Hr)_i\mid\mathcal F]=H_{ii},\qquad
 \E[h_i^+-\sigma_i(D_+)\mid\mathcal F]
 =\gamma_i^2(t_i-s_i).
\]
For a nonzero score the inherited relative-error multiplier is therefore
$b_i=s_i/\mathcal K_{ii}$.  The eigenvalue bounds for $I+F$ imply
\[
 \frac{s_i}{1+\nu}\leq\mathcal K_{ii}\leq\frac{s_i}{1-\nu},
 \qquad 1-\nu\leq b_i\leq1+\nu\leq e^{8\theta}.
\]
Clipping multiplies the incoming relative error by a number in $[0,1]$.
Moreover $\|He_i\|_2\leq\nu\sqrt{s_i}/(1-\nu)$.  Define the symmetric
matrix
\[
 A_i:=\frac{(P_-e_i)(He_i)^\top+(He_i)(P_-e_i)^\top}
 {2\mathcal K_{ii}}.
\]
Both $\|A_i\|_F$ and $\|A_i\|_{2\to2}$ are at most $5\nu/3$.
For a standard Gaussian $r$, diagonalization gives
\[
 \log\E[e^{\lambda(r^\top A_i r-\tr A_i)}\mid\mathcal F]
 \leq\frac{\lambda^2\|A_i\|_F^2}
 {1-2|\lambda|\|A_i\|_{2\to2}}
\]
when the denominator is positive.  Applying this bound to the average of
$b$ samples proves
\[
 \E[e^{\lambda\xi_i}\mid\mathcal F]
 \leq e^{\lambda^2(8\theta)^2/b},
 \qquad |\lambda|\leq b/(32\theta),
\]
where $\xi_i$ is the centered relative innovation.  This is the claimed
interface.  If $\theta=0$, then $H=0$ and the update is the deterministic
value $\gamma_i^2t_i$, not $t_i$ in general.  If a score is zero, its row is
zero and the exact zero convention applies.

One sample computes $P_-r$ and
$\Gamma^{-1}P_+(\Gamma^{-1}r)$, whose difference is $Hr$.
Each projection application uses one permitted system.  The matrices $V$,
$F$, and $H$ are not formed by the algorithm.
\end{proof}

Define \textsc{TicketScore}$(D,\widetilde w,\theta,L_{\rm conf},\mathcal S)$
as follows.  The symbol $\theta=\bot$ means that no Gram bound is supplied;
in that case run Algorithm~\ref{alg:candidate_hybrid_score}.  For a finite
ticket use effective drift $\eta=8\theta$.  If $\eta>h_0$, use the
ordinary hybrid fallback.  Otherwise use the same block test, batch size
$\max\{1,\lceil C_1\eta E^{-2}L_{\rm conf}\rceil\}$, fresh-base rule,
clipped output, and stored-state convention as that algorithm, replacing
only the transition formula by Eq.~\eqref{eq:gram_gram_update}.
A zero ticket uses the deterministic update in the preceding proof.  At a
block restart charge $\eta$, as in the original provider.  All decisions
are made before the corresponding fresh samples.

\begin{corollary}[Maintenance with Gram tickets]
\label{cor:gram_ticket_maintenance}
Before the first invalid ticket or existing auxiliary stop, the hybrid
maintenance and exact-continuation conclusions remain valid for
\textnormal{\textsc{TicketScore}}.  Its cost is
\[
 O(M+E^{-2}L_{\rm conf}(1+\Omega)).
\]
\end{corollary}
\begin{proof}
Lemma~\ref{lem:gram_gram_transition} supplies the same two conditional
inequalities used in Lemma~\ref{lem:candidate_hybrid_maintenance}.
Its block proof therefore applies without changing its stopping or restart
charges.  The exact-continuation coupling supplies exact responses after
the first raw-relative violation and does not condition a predictor on any
later score event.  Tickets will be proved valid for every weighted-chain
outcome on the preceding coarse event.
\end{proof}

\subsection{A pointwise predictor bound and lazy auxiliary repairs}
\label{subsec:gram_lazy_repairs}

A Gram-drift ticket must remain valid for every realization of the weighted
predictor on the coarse event.  The next lemma obtains this pointwise
control by clipping around a coarse estimate, while preserving the
conditional predictor bounds.

\begin{lemma}[Coarse-center clipping]
\label{lem:gram_coarse_clip}
Retain the medoid $\widetilde y$ already formed by
Algorithm~\ref{alg:coarse_defect_data}.  With its notation define
\[
 p^c:=\frac2a(W^+)^{-1/2}(x-\widetilde y),\qquad
 \mathcal B_c:=\{v:\|v-p^c\|_w\leq A_c\overline d\}.
\]
There is an absolute $A_c$ such that, on the coarse event, $u\in\mathcal B_c$.
Replace $p^w$ by its weighted metric projection onto $\mathcal B_c$ before
calling Algorithm~\ref{alg:defect_projection}.  All conclusions of
Lemma~\ref{lem:linear_defect_projection} remain valid, and additionally
\begin{equation}\label{eq:gram_pointwise_predictor}
 \|\widehat u-u\|_w\leq A_p\overline d
\end{equation}
for an absolute $A_p$ and every realization of the independent weighted
chain on that coarse event.  The modification costs $O(m)$ scalar work and
no additional systems.
\end{lemma}
\begin{proof}
Write $y=(I+aL^+)^{-1}x$ and
$J^+=(W^+)^{-1/2}T^+(W^+)^{1/2}$.
The medoid guarantee and $p^2e_0\leq d_+^2$ give
\[
 \|p^c-J^+\psi\|_w=O(\|\widetilde y-y\|_2)=O(d_+).
\]
where the first step follows from the resolvent identity and equivalence of $w$
and $W^+$, and the second step follows from the coarse medoid bound.
The certificate construction gives the explicit lower bound
\begin{equation}\label{eq:gram_defect_floor}
 \overline d^2\geq2a_3\Delta^2/(C^2k),
 \qquad d_+\leq\overline d.
\end{equation}
Take $a_3>0$ in the absolute error upper bound.
Lemma~\ref{lem:large_radius_split_predictor} bounds the deterministic
error by $N_w(J^+\psi-u)\leq a_2(E\delta+\delta^2)$.
Since $E,\delta\leq r_0/k$ and $\delta\leq\Delta$, these displays prove
$\|p^c-u\|_w\leq A_c\overline d$ after fixing $A_c$.

The center and radius are measurable in the coarse sigma-field before
drawing the independent weighted chain.  Projection onto a closed convex
set containing $u$ cannot increase distance to $u$, so the conditional
second-moment bound for $p^w-u$ is unchanged.  Unbiasedness is not needed:
the existing bias bound follows from conditional Jensen.  The subsequent
radial ball and the exact defect body also contain $u$.  The approximate
last projection adds at most $\epsilon_{\rm P}s/C$.
The first projection has distance at most $2A_c\overline d$ from $u$;
Eq.~\eqref{eq:gram_defect_floor} and $C=16\sqrt k$ give
$s=O(\overline d)$.  This proves
Eq.~\eqref{eq:gram_pointwise_predictor}.  The existing coordinate
bound $\|\widehat u-u\|_\infty\leq s/32$ is unchanged.
\end{proof}

The pointwise predictor bound lets us measure a principal query transition
by the retained defect.  The next lemma turns this estimate into a
computable Gram-drift ticket fixed before the next score sketch.

\begin{lemma}[Certified principal Gram drift]
\label{lem:gram_principal_drift}
Use the canonical query
\[
 D(x,\widehat g):=\operatorname{Diag}(\widehat g)^{-a/2}
 \operatorname{Diag}(s(x)).
\]
For the predictor of Lemma~\ref{lem:gram_coarse_clip}, the transition
from $D(x,\widehat g)$ to
$D(x^+,\widehat g\odot e^{\widehat u})$ has
$\nu\leq A_G\overline d$ for an absolute $A_G$.
The ticket $\theta=A_G\overline d$ is valid for every weighted-chain
realization on the coarse event, and $8\theta\leq h_0$ after the final
absolute radius choice.
\end{lemma}
\begin{proof}
Global inflation of $W^+$ changes the query only by a common scalar, so its
projector is that of the canonical query.  Write $P$ for this projector,
$z:=\psi-a\widehat u/2$, and $y=(I+aL^+)^{-1}(W^+)^{1/2}\psi$.
The resolvent identity gives
\[
 z=(W^+)^{-1/2}y+\frac a2(J^+\psi-\widehat u),\qquad
 \|P\operatorname{Diag}(v)P\|_F^2=v^\top(P\circ P)v.
\]
The defect factorization gives $y^\top Q^+y\leq d_+^2/p^2$.
Using $P\circ P\preceq W^+$, the deterministic surrogate error, and
Eq.~\eqref{eq:gram_pointwise_predictor}, we get
$\|P\operatorname{Diag}(z)P\|_F=O(\overline d)$.
The movement bounds give
$\|z\|_\infty=O(\delta)$ and $\|z\|_{W^+}=O(\delta/C)$, and hence
\[
 \|P\operatorname{Diag}(e^{2z}-1-2z)P\|_F
 \leq\|e^{2z}-1-2z\|_{W^+}
 \leq O(\|z\|_\infty\|z\|_{W^+})
 \leq O(\delta^2/C)=O(\overline d),
\]
where the first step follows from $P\circ P\preceq W^+$, the second
step follows from the scalar exponential remainder, the third step follows
from the movement bounds, and the last step follows from
Eq.~\eqref{eq:gram_defect_floor} and $\delta\leq r_0/k$.
Adding the linear part proves the Gram bound.  Finally the certificate
upper bound gives $\overline d=O(R/C)=O(E)$, so choosing the radius and
accuracy constants after $A_G$ makes $8A_G\overline d\leq h_0$.
\end{proof}

The state now includes an auxiliary error budget $e$ and a pending ticket
$\theta$.  At cold initialization set $e=E/4$, matching the accuracy
actually requested by Algorithm~\ref{alg:lp_solve_fresh}; after exact epoch
normalization set $e=E/8$.  In either case set $\theta=\bot$ for the
initial alignment from the last actual fixed-point query.

\begin{algorithm}[!ht]
\caption{Preparing a canonical query with lazy auxiliary repair}
\label{alg:gram_prepare}
\begin{algorithmic}[1]
\Procedure{LazyPrepare}{$x,\widehat g,e,\theta,L_{\rm conf},\mathcal S$}
\Comment{Lemma~\ref{lem:gram_lazy}}
\State $(h,\mathcal S)\gets\Call{TicketScore}
{D(x,\widehat g),\widehat g,\theta,L_{\rm conf},\mathcal S}$.
\State $\theta\gets\bot$.
\If{$e>E/4$}
  \For{$j=1,\ldots,4$}
    \State $\widehat g\gets\widehat g^{1-p/2}\odot(h+v)^{p/2}$.
    \State $(h,\mathcal S)\gets\Call{TicketScore}
    {D(x,\widehat g),\widehat g,\bot,L_{\rm conf},\mathcal S}$.
  \EndFor
  \State $e\gets E/8$.
\EndIf
\State \Return $(\widehat g,e,h,\mathcal S)$.
\EndProcedure
\end{algorithmic}
\end{algorithm}

The auxiliary weight need not be refined after every Newton step.  The
following lemma shows that threshold-triggered repairs preserve its
accuracy and charges their score drift to the accumulated predictor errors.

\begin{lemma}[Lazy repair and the score ledger]
\label{lem:gram_lazy}
After a projected predictor, replace the mandatory fixed-point call by
$\widehat g^+=\widehat g\odot e^{\widehat u}$,
$e^+=e+s/32$, and the pending ticket $A_G\overline d$.
Call Algorithm~\ref{alg:gram_prepare} before the next nonvacuous
centering call.  On the coarse and raw-relative-or-exact events, the auxiliary
error is always at most $e\leq E/4+E/3200<E/2$.  On a prefix beginning at
cold initialization or an accepted normalized checkpoint, all score queries
have total effective drift
\begin{equation}\label{eq:gram_lazy_drift}
 \Omega=O(E+\sum_j\overline d_j)
\end{equation}
and their number is $O(1+b)$ for a prefix of $b$ atomic instructions.
\end{lemma}
\begin{proof}
The coordinate predictor bound gives the error update $e+s/32$ directly.
Before a new predictor the budget is at most $E/4$, either because no repair
was needed or because a repair has reset it.  Since $s\leq E/100$, the
displayed upper bound follows by induction.

For a repair, $p\geq4/5$ gives $1-p/2\leq3/5$.
Eq.~\eqref{eq:candidate_anchor_transfer} and
Eq.~\eqref{eq:gram_eta} bound each inexact fixed-point map error by
$c_{\rm anc}E$.  Starting with error at most $E/2$, four iterations give
\[
 e_{\rm new}\leq\frac12(\frac35)^4E
 +\frac{c_{\rm anc}E}{1-3/5}<E/8.
\]
All four queries and the final alignment are at the same $x$, with
companions within $O(E)$ of its ideal weight.  The ordinary branch therefore
charges $O(E)$ for the entire repair.  The last query supplies the score at
the returned vector, so no separate alignment is charged at the next call.

Except for the first repair, each repair follows accumulated increments
$s_j/32$ exceeding $E/8$.  Thus the number of repairs is
$O(1+\sum_j s_j/E)$.  By Eq.~\eqref{eq:gram_defect_floor},
$s_j=O(\overline d_j)$.  Every pending principal transition is consumed
before any repair at its new endpoint and costs $O(\overline d_j)$ by
Lemma~\ref{lem:gram_principal_drift}.  If there is no repair, the
canonical query is already the required surrogate query.  The initial
alignment costs $O(E)$ because the stored last-query companion and returned
auxiliary vector both approximate the same fixed point to $O(E)$.
This proves Eq.~\eqref{eq:gram_lazy_drift}.
There is at most one repair, with four extra queries, per centering call.
Discarding an unconsumed ticket at a terminal output or exact normalization
does not create another score query.
\end{proof}

Define \textsc{CoarseDefectWithCenter} to be
Algorithm~\ref{alg:coarse_defect_data} with its medoid assigned to
$\widetilde y$ and its return tuple extended by that vector; all samples and
their ordering are unchanged.  The next routine replaces only
\textsc{PredictedCenter} inside the existing path stream.

\begin{algorithm}[!ht]
\caption{Predicted centering with a coarse-center clip and a lazy auxiliary weight}
\label{alg:gram_center}
\begin{algorithmic}[1]
\Procedure{LazyPredictedCenter}{$x,w,\widehat g,e,\theta,t,c_{\rm path},
\zeta,L_{\rm conf},\mathcal S$}
\Comment{Lemma~\ref{lem:gram_center_correctness}}
\State $(\widehat g,e,h,\mathcal S)\gets\Call{LazyPrepare}
{x,\widehat g,e,\theta,L_{\rm conf},\mathcal S}$.
\State Form the projected Newton point $x^+$; return
\textnormal{\textsc{Fail}} if it is infeasible or not interior.
\State Compute $\psi$, $\Delta$, $W^+$, and the implicit $P^+$ as in
Algorithm~\ref{alg:fresh_predicted_center}; form $S^+$ using the score $h$.
\State $(\overline d^2,p^\infty,p^w,\widetilde y)\gets
\Call{CoarseDefectWithCenter}{W^+,S^+,P^+,\psi,\Delta,\zeta}$.
\State $\overline d\gets\sqrt{\overline d^2}$ and form $r,s$ and the
same nonemptiness and movement tests as in
Algorithm~\ref{alg:fresh_predicted_center}; return
\textnormal{\textsc{Fail}} if a test fails.
\State $p^c\gets(2/a)(W^+)^{-1/2}((W^+)^{1/2}\psi-\widetilde y)$.
\State $p^w\gets p^c+\min\{1,A_c\overline d/\|p^w-p^c\|_w\}(p^w-p^c)$,
with multiplier $1$ when $p^w=p^c$.
\State $\widehat u\gets\Call{DefectProjection}{w,C,r,s,p^\infty,p^w}$.
\State $\widehat g^+\gets\widehat g\odot e^{\widehat u}$,
$e^+\gets e+s/32$, and $\theta^+\gets A_G\overline d$.
\State $d_{\rm obs}\gets\log w+\widehat u-\log\widehat g^+$.
\State $\chi\gets\Call{MixedBallLinearOracle}
{w,C,1.1s,\nabla\Phi(d_{\rm obs}),\epsilon_{\rm g}}$.
\State $w^+\gets w\odot e^{\widehat u+\chi}$.
\State \Return $(x^+,w^+,\widehat g^+,e^+,\theta^+,\mathcal S)$.
\EndProcedure
\end{algorithmic}
\end{algorithm}

We now combine the clipped predictor and lazy auxiliary repairs into one
centering call.  The next lemma verifies the maintained invariants and the
retained-defect decrease for this procedure.

\begin{lemma}[Centering invariants and retained decrease]
\label{lem:gram_center_correctness}
On the coarse event, Algorithm~\ref{alg:gram_center} has the same
auxiliary, movement, and conditional residual-game guarantees as
Lemma~\ref{lem:fresh_predicted_center_correctness}.  It uses
$O(\ell+q)$ non-score systems, $O(\log q)$ system depth, and
$O((\operatorname{nnz}(A)+m)\ell^2)$ non-score scalar work.
Moreover, before the tracking stop, a nonvacuous call satisfies
\begin{equation}\label{eq:gram_retained_center}
 \delta_t(x^+,w^+)\leq\delta-\frac\delta k
 -\frac B3\frac{C^2\overline d^2}{\delta}.
\end{equation}
\end{lemma}
\begin{proof}
Lemma~\ref{lem:gram_lazy} supplies the required auxiliary accuracy.
The score at the canonical query is also a score for the inflated surrogate,
because the two queries differ only by a global scalar.  The score-interface
lemma therefore supplies $S^+$, and
Lemma~\ref{lem:gram_coarse_clip} preserves the three conditional
predictor estimates in the original coarse sigma-field.  In particular the
weighted chain is not conditioned on future transport, work, or acceptance.
Writing $u-\widehat u=b_k+Z_k$, the actual observation obeys
\[
 d_{\rm obs}-(\log w-\log g(x)-Z_k)
 =\log g(x^+)-\log\widehat g^+-b_k.
\]
Its infinity norm is at most $E/2+s/8<E$, which is the original observed
softmax interface.  All feasibility and movement tests remain valid.

For the retained term, Young's inequality and
Eq.~\eqref{eq:gram_eta} give
\[
 r+1.1s\leq(1-\frac7{6k})\Delta
 -\frac B2\frac{C^2\overline d^2}{\Delta}
 \leq(1-\frac9{8k})\delta
 -\frac B3\frac{C^2\overline d^2}{\delta},
\]
where the first step follows from
$a_0\overline d\leq BC^2\overline d^2/(2\Delta)
+a_0^2\Delta/(2BC^2)$ and $C^2=256k$, and the second step follows from
$\delta\leq\Delta\leq(1+1/(128k))\delta$.
Define $z=N_w(\log w^+-\log w)$.  The Newton and weight-change bounds give
\[
 \delta_t(x^+,w^+)\leq(1+4z)(4\delta^2+z)
 \leq z+9\delta^2
 \leq\delta-\frac\delta k-\frac B3\frac{C^2\overline d^2}{\delta},
\]
where the first step follows from Lemma~\ref{lem:ls_path_facts}, the
second step follows from $0\leq z\leq\delta\leq1/16$, and the third
step follows from the preceding display and $\delta\leq1/(160k)$.
The additional clipping uses only $O(m)$ work, proving the stated costs.
\end{proof}

\section{A four-logarithm bound from the projected residual}
\label{sec:log_four}

We retain the Gram tickets and lazy auxiliary repairs of
Section~\ref{sec:gram_transport_lazy_repairs}.  We amplify the weighted predictor and
use a squared potential of the computed Newton residual.  The conservation
identity for the Lewis-weight Jacobian makes the path-parameter forcing
couple to the dissipated part of this potential.  This permits larger path
steps without increasing the total score drift.

Section~\ref{subsec:four_log_amplification_forcing} amplifies the weighted
predictor and proves the dissipative path-forcing estimate.
Section~\ref{subsec:four_log_finite_steps} controls finite Newton steps and
establishes the squared-residual potential inequality.  Finally,
Section~\ref{subsec:four_log_path_ledger} sets the path and terminal
conventions and accounts for all attempted work to prove
Theorem~\ref{thm:log_four_formal}.

Write $k:=c_k=8q$, so that $p=1-2/k$, $C=16\sqrt k$, and
$K=1/(128\sqrt k)$.  Retain $E=K/(A_EA_L\ell)$ and $R=c_RCE$.
All new constants below are absolute and are chosen before the final radius,
normalization, confidence, outer-switch, and resource constants.
For this section, defer the choice of $\eta_0$ and its dependent accuracy
constants until after the potential coefficient $\Lambda$ and the error
coefficient $\epsilon_*$ have been fixed.  The choices still satisfy
Eq.~\eqref{eq:constant_eta_condition} and Eq.~\eqref{eq:gram_eta}.

\subsection{Amplification and the forcing identity}
\label{subsec:four_log_amplification_forcing}

The squared-residual analysis needs the weighted predictor to be accurate
on the same high-probability event as the coordinate predictor and defect
certificate.  The next lemma obtains this guarantee by aggregating
independent weighted-chain copies.

\begin{lemma}[Amplified weighted predictor]
\label{lem:four_log_amplification}
In Algorithm~\ref{alg:gram_center}, replace its independent weighted
chain by $L=O(\log(1/\zeta))$ independent copies conditional on the original
coarse data.  Apply the finite medoid rule to their $W^{1/2}$-normalized
outputs before the existing coarse-center clip and defect-body projection.
The coordinate and certificate routines use failure budget $\zeta/2$ and
the medoid uses failure budget $\zeta/2$.  For sufficiently small absolute
predictor and projection accuracy constants, the enlarged coarse event has
conditional probability at least $1-\zeta$ and satisfies
\begin{equation}\label{eq:four_log_predictor_accuracy}
 \|\widehat u-u\|_\infty\leq s/32,
 \qquad C\|\widehat u-u\|_w\leq s/32,
 \qquad s=\eta_0(\overline d+\Delta/k).
\end{equation}
Here $u=\log g(x^+)-\log g(x)$ and $\Delta=\overline\delta_t(x,w)$.
All lazy-repair and Gram-ticket conclusions remain valid.  With
$\log(1/\zeta)=O(\ell)$, the non-score cost per call is $O(k\ell)$
systems, $O(\log k)$ system rounds, and
$O((\operatorname{nnz}(A)+m)\ell^2)$ scalar work.
\end{lemma}
\begin{proof}
Conditional on the original coarse data, the weighted-chain guarantee gives
\[
 \E[C^2\|p^w-u\|_w^2\mid\mathcal G]
 \leq\epsilon_1^2(256\overline d^2+\Delta^2/k^2).
\]
Choose $\epsilon_1$ sufficiently small relative to $\eta_0$.  Markov's
inequality then makes each copy lie within $s/(10^4C)$ of $u$ with
probability at least $3/4$.  Conditional independence and the finite medoid
rule give an output within $s/(10^3C)$ of $u$, except with probability
at most $\zeta/2$, after enlarging the absolute multiplier in $L$.
The clipping balls and exact defect body contain $u$ on the original
coarse event.  Their weighted metric projections cannot increase distance
to $u$, and the last approximate projection adds at most
$\epsilon_{\rm P}s/C$.  This proves the weighted claim for sufficiently
small $\epsilon_{\rm P}$.  The existing coordinate box gives the other
claim.  The clipped output still satisfies
Eq.~\eqref{eq:gram_pointwise_predictor} on the original coarse event
for every realization of the copies, so its pending Gram ticket is valid.

Include all copies in the new coarse sigma-field.  Before its first bad
event the residual game now has predictable bias $u-\widehat u$ with
$N_w(u-\widehat u)\leq s/16$ and centered noise zero.  The observation
identity in Lemma~\ref{lem:gram_center_correctness} still applies.
For the stopped analysis the decision to execute the greedy correction is
made after this enlarged coarse field; no conditional mean-square identity
is asserted after conditioning on a future success event.  The independent
copies can run in parallel.  Their systems and the finite medoid's
$O(mL^2)$ scalar work give the stated costs.
\end{proof}

For the rest of this section define the actual projected residual by
\[
 r:=P_{x,w}r_t^{c_{\rm path}}(x,w),\qquad
 X:=\|r\|_\infty,\qquad Y:=\|r\|_w,\qquad \Delta:=X+CY.
\]
Thus $r$, unlike the unprojected vector in
Definition~\ref{def:centrality_permitted}, is the vector used in the Newton
step $x^+=x-\operatorname{Diag}(s(x))r$.
Define $\zeta_i:=\phi_i'(x_i)/\sqrt{\phi_i''(x_i)}$,
$Z:=\operatorname{Diag}(\zeta)$, and $P:=P_{x,w}$.
The symbols $Z$ and $P$ in this section are deterministic operators, not
the centered predictor noise or the Lewis-query projector.
For the exact Lewis vector $g=g(x)$ and its Jacobian $B$, define
\begin{equation}\label{eq:four_log_losses}
\begin{aligned}
 z&:=Zr,& b&:=Bz,& v_0&:=Zb,& m_0&:=Pv_0,& f&:=P\zeta,\\
 d_0^2&:=p^2\|z\|_g^2-\|b\|_g^2,
 &j_1^2&:=Y^2-\|z\|_w^2,
 &j_2^2&:=\|b\|_w^2-\|v_0\|_w^2,\\
 j_3^2&:=\|(I-P)v_0\|_w^2,
 &\mathfrak D^2&:=d_0^2+j_1^2+j_2^2+j_3^2.
\end{aligned}
\end{equation}
These losses are used only in the analysis.

The path-parameter change introduces a forcing term in the projected
residual.  The next lemma bounds its interaction with the residual map
using the nonnegative losses defined above.

\begin{lemma}[Dissipative path forcing]
\label{lem:four_log_forcing}
Suppose $\|\log w-\log g(x)\|_\infty\leq K$.  All losses in
Eq.~\eqref{eq:four_log_losses} are nonnegative, and
\begin{equation}\label{eq:four_log_linear_energy}
 \|m_0\|_w^2\leq(p^2+32K^2)Y^2-\tfrac12\mathfrak D^2,
 \qquad
 |\langle m_0,f\rangle_w|\leq3\sqrt n(KY+\mathfrak D).
\end{equation}
Consequently there is an absolute $A_0$ such that, for every real $\xi$,
\begin{equation}\label{eq:four_log_forced_energy}
 \|m_0-\xi f\|_w^2
 \leq(1-3/k)Y^2-\tfrac14\mathfrak D^2+A_0n\xi^2.
\end{equation}
There is also an absolute $A_2$ such that
\begin{equation}\label{eq:four_log_linear_coordinate}
 \|m_0-\xi f\|_\infty
 \leq pX+A_2\mathfrak D+A_2\sqrt n|\xi|.
\end{equation}
\end{lemma}
\begin{proof}
The spectral and resolvent identities give $GB=B^\top G$, $B\mathbf1=0$,
and $0\preceq G^{1/2}BG^{-1/2}\preceq pI$, where
$G:=\operatorname{Diag}(g)$.  In particular $g^\top B=0$.
Since $|\zeta_i|\leq1$ and $P$ is a $w$-orthogonal projection, the four
losses are nonnegative.  If $w=g$, successive diagonal contractions and
the projection identity give
\[
 p^2Y^2-\|m_0\|_g^2=p^2j_1^2+d_0^2+j_2^2+j_3^2.
\]

For the maintained metric, define $e_b=b-pz$ and $T=W-G$.
Lemma~\ref{lem:jacobian_defect}-\hyperref[item:jacobian_defect_a]{(a)}
gives $\|e_b\|_g\leq d_0$.  Since $K\leq1/128$,
$|T|\preceq2KG$.  Thus
\[
\begin{aligned}
 p^2\|z\|_w^2-\|b\|_w^2
 &=d_0^2-2p\langle z,e_b\rangle_T-\langle e_b,e_b\rangle_T\\
 &\geq(1-2K)d_0^2-4K\|z\|_g d_0\\
 &\geq\tfrac12d_0^2-32K^2\|z\|_w^2,
\end{aligned}
\]
where the first step follows from $b=pz+e_b$, the second step follows
from Cauchy--Schwarz and $|T|\preceq2KG$, and the third step follows from
Young's inequality and $G\preceq e^KW$.  Adding the other three losses
and using $p^2\geq1/2$ proves the first inequality in
Eq.~\eqref{eq:four_log_linear_energy}.

Self-adjointness of $P$ and $v_0=Zb$ give
\begin{equation}\label{eq:four_log_forcing_identity}
 \langle m_0,f\rangle_w
 =\langle b,\mathbf1\rangle_w
 -\langle b,\mathbf1-\zeta^{\odot2}\rangle_w
 -\langle(I-P)v_0,\zeta\rangle_w.
\end{equation}
The first term equals $\langle b,\mathbf1\rangle_{w-g}$, since
$g^\top B=0$, and has magnitude at most
$(e^K-1)\sqrt{2n}\|b\|_g\leq1.5K\sqrt nY$.
The remaining terms have sum at most
$\sqrt{\sum_iw_i}(j_2+j_3)\leq2.1\sqrt n\mathfrak D$ by
Cauchy--Schwarz, $\sum_i g_i=2n$, and $(1-\zeta_i^2)^2\leq1-\zeta_i^2$.
This proves the second inequality in
Eq.~\eqref{eq:four_log_linear_energy}.

We have $\|f\|_w^2\leq2e^Kn$ and
$p^2+32K^2\leq1-7/(2k)$ for $k\geq8\log4$.
Expand $\|m_0-\xi f\|_w^2$ and use
\[
 6\sqrt n|\xi|\mathfrak D\leq\tfrac14\mathfrak D^2+36n\xi^2,
 \qquad
 6K\sqrt n|\xi|Y\leq Y^2/(2k)+18kK^2n\xi^2.
\]
Since $kK^2=1/128^2$, this proves
Eq.~\eqref{eq:four_log_forced_energy}.

For the coordinate bound, write $Bz=Dz+e$ using
Lemma~\ref{lem:jacobian_defect}-\hyperref[item:jacobian_defect_b]{(b)}.
Then $0\preceq D\preceq pI$ and $\|e\|_\infty\leq2d_0/p$.
The sensitivity bound for the projection onto $W^{-1/2}A_x$ gives
$\|(I-P)v_0\|_\infty\leq A j_3$ for an absolute $A$.
Indeed its $i$-th coordinate is at most
$\sqrt{\sigma_i(W^{-1/2}A_x)/w_i}\|(I-P)v_0\|_w$.
Also $\|P\zeta\|_\infty\leq1+A\|\zeta\|_w=O(\sqrt n)$.
These bounds prove Eq.~\eqref{eq:four_log_linear_coordinate}.
\end{proof}

\subsection{Finite Newton steps}
\label{subsec:four_log_finite_steps}

To use the linear energy estimate in the algorithm, we must control a
finite Newton step, the approximate weight update, and the changing
projection.  The following lemma bounds these errors in both the weighted
and coordinatewise components of the residual.

\begin{lemma}[Finite residual transport]
\label{lem:four_log_finite_transport}
Use the standard Lee--Sidford barriers: $-\log(x-l)$ on a lower
half-line, $-\log(u-x)$ on an upper half-line, and
$-\log\cos(a(x-(l+u)/2))$, $a=\pi/(u-l)$, on a finite interval.
Suppose $\Delta\leq R\leq r_0/k$, the tracking band holds, and the
amplified coarse event occurs.  Perform the amplified lazy centering step
and set $t^+=(1+\xi)t$, where $|\xi|\leq R$.
There are absolute $A_1,A_2,c_D>0$, independent of $\eta_0$ and
$\epsilon_*$, such that, for any prescribed sufficiently small absolute
$\epsilon_*>0$, choosing $\eta_0$ sufficiently small, then the predictor
and projection accuracies relative to $\eta_0$, and finally $r_0,c_R$
sufficiently small gives
\begin{equation}\label{eq:four_log_finite_components}
\begin{aligned}
 (Y^+)^2&\leq(1-2/k)Y^2-c_D\mathfrak D^2
             +A_1n\xi^2+\epsilon_*X^2/k^2,\\
 X^+&\leq(1-3/(2k))X+A_2\mathfrak D
             +\epsilon_*Y/\sqrt k+A_2\sqrt n|\xi|.
\end{aligned}
\end{equation}
Here $X^+,Y^+$ are computed at $(x^+,w^+,t^+)$, with the new weight in
$Y^+$.
\end{lemma}
\begin{proof}
Direct differentiation of each of the three displayed barriers gives
\begin{equation}\label{eq:four_log_barrier_identity}
 \phi_i'''=2\phi_i'\phi_i'',\qquad
 (\log s_i)'=-\phi_i',\qquad
 \zeta_i'=(1-\zeta_i^2)\sqrt{\phi_i''}.
\end{equation}
Along $x_a=x-a\operatorname{Diag}(s(x))r$, scalar Taylor estimates and
self-concordant Hessian comparison therefore give, coordinatewise,
\[
 \psi:=\log s(x^+)-\log s(x)=Zr+O(r^{\odot2}),
 \qquad |\zeta(x^+)-\zeta(x)|=O(|r|).
\]
The row-log displacement on the segment is $O(X)$, and target stability
makes its exact-weight log displacement $O(X)$ as well.  Apply the
separate $w\to w$ and $\infty\to\infty$ estimates in the proof of
Lemma~\ref{lem:jacobian_stability}, and integrate the Jacobian along this
segment.  Since $\|r^{\odot2}\|_w\leq XY$, this gives
\begin{equation}\label{eq:four_log_separate_remainders}
 \|u-BZr\|_w\leq AXY,
 \qquad \|u-BZr\|_\infty\leq AX^2.
\end{equation}
No conversion from the infinity norm to the weighted norm is used here.

The defect is a seminorm, since it is induced by
$p^2G-B^\top GB\succeq0$, and it is bounded by $p\|\cdot\|_g$.
The certificate upper bound and the last Taylor estimate give
\begin{equation}\label{eq:four_log_certificate}
 \overline d^2\leq A(d_0^2+X^2Y^2+\Delta^2/(C^2k)),
 \qquad \overline d\leq A(d_0+\Delta/k).
\end{equation}
The second inequality uses $C=16\sqrt k$ and $X\leq R\leq r_0/k$.
Define $v=\log(w^+/w)$.  The amplified predictor and the greedy response
give $N_w(v)\leq\Delta$ and
$N_w(v-u)\leq1.2\eta_0(\overline d+\Delta/k)$.
In particular $\|v\|_\infty\leq R$.

Define $P^+=P_{x^+,w^+}$.  The ordinary projection onto $W^{-1/2}A_x$
changes under a row-log displacement $\psi-v/2=O(R)$.
Lemma~\ref{lem:relative_drift} bounds its operator difference by $O(R)$
and its $i$-th column difference by
$O(R)\sqrt{\sigma_i(W^{-1/2}A_x)}$.
Conjugate back to the physical projection.  In the difference $P^+-P$
the identity terms cancel, and the two diagonal conjugation errors are
also $O(R)$ times projection columns.  Dividing the row bound by
$\sqrt{w_i}$ and applying sensitivity gives
\begin{equation}\label{eq:four_log_projection_transport}
 \|P^+-P\|_{w\to w}\leq AR,
 \qquad \|P^+-P\|_{w\to\infty}\leq AR.
\end{equation}
Thus these constants do not depend on $\min_iw_i$.

Choose the old affine multiplier $\lambda$ satisfying
$tc_{\rm path}+w\odot\phi'(x)-A\lambda=W\operatorname{Diag}(s(x))^{-1}r$.
Define
\[
 n_N:=s(x^+)/s(x)\odot r+s(x^+)\odot(\phi'(x^+)-\phi'(x)).
\]
Newton cancellation gives $|(n_N)_i|\leq Ar_i^2$.
Evaluating the new normalized residual at $(1+\xi)\lambda$ gives the
exact identity
\begin{equation}\label{eq:four_log_exact_residual}
 r^+=P^+((1+\xi)e^{-v}\odot n_N
       +(\mathbf1-(1+\xi)e^{-v})\odot\zeta(x^+)).
\end{equation}
Indeed the old affine residual equals
$w\odot(s(x)^{-1}\odot r-\phi'(x))$ after subtracting $A\lambda$;
substitution of $w^+=w\odot e^v$ proves the formula.

Expand $\mathbf1-(1+\xi)e^{-v}=v-\xi\mathbf1+O(v^{\odot2}+|\xi v|)$.
Use Eq.~\eqref{eq:four_log_separate_remainders} to replace $u$ by $b$,
Eq.~\eqref{eq:four_log_projection_transport} to replace $P^+$ by $P$,
and $|\zeta(x^+)-\zeta(x)|=O(|r|)$ to replace the diagonal factor.
For an absolute $A$, uniform over $0<\eta_0\leq1$, this gives
\begin{equation}\label{eq:four_log_finite_error}
\begin{aligned}
 r^+&=m_0-\xi f+e,\\
 \|e\|_w&\leq ARY+A\eta_0(d_0+\Delta/k)/C+AR|\xi|\sqrt n,\\
 \|e\|_\infty&\leq AR(X+Y)+A\eta_0(d_0+\Delta/k)+AR|\xi|\sqrt n.
\end{aligned}
\end{equation}
For example, the Newton remainder has weighted norm $O(XY)$, the
exponential remainder has weighted norm
$O((R+|\xi|)\|v\|_w)$, and
$\|v\|_w\leq A Y+A\eta_0\Delta/(Ck)$; these give the displayed weighted
bound.  The coordinate bound follows from the same coordinate estimates
and $\|P^+h\|_\infty\leq\|h\|_\infty+A\|h\|_w$.
Here $\|v\|_\infty\leq A(X+Y)$: the diagonal-plus-defect estimate
gives $\|b\|_\infty\leq pX+2d_0/p$, while $d_0\leq AY$ and
$\Delta/k=X/k+16Y/\sqrt k$.  Thus the coordinate exponential
remainder does not introduce a factor $C$ in front of $RY$.

Define $F:=\sqrt n|\xi|$.  Since $d_0\leq\mathfrak D$ and
$\Delta=X+16\sqrt kY$, Eq.~\eqref{eq:four_log_finite_error} gives
\begin{equation}\label{eq:four_log_error_parameters}
\begin{aligned}
 \|e\|_w&\leq A((r_0+\eta_0)Y/k
       +\eta_0\mathfrak D/\sqrt k+\eta_0X/k^{3/2}+r_0F/k),\\
 \|e\|_\infty&\leq A((r_0+\eta_0)X/k
       +(r_0+\eta_0)Y/\sqrt k+\eta_0\mathfrak D+r_0F/k).
\end{aligned}
\end{equation}
The coefficient of the greedy correction is $A\eta_0$; increasing the
predictor accuracy alone does not make it smaller.

We can show, for sufficiently small fixed upper bounds on $r_0,\eta_0$,
\begin{equation}\label{eq:four_log_uniform_components}
\begin{aligned}
 (Y^+)^2&\leq(1-5/(2k)+A(r_0+\eta_0)/k)Y^2
       -\mathfrak D^2/8+AF^2+A\eta_0^2X^2/k^2,\\
 X^+&\leq(1-2/k+A(r_0+\eta_0)/k)X+A_2\mathfrak D
       +A(r_0+\eta_0)Y/\sqrt k+A_2F.
\end{aligned}
\end{equation}
Here $A,A_2$ are fixed before $\eta_0$ and $\epsilon_*$ are chosen.
For the first inequality, expand $\|m_0-\xi f+e\|_w^2$ and use
Eq.~\eqref{eq:four_log_forced_energy}, $\|m_0\|_w\leq AY$, and
$\|f\|_w\leq A\sqrt n$.  The additional terms are at most
$2AY\|e\|_w+2\|e\|_w^2+AF^2$.  We can show
\[
 A\eta_0Y\mathfrak D/\sqrt k
 \leq\mathfrak D^2/32+8A^2\eta_0^2Y^2/k,
 \qquad
 A\eta_0YX/k^{3/2}
 \leq Y^2/(8k)+2A^2\eta_0^2X^2/k^2,
\]
where both inequalities follow from Young's inequality.
The other products contribute at most
$A(r_0+\eta_0)Y^2/k+AF^2$, and the squared error contributes at most
$A\eta_0^2\mathfrak D^2/k+A\eta_0^2X^2/k^3$
in addition to terms of the same form.  Fix the upper bound on $\eta_0$
so that this extra loss term is at most $\mathfrak D^2/32$.
Finally $W^+\preceq e^RW$ costs at most another
$Ar_0Y^2/k$ and changes the absolute coefficients by bounded factors.
This leaves the displayed $1/8$ loss coefficient and $5/(2k)$ margin,
after enlarging $A$.  The second inequality follows from
Eq.~\eqref{eq:four_log_linear_coordinate} and
Eq.~\eqref{eq:four_log_error_parameters}; enlarging a fixed $A_2$ covers
all $0<\eta_0\leq1$ and $0<r_0\leq1$.

Choose $\eta_0$ and then $r_0$, within these fixed upper bounds, so that
\begin{equation}\label{eq:four_log_smallness}
 A(r_0+\eta_0)\leq\min\{1/2,\epsilon_*\},
 \qquad A\eta_0^2\leq\epsilon_*.
\end{equation}
Choose the predictor and projection accuracies relative to $\eta_0$
before fixing $r_0$, as required by
Lemma~\ref{lem:constant_hierarchy}.  Eq.~\eqref{eq:four_log_uniform_components}
now proves Eq.~\eqref{eq:four_log_finite_components}, with $c_D=1/8$ and
a fixed $A_1$.  Thus neither $A_1,A_2$ nor $c_D$ depends on the final
choice of $\epsilon_*$.
\end{proof}

The preceding component bounds must be combined without losing the negative
defect term.  The next lemma gives a one-step bound for a squared residual
potential, with both contraction and explicit defect dissipation.

\begin{lemma}[Squared residual potential]
\label{lem:four_log_potential}
There are absolute $\Lambda\geq256$, $c_*,A_*>0$ such that
\begin{equation}\label{eq:four_log_potential}
 \mathcal V:=X^2+\Lambda kY^2,
 \qquad
 \mathcal V^+\leq(1-1/(2k))\mathcal V
                  -c_*k\mathfrak D^2+A_*kn\xi^2.
\end{equation}
Also $\Delta^2\leq2\mathcal V$ and
$\mathcal V\leq L_\Lambda\Delta^2$, where
$L_\Lambda:=\max\{1,\Lambda/256\}$.
\end{lemma}
\begin{proof}
Square the coordinate inequality in
Eq.~\eqref{eq:four_log_finite_components}, using
$(a+b)^2\leq(1+1/(2k))a^2+(1+2k)b^2$ and then applying
$(a+b+c)^2\leq3(a^2+b^2+c^2)$ to its last three terms.
Since $k\geq8\log4$, this gives
\[
 (X^+)^2\leq(1-2/k)X^2+9kA_2^2\mathfrak D^2
             +9\epsilon_*^2Y^2+9kA_2^2n\xi^2.
\]
Choose $\Lambda\geq\max\{256,32A_2^2/c_D\}$ and then choose
$\epsilon_*$ so that $\Lambda\epsilon_*\leq1/4$ and
$9\epsilon_*^2\leq\Lambda/4$.
Adding $\Lambda k$ times the weighted inequality proves
Eq.~\eqref{eq:four_log_potential}, with positive slack in both contraction
terms and in the loss coefficient.  The norm comparisons follow from
$C^2=256k$ and $(X+CY)^2\leq2X^2+2C^2Y^2$.
\end{proof}

\subsection{Path steps, terminal accuracy, and checkpoint caps}
\label{subsec:four_log_path_ledger}

Define $\rho:=R/4$ and choose an absolute $\theta>0$ such that
$A_*\theta^2\leq1/4$.  Replace the path step by
\begin{equation}\label{eq:four_log_path_step}
 \alpha:=\frac{\theta\rho}{k\sqrt n},
 \qquad t^+:=\operatorname{median}\{(1-\alpha)t,t_{\rm goal},(1+\alpha)t\}.
\end{equation}
Use the amplified version of Algorithm~\ref{alg:gram_center}.
After each complete centering and path update, compute the new projected
residual and reject the attempt unless
\begin{equation}\label{eq:four_log_path_test}
 \mathcal V^+\leq\rho^2-\rho^2/(8k).
\end{equation}
The test also applies to a zero-residual path update, which uses no predictor.
All original feasibility, positivity, score, and local-resource tests remain.

At normalized checkpoints strengthen the requested accuracy to
$\epsilon_{\rm rec}\leq\rho/(8\sqrt{L_\Lambda})$ and certify
$\mathcal V\leq\rho^2/4$ in addition to the existing outward tests.
This is an absolute-factor strengthening of the polynomial normalization
accuracy.  Choose $c_{\rm sw}$ sufficiently small and then $A_{\rm sw}$
sufficiently large that the cost-switch estimate in
Lemma~\ref{lem:outer_reduction} gives
$\overline\delta^{c}_{t_{\rm sw}}\leq\rho/(4\sqrt{L_\Lambda})$.
Neither change introduces dependence on $\kappa$ into normalization.

The two path phases require different terminal-accuracy guarantees, and
checkpoint normalization can interrupt the first terminal loop.  The next
lemma specifies the restart and final-tail conventions that preserve both
guarantees.

\begin{lemma}[Terminal convention]
\label{lem:four_log_terminal}
For the first path phase use a fixed-$t$ amplified terminal loop of length
$J_{\rm sw}=\lceil8k\log(4R/\tau_{\rm sw})\rceil$.
If a checkpoint normalization occurs inside this loop, restart its counter
at zero.  Choose $B_{\rm ep}>4J_{\rm sw}$.
For the second phase, end the checkpointed stochastic stream when
$t=t_{\rm out}$; outside that stream run fixed-weight projected Newton
steps until the computable proxy is at most $\tau_{\rm out}$.
The latter tail uses no score queries or weight updates and is run only once.
On the good event these conventions give the two terminal guarantees of
Lemma~\ref{lem:outer_reduction}.  The first convention adds at most
$J_{\rm sw}$ successful-stream instructions.  The final tail costs
$O(1+\log(2+\log(R/\tau_{\rm out})))=O(\kappa+\ell)$ systems.
\end{lemma}
\begin{proof}
With $\xi=0$, Eq.~\eqref{eq:four_log_potential} contracts $\mathcal V$
by $1-1/(2k)$.  The comparison $\Delta^2\leq2\mathcal V$ proves the
first terminal guarantee after $J_{\rm sw}$ consecutive iterations.
Normalization can raise a small residual, so a pre-normalization counter
does not certify these consecutive iterations.  Resetting it restores
that implication.  Since $J_{\rm sw}<B_{\rm ep}$, after such an accepted
checkpoint the complete first terminal loop fits before the next
checkpoint.  Thus at most one successful reset is needed.  Its length is
$O(k\log\ell)$, since $\tau_{\rm sw}$ is independent of $U,\epsilon$.

At the end of the second path loop, $\Delta\leq\sqrt2\rho$ and tracking
holds within $K/4$.  Keep $w,t$ fixed.  The Newton estimate and proxy
comparison give $\Delta_{j+1}\leq8\Delta_j^2$ while the tracking band
holds.  Starting with sufficiently small $R$, the sum of the normalized
steps is $O(R)$; target stability then changes $\log g$ by at most
$O(R)\leq K/4$, after the final radius choice.  This closes the tracking
and interiority induction.  Iterating the quadratic inequality gives the
displayed number of steps.  Announce this deterministic tail cap in advance,
test feasibility and the proxy at each step, and return
\textnormal{\textsc{Fail}} if the cap is reached without success.
On the good event it is not reached.  At termination
Lemma~\ref{lem:fixed_weight_terminal_distance} and the final paragraph of
Lemma~\ref{lem:outer_reduction} give the output guarantee.  Since the tail
is outside the retry loop, its cost is not multiplied by the number of
attempts.
\end{proof}

The squared-potential decrease controls both the number of path instructions
and the score drift accumulated within each attempted epoch.  The next
lemma converts these bounds into deterministic resource caps, including
the costs of rejected attempts.

\begin{lemma}[Four-logarithm prefix and retry ledger]
\label{lem:four_log_ledger}
The preceding path and terminal conventions, executed with the existing
instruction-quota checkpoint wrapper and the strengthened normalization,
have a successful-stream cap
\begin{equation}\label{eq:four_log_call_cap}
 N:=\lceil A_N'\sqrt n k\ell\kappa\rceil.
\end{equation}
For every good prefix of at most $b$ atomic instructions in an attempt,
\begin{equation}\label{eq:four_log_prefix}
 \sum_j\overline d_j^2\leq A\frac{\rho^2}{k}(1+b/k),
 \qquad
 \Omega\leq A E(1+\sqrt b+b/\sqrt k).
\end{equation}
Use the deterministic local system cap
\begin{equation}\label{eq:four_log_attempt_cap}
 S_{\rm att}^{(4)}(b):=A_{
 \rm loc}'(b(k\ell+1)+E^{-2}\ell
 +E^{-1}\ell(1+\sqrt b+b/\sqrt k)
 +\mathbf1_{\{b=B_{\rm ep}\}}P_{\rm norm}).
\end{equation}
Keep the query and round caps from Eq.~\eqref{eq:attempt_resource_caps}
with larger absolute multipliers and take the scalar cap to be
$A_{\rm loc}'(\operatorname{nnz}(A)+m)(S_{\rm att}^{(4)}(b)+b\ell^2)$.
In addition to these four caps, charge scalar parallel depth to the
attempt-local cap
\begin{equation}\label{eq:four_log_scalar_depth_cap}
 H_{\rm att}(b):=A_{\rm loc}'
 (b\ell^2+1+\mathbf1_{\{b=B_{\rm ep}\}}P_{\rm norm}).
\end{equation}
Each scalar block announces its depth before execution; exceeding this cap
rejects the attempt.  The corresponding shared counter is never rolled back.  Define
\begin{equation}\label{eq:four_log_global_caps}
\begin{aligned}
 N_{\rm cap}&:=N, & B_{\rm ep}&:=(4m)^{A_{\rm ep}},\\
 M_{\rm ep}&:=\lceil A_{\rm call}(B_{\rm ep}+q\ell+1)\rceil,
 & L_{\rm conf}&:=\log(40mM_{\rm ep}),\\
 M_{\rm calls}&:=\lceil A_B(N+q\kappa+\log(q/E)+1)\rceil,\\
 B_{\rm sys}&:=\lceil A_B(\sqrt n q\ell^3\kappa
                    +\min\{q\ell^3,n\log\ell\})\rceil,\\
 B_{\rm rnd}&:=\lceil A_B(\sqrt n q\ell\log q\kappa
                    +\kappa+\log\ell+1)\rceil,\\
 B_{\rm ar}&:=\lceil A_B(\operatorname{nnz}(A)+m)B_{\rm sys}\rceil,
 & H_{\rm cap}&:=\lceil A_{\rm dep}N\ell^2\rceil.
\end{aligned}
\end{equation}
Here $A_B,A_{\rm dep}$ are sufficiently large absolute constants chosen
after the geometric, confidence, normalization, and local-cap constants.
These shared caps include the one-time initialization and terminal tail.
Save and restore the auxiliary error budget $e$ and pending ticket $\theta$
with every mathematical checkpoint, and reset $(e,\theta)$ to $(E/8,\bot)$
after normalization.  Resource counters are never rolled back.
Including rejected attempts and the final deterministic tail, the system
count is
\begin{equation}\label{eq:four_log_system_count}
 O(\sqrt n q\ell^3\kappa+\min\{q\ell^3,n\log\ell\}).
\end{equation}
The corresponding scalar work is at most
$O(\operatorname{nnz}(A)+m)$ times this bound.
\end{lemma}
\begin{proof}
If $\mathcal V\leq\rho^2$, Eq.~\eqref{eq:four_log_potential} and
$|\xi|\leq\alpha$ give
\[
 \mathcal V^+\leq(1-1/(2k))\rho^2+A_*kn\alpha^2
 \leq(1-1/(4k))\rho^2,
\]
where the first step discards the nonnegative dissipated loss and the
second step follows from Eq.~\eqref{eq:four_log_path_step} and
$A_*\theta^2\leq1/4$.  Thus a good step passes
Eq.~\eqref{eq:four_log_path_test} with slack.  If $r=0$, its new residual
is exactly $-\xi P\zeta$, and the same bound follows after enlarging
$A_*$.  The initialization has $r=0$.  The strengthened switch and
normalization restore the potential invariant, and $\Delta\leq\sqrt2\rho<R$
licenses every centering call.  A full signed path update changes
$|\log t|$ by $\Omega(\alpha)$.  The original outer parameters have total
logarithmic path length $O(\kappa)$, and
Lemma~\ref{lem:four_log_terminal} adds only $O(k\log\ell)$ instructions.
This proves Eq.~\eqref{eq:four_log_call_cap} for a sufficiently large
$A_N'$.

Sum Eq.~\eqref{eq:four_log_potential} over a prefix and keep both negative
terms.  The cost switch can occur only once and adds at most $\rho^2$
to the potential; terminal iterations have $\xi=0$.  Consequently
\[
 \frac1{2k}\sum_j\mathcal V_j+c_*k\sum_j\mathfrak D_j^2
 \leq2\rho^2+A_*kn\sum_j\xi_j^2
 \leq2\rho^2+A_*\theta^2 b\rho^2/k.
\]
where the first step follows from telescoping and the nonnegativity of the final
potential, and the second step follows from $|\xi_j|\leq\alpha$.
Eq.~\eqref{eq:four_log_certificate}, $R\leq r_0/k$, and
$\Delta^2\leq2\mathcal V$ give
$\overline d_j^2\leq A(\mathfrak D_j^2+\mathcal V_j/k^2)$.
This proves the squared-defect bound in Eq.~\eqref{eq:four_log_prefix}.
Next,
\[
 \sum_j\overline d_j
 \leq\sqrt{b\sum_j\overline d_j^2}
 \leq A\frac\rho{\sqrt k}(\sqrt b+b/\sqrt k),
\]
where the first step follows from Cauchy--Schwarz and the second step
follows from the squared-defect bound and
$\sqrt{b(1+b/k)}\leq\sqrt b+b/\sqrt k$.
Since $\rho/\sqrt k=\Theta(E)$,
Lemma~\ref{lem:gram_lazy} proves the drift bound.
The preparation queries of a partially completed last instruction consume
only the preceding ticket and at most one repair.  Their cost is already
charged to the preceding defect and accumulated error increments.
Alternatively, any one extra good principal ticket satisfies
$\overline d=O(R/C)=O(E)$ by the certificate upper bound, and any one
extra repair costs $O(E)$.  These are absorbed by the initial $E$ term.
No bound on a bad prefix is asserted.

On an attempted epoch, use the exact-continuation score coupling and stop
before a bad enlarged coarse event.  Lemma~\ref{lem:four_log_amplification}
supplies the bias-only residual-game hypotheses on this stopped process.
The coarse union failure, score failure, and tracking failure are at most
$1/100$, $1/40$, and $1/20$, respectively, with the same polynomial
confidence horizon.  Their intersection has probability at least $4/5$
conditionally on every certified starting state.  On this event every
prefix fits Eq.~\eqref{eq:four_log_attempt_cap}; all endpoint tests pass,
and the strengthened exact normalizer returns a valid next state.
At the first-terminal checkpoint reset the terminal counter as specified
in Lemma~\ref{lem:four_log_terminal}.  No other descriptor is reset.
The acceptance and retry proof of Lemma~\ref{lem:checkpointed_path}
therefore applies with the new successful-stream cap $N$.
The final uncertified segment ends before the deterministic tail, so its
good event licenses that tail and its output guarantee.

On other outcomes the preannounced local caps reject before larger work;
no defect estimate is asserted for a bad prefix.  With
$A\leq4(\lceil N/B_{\rm ep}\rceil+1)$ attempts and announced quotas
$b_1,\ldots,b_A$, the same deterministic bounds give
$\sum_a b_a=O(N)$ and $A=O(N/B_{\rm ep}+1)$.
Choose $N,B_{\rm ep}\geq k$.  Then Cauchy--Schwarz gives
$\sum_a\sqrt{b_a}=O(N/\sqrt k)$.
Summing every cap, including rejected attempts, gives
\[
 O(Nk\ell+NE^{-1}\ell/\sqrt k
 +(N/B_{\rm ep}+1)E^{-2}\ell
 +(N/B_{\rm ep})P_{\rm norm}).
\]
The original polynomial epoch choice absorbs the normalization and repeated
base costs.  Substitute $N=O(\sqrt n k\ell\kappa)$,
$E^{-1}=\Theta(\sqrt k\ell)$, and $k=8q\leq8\ell$ to obtain
$O(\sqrt n q\ell^3\kappa+q\ell^3)$.
The isolated base and deterministic tail are absorbed by the leading term.
Adding the unchanged cold initialization proves
Eq.~\eqref{eq:four_log_system_count}.
The scalar term $O((\operatorname{nnz}(A)+m)N\ell^2)$ has the same
leading order.  No confidence parameter depends on $\kappa$.
The query-cap sum is $O(N)$, including checkpoint queries, since
$(N/B_{\rm ep})P_{\rm norm}=O(N)$.  Adding initialization gives
$O(N+q\kappa+\log(q/E)+1)$ queries.  The round-cap sum, initialization,
and once-only tail give
$O(\sqrt n q\ell\log q\kappa+\kappa+\log\ell+1)$ rounds.
Thus sufficiently large fixed multipliers in
Eq.~\eqref{eq:four_log_global_caps} cover every good trajectory.
On every other trajectory the counters reject the next operation before
any shared cap is exceeded.

We also verify the scalar-depth cap.  On a good instruction, the two
scalar optimizers have $J=O(\ell)$ iterations, each of depth $O(\ell)$,
by Lemmas~\ref{lem:mixed_ball_linear_oracle}
and~\ref{lem:computable_defect_projection}.  Sketch rows and amplified
copies run in parallel.  Every batch inside an attempt has polynomial
size in $m$, so its reductions have depth $O(\ell)$; the $O(\log q)$
successive chain rounds and the constant number of auxiliary repairs
therefore fit in $O(\ell^2)$ scalar depth per instruction.  Normalization
has at most $P_{\rm norm}$ scalar operations and hence at most that depth.
Thus a good attempt never reaches Eq.~\eqref{eq:four_log_scalar_depth_cap}
after increasing $A_{\rm loc}'$.  On any other outcome the cap rejects
before a larger depth is incurred, including when a scalar optimizer
requests more iterations outside the tracking band.
Summing the local caps uses $\sum_a b_a=O(N)$ and
$(N/B_{\rm ep})P_{\rm norm}=O(N)$, giving $O(N\ell^2)$ scalar depth
for the attempted stream.  Cold initialization has
$O(\ell\log\ell)$ scalar depth, and the deterministic tail has
$O((\kappa+\ell)\ell)$ scalar depth.  Since
$\kappa\geq\ell$ and $N\geq\sqrt n q\ell\kappa$ after fixing its
absolute multiplier, both are absorbed by the shared scalar-depth cap.
\end{proof}

\begin{theorem}[Linear programming]
\label{thm:log_four_formal}
Let the linear program in Eq.~\eqref{eq:intro_lp}, its interior point $x_0$,
and a finite certified magnitude bound $U$ supplied to the algorithm satisfy
the assumptions in Section~\ref{sec:intro}, and let $\epsilon\in(0,1)$.
Define $\ell:=\log(4m)$, $q:=\log(4m/n)$,
$\kappa:=A_\kappa(1+\log(emU/\epsilon))$, and $L_*:=\ell$.
Take the geometric parameters from Lemma~\ref{lem:sharper_parameters},
$L_{\rm ch}:=A_L\ell$, $E:=K/(A_EL_{\rm ch})$, and $R:=c_RCE$.
Run Algorithm~\ref{alg:lp_solve_fresh} with the amplified lazy predictor of
Lemma~\ref{lem:four_log_amplification}, the path step and tests in
Eq.~\eqref{eq:four_log_path_step} and Eq.~\eqref{eq:four_log_path_test},
and the terminal and checkpoint conventions and resource caps of
Lemmas~\ref{lem:four_log_terminal} and~\ref{lem:four_log_ledger}.
The algorithm halts on every outcome within
\[
 O(\sqrt n q\ell^3\kappa+\min\{q\ell^3,n\log\ell\})
\]
permitted systems and, with constant probability, returns a feasible
$\epsilon$-optimal solution.  Its work is
\[
 O(\sqrt n\log^4(4m)\log(emU/\epsilon)\mathcal T_w).
\]
Its linear-system depth is
\[
 O((\sqrt n q\ell\log q\kappa+\kappa+\log\ell)\mathcal T_d).
\]
Including scalar operations, its parallel depth is
\[
 O(\sqrt n\log^{3.5}(4m)\log(emU/\epsilon)\mathcal T_d).
\]
\end{theorem}
\begin{proof}
Fix the inherited geometric and certificate constants, including $B$,
and the uniform constants $A_1,A_2,c_D$ in
Lemma~\ref{lem:four_log_finite_transport} first.  Choose $\Lambda$, then
$\epsilon_*$, then $\eta_0$ small enough for
Eq.~\eqref{eq:four_log_smallness}, Eq.~\eqref{eq:constant_eta_condition},
and Eq.~\eqref{eq:gram_eta}.  Choose the predictor and projection
accuracies relative to this $\eta_0$, including the amplification
requirements.  Next choose $r_0,c_R$ sufficiently small
and the accuracy multipliers sufficiently large.  Complete the existing
joint normalization, epoch, confidence, and switch choice of
Lemma~\ref{lem:candidate_hybrid_constant_hierarchy}, including the
additional absolute normalization and switch constraints above before
its choice of $H$.  Then choose the resource caps.  No radius or switch
constant is changed after this joint choice.

For the floor $v=(n/m)\mathbf1$, the ideal initial weight satisfies
$n/m\leq g_i(x_0)\leq2$.  Hence the cold start $\mathbf1$ has logarithmic
error at most $q$.  Lemma~\ref{lem:floor_homotopy_initialization}, with
$D=q$ and target accuracy $E/4$, gives
\[
 \|\log\widehat g-\log g(x_0)\|_\infty\leq E/4.
\]
We can show
\[
 \Phi(\log\widehat g-\log g(x_0))
 \leq2m e^{1/400}\leq3m\leq M_{\rm ch},
\]
where the first step follows from the preceding accuracy bound and the
definition of $\Phi$, the second step follows from $2e^{1/400}<3$, and
the third step follows from the definition of $M_{\rm ch}$.
Lemma~\ref{lem:candidate_raw_initialization} makes the final initializer
response raw-relative and stores the companion of its actual query.
This event is measurable before either path phase.
Setting $w=\widehat g$ and $d=-w\odot\phi'(x_0)$ gives zero projected
residual for the artificial cost.  Initialize the lazy state by
$(e,\theta)=(E/4,\bot)$, carry it through the two path phases, and reset
it to $(E/8,\bot)$ after normalization.
Lemmas~\ref{lem:four_log_potential} and~\ref{lem:four_log_ledger} give
constant success for the checkpointed stream.  The first terminal loop
and cost switch supply the second path's hypotheses.  The final
fixed-weight tail is deterministic and, on this good event, the outer
reduction gives a feasible $\epsilon$-optimal output.

Use the shared caps in Eq.~\eqref{eq:four_log_global_caps} and the local
caps in Lemma~\ref{lem:four_log_ledger}.  Charge scalar depth to
Eq.~\eqref{eq:four_log_scalar_depth_cap} and its shared counter, and
charge the one-time deterministic tail separately before each permitted solve.
All counters are monotone through rejected attempts, and every operation
announces its charge before work or randomness.  Thus every outcome
terminates within the displayed bounds, while no cap changes a good
trajectory.  The copies and sketch rows are parallel, leaving
$O(\log q)$ system rounds per path instruction.  Adding initialization
and the once-only tail gives the linear-system depth bound.

By Lemma~\ref{lem:four_log_ledger}, total parallel depth is
\begin{equation}\label{eq:four_log_total_depth}
 O((\sqrt n q\ell\log q\kappa+\kappa+\log\ell)\mathcal T_d
      +\sqrt n q\ell^3\kappa).
\end{equation}
To absorb the scalar term, we use the bounded-fan-in model stated in
Section~\ref{sec:intro}.  Write the $i$-th row of $A$ as $a_i^\top$ and
define $M:=A^\top\mathbf DA$.  For fixed positive $\mathbf D$, choose
a right-hand side $\mathbf q$ outside the finitely many hyperplanes
$a_i^\top M^{-1}\mathbf q=0$.  This is possible because every row is
nonzero.  Then
\[
 \frac{\partial(M^{-1}\mathbf q)}{\partial\mathbf D_{ii}}
 =-M^{-1}a_i a_i^\top M^{-1}\mathbf q\neq0.
\]
Thus every one of the $m$ diagonal inputs affects an output coordinate,
and some one of the $n$ outputs depends on at least $m/n$ inputs.
A depth-$d$ bounded-fan-in computation depends on at most $a^d$ inputs
for an absolute $a>1$.  Since preprocessing is independent of the
diagonal inputs and $\mathcal T_d\geq1$, this proves
$\mathcal T_d=\Omega(1+\log(m/n))=\Omega(q)$.

We can show
\[
 \sqrt n q\ell^3\kappa
 \leq O(\sqrt n\ell^3\kappa\mathcal T_d)
 \leq O(\sqrt n\ell^{3.5}\kappa\mathcal T_d),
\]
where the first step follows from $q=O(\mathcal T_d)$, and the second
step follows from $\ell\geq1$.  The system-depth term in
Eq.~\eqref{eq:four_log_total_depth} obeys the same bound because
$q\leq\ell$, $\log q\leq\ell$, and $\kappa\geq\ell$.
This proves the claimed total parallel depth.

Finally $q\leq\ell$ gives
$q\ell^3\leq\ell^4$, and
the work convention $\mathcal T_w\geq\operatorname{nnz}(A)+m$ in
Section~\ref{sec:intro} absorbs the scalar work.
\end{proof}

\section*{Acknowledgments and AI Disclosure}
The author has been working on this project since January 2025. The first version of this draft was completed on September 2, 2026 (\url{10.5281/zenodo.22242576}). The AI tools used in preparing the first version were mainly ChatGPT Pro 5.6 and Codex Sol 5.6, and the author used them only to improve the grammar and presentation of the paper. The current version obtains a $\log^4$ factor with the help of ChatGPT Pro 6 and Codex 6. All of the proofs have been completely rewritten by the author. The author takes full responsibility for the correctness of the proofs in this paper.

\bibliographystyle{alpha}
\bibliography{ref}

\end{document}